\documentclass[12pt]{amsart}
\usepackage[utf8]{inputenc}
\newbox\mybox 
\newdimen\myboxwidth
 \usepackage{stmaryrd, mathrsfs}
\usepackage{amsmath,amsfonts,amsthm,amssymb, graphicx, dsfont,wrapfig}
\usepackage[usenames,dvipsnames]{xcolor}
\usepackage[margin= 1in]{geometry}
\usepackage{comment,enumerate}

\newcommand{\Z}{\mathds{Z}}
\newcommand{\Zd}{\mathds{Z}^d}

\newcommand{\E}{\mathds{E}}
\newcommand{\prob}{\mathds{P}}

\newcommand{\cC}{\mathcal{C}}
\newcommand{\cD}{\mathcal{D}}
\newcommand{\cT}{\mathcal{T}}
\newcommand{\cE}{\mathcal{E}}
\newcommand{\fC}{\mathfrak{C}}

\newcommand{\md}{\mathrm{d}}

\newcommand{\clan}{\mathscr{X}(\cC)}

\newcommand{\sD}{\mathscr{D}}
\newcommand{\dZ}{\mathds{Z}}

\newcommand{\Go}{\boldsymbol{\omega}}
\newcommand{\sA}{\mathscr{A}}
\newcommand{\ve}{\mathbf{e}}
\newcommand{\longestH}{\mathfrak{L}_H}
\newcommand{\longest}{\mathfrak{L}}
\let\go=\omega

\DeclareMathOperator{\pr}{\mathds{P}}

\def\beq{ \begin{equation} }
 \def\eeq{ \end{equation} }
 \def\beqx{ \begin{equation*} }
 \def\eeqx{ \end{equation*} }
 \def\beqa{\begin{eqnarray}}
 \def\eeqa{\end{eqnarray}}
 \def\beqax{\begin{eqnarray*}}
 \def\eeqax{\end{eqnarray*}}

\newcommand{\sa}[1]{\ensuremath{\,{\buildrel #1 \over \longleftrightarrow}\,}}
\newcommand{\sS}{\mathscr{S}}
\newcommand{\dN}{\mathds{N}}

\let\del=\partial
\let\gd=\delta

\let\gl=\lambda
\let\lra=\leftrightarrow

\let\c=a
\let\C=A

\title[Subcritical High Dimensional Percolation]{Subcritical Connectivity and Some Exact Tail Exponents in High Dimensional Percolation}

\date{\today}

\newtheorem{thm}{Theorem}
\newtheorem{lem}{Lemma}
\newtheorem{prop}[lem]{Proposition}
\newtheorem{defin}{Definition}
\newtheorem{cor}[lem]{Corollary}
\newtheorem{clam}[lem]{Claim}
\newtheorem{remk}{Remark}

\begin{document}
\author[S. Chatterjee]{Shirshendu Chatterjee }
		\address{S.~Chatterjee\\
			Department of Mathematics\\
			City University of New York, City College and Graduate Center}
		\thanks{The research of S.C.~is supported in part by NSF grant DMS-1812148.}
		\email{shirshendu@ccny.cuny.edu}

\author[J. Hanson]{Jack Hanson}
	\address{J.~Hanson\\
		Department of Mathematics\\
		City University of New York, City College and Graduate Center}
		 \thanks{The research of J.H.~is supported in part by NSF grant DMS-1954257.} 
		 \email{jhanson@ccny.cuny.edu}
	
	\author[P. Sosoe]{Philippe Sosoe}
	\address{P.~Sosoe\\
	Department of Mathematics\\
	Cornell University}
	\thanks{The research of P.S.~is supported in part by NSF grant DMS-1811093.}
	\email{psosoe@math.cornell.edu}
	
\maketitle

\begin{abstract}
    In high dimensional percolation at parameter $p < p_c$, the one-arm probability $\pi_p(n)$ is known to decay exponentially on scale $(p_c - p)^{-1/2}$. We show the same statement for the ratio $\pi_p(n) / \pi_{p_c}(n)$, establishing a form of a hypothesis of scaling theory.

As part of our study, we provide sharp estimates (with matching upper and lower bounds) for several quantities of interest at the critical probability $p_c$. These include the tail behavior of volumes of, and chemical distances within, spanning clusters, along with the scaling of the two-point function at ``mesoscopic distance'' from the boundary of half-spaces. As a corollary, we obtain the tightness of the number of spanning clusters of a diameter $n$ box on scale $n^{d-6}$; this result complements a lower bound of Aizenman \cite{A97}.
\end{abstract}

\section{Introduction}\label{sec:int}
In this paper, we address several questions involving geometric properties of the random graphs generated from the \emph{(bond) percolation} model on the canonical $d$-dimensional \emph{hypercubic lattice} $\Zd$ and its subgraphs, namely the \emph{boxes} or $\ell^\infty$ balls and the \emph{half-space} with normal direction $\ve_1$, for sufficiently high dimenson $d$.  Substantial progress has been made on the mathematical understanding of properties of these random graphs on $\Zd$ for $d$ large and $d =2$, as well as on the two-dimensional \emph{triangular lattice}.

It is well known that for any $d\ge 2$ the percolation model on $\Zd$ (and many subgraphs) exhibit a nontrivial phase transition, with a critical point separating the highly connected supercritical regime from the highly disconnected subcritical regime.
There are many useful tools and a well-developed theory for studying the percolation model on $\Z^2$ and on the triangular lattice at and near the critical point. In particular, the following key facts have been established. First, the behavior of two-dimensional percolation at criticality and near criticality are very closely related via scaling or hyperscaling relations (first observed by Kesten \cite{K87}) which relate several key quantities of interest.  Second, critical percolation on the triangular lattice exhibits conformal invariance, as shown by Smirnov \cite{S01},  which has been used to show that SLE$_6$ is the scaling limit of interfaces in the model. Finally, many power laws can be exactly computed via the connection to SLE \cite{LSW01, LSW01a}. The latter two classes of results have been proven only for the triangular lattice, though they are conjectured to extend to $\Z^2$.  
Notably, many of the aforementioned techniques apply to subgraphs of $\Z^2$ or the triangular lattice as well. We direct the reader to \cite{W09} for an overview.

For $\Zd$ with $d$ large, several key aspects of percolation are less well-understood. Much less is known about the near-critical regime and the behavior of the model in subgraphs such as sectors.  One of the  main aims of this paper is to narrow the gap between knowledge about the percolation model for $d = 2$ and for $d$ large. Another related main aim is to obtain sharp results about the tail behaviors of several quantities for which only the rough scaling behaviors had so far been identified, for example through computing low moments. We show new refined results for various connectivity probabilities involving finite boxes at the near-(sub)critical regime, and we derive tail behavior of some percolation quantities at criticality. More specifically, we obtain (a)  precise asymptotic behavior of the subcritical one-arm probability, with the correlation length determined up to constants; (b)  upper and lower bounds establishing exponential decay for both the lower tail and the upper tail probabilities of the ``chemical'' (graph) distance within open clusters; (c) upper and lower bounds establishing stretched exponential decay (with exponent $1/3$) of the lower tail of the cardinality of open clusters; and, as a result of the previous point, (d) tightness of the number of spanning clusters of large boxes on scale $n^{d-6}$, complementing a well-known result of Aizenman \cite{A97}, who derived a matching lower bound on this order. As a technical tool which may be interesting in its own right, we (e) derive up-to-constant asymptotics for connectivity probabilities in half-spaces, in the case that a vertex is ``mesoscopically close'' to the boundary of the half-space.

The questions studied here are related to longstanding conjectures about high-dimensional percolation.
 For instance, precise information about the distribution of vertices within clusters and chemical distances between far away vertices would allow one to obtain the scaling limit of simple random walk on large critical percolation clusters \cite{BCF19}. We believe that many of the results and techniques that we obtain here  could be useful for studying this and other open problems of the model.

\subsection{Definition of model and main results}
In our work, we will consider percolation with base graph the cubic or hypercubic lattice $\mathbb{Z}^d$. The usual standard basis coordinates of a vertex $x \in \Zd$ will be denoted by $x(i) = x \cdot \ve_i$, so $x = (x(1), x(2), \ldots, x(d))$. The origin is denoted by
\[0 = (0, 0, \ldots 0).\]
We will write $\|x\|_p$ for the usual $\ell^p$ norm of an $x \in \mathbb{R}^d$; if the $p$ subscript is omitted, we mean the $\ell^\infty$ norm.
The hypercubic lattice has vertex set $\Zd$  and edge set
\[\mathcal{E}(\mathbb{Z}^d):=\big\{\{x,y\}: \, \|x - y\|_1:=\sum_{i=1}^d|x(i)-y(i)| = 1\big\}.\]
(We also use the symbol $\Zd$ to refer to the graph.) Given a subset $A \subseteq \Zd$, the symbol $\partial A$ denotes the set $\{x \in A:\, \exists y \in  \Zd \setminus A \text{ with }  \|y-x\|_1 = 1\}$.

We will also consider subgraphs of the hypercubic lattice. A few other settings will be briefly discussed: we will mention some past results on the two-dimensional triangular lattice, and many high-dimensional results also extend to the \emph{spread-out lattice} having vertex set $\Zd$ but with edges between all pairs of vertices with $\ell^\infty$ distance at most some constant. In fact, the new results of the present work all extend to the spread-out lattices under standard assumptions; see the discussion at Remark \ref{remk:so} below.

The \emph{half-space} is the subgraph of the hypercubic lattice induced by the set of vertices $\Zd_+$ having nonnegative first coordinate: $\Zd_+ = \{x \in \Zd: \, x(1) \geq 0\}$. We also call \emph{half spaces} isomorphic graphs obtained by translation, reflection, or by permutation of coordinates. We note that we do not consider half-spaces with normal vectors other than $\pm \ve_i$.
The boxes or $\ell^\infty$ balls in these graphs are the subgraphs induced by the following vertex sets:
\[B(n) = [-n, n]^d \quad \text{and} \quad B_H(n) = B(n) \cap \Zd_+\ , \quad \text{respectively.} \]
As above, we blur the distinction between these vertex sets and the subgraphs they induce, using the same symbols to denote both.

We study the Bernoulli bond percolation model --- abbreviated percolation --- on the above and other subgraphs of $\Zd$. For its definition, we fix a $p \in [0,1]$ and let $\omega = (\omega_e)_{e\in\mathcal{E}(\mathbb{Z}^d)}$ be a collection of independent and identically distributed (i.i.d.) Bernoulli($p$) random variables associated to edges $e$ of $\Zd$. We write $\Omega$ for the space $\{0, 1\}^{\mathcal{E}(\Zd)}$ of possible values of $\omega$, with associated Borel sigma-algebra. An edge $e$ such that $\omega_e = 1$  will be called \emph{open}, and an edge $e$ such that $\omega_e = 0$ will be called \emph{closed}. The main object of study is the (random) \emph{open graph}, having vertex set $\Zd$ and edge set consisting of all open edges $e \in \mathcal{E}(\Zd)$, along with subgraphs of this open graph. Indeed, the open graph  of $\Zd$ naturally induces graphs on vertex subsets of $\Zd$: if $G$ is a set of vertices, then the open subgraph of $G$ has edge set consisting of those $e = \{x,y\} \in \mathcal{E}(\Zd)$ with both $x, y \in G$ and $\omega_e = 1$.


Given a realization of $\omega$ and a subgraph $G$ of $\Zd$ (including $\Zd$ itself), the \emph{open clusters} are the components of the open subgraph of $G$. To distinguish various choices of $G$, we write $\fC_G(x)$ for the open cluster containing $x$ in the open subgraph of $G \cup \{x\}$. We write $\fC(x) = \fC_{\Zd}(x)$ and $\fC_H(x) = \fC_{\Zd_+}(x)$ for brevity. We will define the event
 \begin{equation}\label{eqn: connection}
     \left\{x \sa{G} y\right\} := \{y \in \fC_G(x)\}
\end{equation}     
and we abbreviate $\left\{x \sa{\Zd} y\right\}$ to $\{x \lra y\}$. 
 
 The distribution of $\omega$ will be denoted by $\pr_p$ to indicate its dependence on the parameter $p$. We define the \emph{critical probability} (of the entire ambient graph $\Zd$) by
 \begin{equation} p_c := \inf\left\{p: \pr_p( |\fC_{\Zd}(0)| = \infty )>0\right\}. \label{eq:pcdef}\end{equation}
Here and later $|\cdot |$ denotes the cardinality of a set. When $p < p_c$ (resp.~$p = p_c$, $p > p_c$), the model is said to be \emph{subcritical} (resp.~\emph{critical}, \emph{supercritical}).  We stress that the value of $p_c$ depends on the value of $d$. One can define $p_c$ analogously for other graphs, including subgraphs of $\Zd$ --- we will touch on this in discussing some results in this introduction, but keep $p_c$ as defined in \eqref{eq:pcdef} for the remaining sections of the paper.
	
On $\Zd$ with $d \geq 2$, it is widely conjectured that $\prob_{p_c}$-almost surely there exists no infinite open cluster. Among other cases, this conjecture is proved in ``high dimensions'', when $d$ is sufficiently large; the current strongest results establish it for $d \geq 11$. For all these large values of $d$, more has been shown: for example, the probability of having long critical point-to-point connections is asymptotic to the Green's function of simple random walk. This fact is expected to be true for all $d > 6$, the expected \emph{upper critical dimension} of the model. We will discuss these issues in more detail in Section \ref{sec:pastwork}.

All results of this paper will hold as long as $d > 6$ and the aforementioned Green's function asymptotic holds. We introduce this formally, for use as a hypothesis of our theorems:
\begin{defin} \label{defin:whatishighd}The phrases \emph{high dimensions} and \emph{high-dimensional} refer to the hypercubic lattice $\Zd$ for any value of $d > 6$ such that
\[c \|x-y\|^{2-d} \leq \prob_{p_c}(x \lra y) \leq C\|x-y\|^{2-d} \]
holds for all pairs of distinct vertices $x$ and $y$, for some uniform constants $c, C > 0$.
	\end{defin}
	As mentioned above, this definition can be broadened to include the spread-out lattice; see Remark \ref{remk:so} below.
	We direct the reader to the survey \cite{HH17} for detailed discussion of high-dimensional percolation and related models. For an introduction to percolation on $\Zd$ for general $d$, and for an expository treatment of fundamental results, we refer to \cite{G99}. The book \cite{LP17} discusses percolation in some detail, including in general settings beyond the hypercubic lattice.
After the introduction, we will always assume we are in the high-dimensional setting of Definition \ref{defin:whatishighd}.

 The main results of the paper, Theorems \ref{thm: bdy-connects}--\ref{thm:scalingub} in this section, relate to the behavior of the open clusters $\fC_{B(n)}(x)$ and $\fC_{\Zd_+}(x)$ in high dimensions, for $p = p_c$ and $p < p_c$ but ``close to'' $p_c$. As we state our theorems, we will introduce the definitions of the relevant quantities of interest. To allow us to discuss past results outside of the high dimensional setting, we make these definitions for general $d$.
\begin{defin} \label{defin:armsetc}

\begin{itemize}
 \item The site $x$ has \emph{one arm} (in the extrinsic metric) to distance $n$ in $G$ if 
 \[\sup\{\|y - x\|_\infty: \, y \in \fC_G(x) \} \geq n.\]  In the case $G = \Zd$, we often simply say that $x$ has one arm to distance $n$ without referring to $G$. The corresponding events are called \emph{arm events} or \emph{one-arm events}.
 
 We also set
 \[\pi_p(n):=\prob_p(\text{the origin $0$ has an arm to distance $n$}).\]
 We sometimes write $\pi(n)$ for $\pi_{p_c}(n)$.
 \item  The \emph{correlation length} $\xi(p)$ is defined for $p < p_c$ by 
 \[\xi(p) :=- \lim_{n \to \infty} n[\log \pi_p(n)]^{-1} = -\lim_{n \to \infty} n [\log \prob_p(0 \lra n \ve_1)]^{-1};\]
 for the existence of the limit and the equality, see e.g.~\cite[(6.10) and (6.44)]{G99}.
 \end{itemize}
%
\end{defin}
We now begin to state the main results of this paper. The first theorem gives precise bounds on the asymptotic behavior  of the one-arm probability in high dimensional percolation in the regime $n \to \infty$ and $p \nearrow p_c$. 
\begin{thm}\label{thm: bdy-connects}
In the setting of percolation in high dimensions, there is a constant $C>0$, depending only on $d$, such that for all $n\in \mathbb{N}$ and for all $p\in (0,p_c]$,
\begin{equation}
    \frac{1}{C}n^{-2}\exp\big(-C n\sqrt{p_c-p}\big)\le  \pi_p(n) \le Cn^{-2}\exp\left(-\frac{ n\sqrt{p_c-p}}{C}\right).
\end{equation}
\end{thm}
 \noindent The new content of the theorem is in the case $p < p_c$. The analogous inequalities in the case $p = p_c$ are the main result of \cite{KN11}. 
 
 It is expected (see, e.g.~\cite[(9.16) and Section 9.2]{G99}) that subcritical connectivity events on linear scale $n$ obey ``scaling hypotheses'' in the simultaneous limit $n \to \infty$ and $p \nearrow p_c$: one expects quantities such as $\pi_p(n) / \pi(n)$ to behave as $f(n/\xi(p))$ for some rapidly decaying $f$. It has been shown \cite{H90} that $\xi(p) \asymp (p_c - p)^{-1/2}$ as $p \nearrow p_c$. So, in this language, Theorem \ref{thm: bdy-connects} establishes such a scaling form for $\pi_p$, up to constants in the determination of $\xi(p)$.
 Here and later, we use the usual asymptotic notation: given two functions $f, g$ on a subset $U$ of $\mathbb{R}$, we say that $f(t) \asymp g(t)$ as $t$ approaches $t_0$ if  $\limsup_{t \to t_0} f(t)/g(t)$ and $\limsup_{t \to t_0} g(t)/f(t)$ are both finite, where both limits are taken within $U$. If $f,g$ instead map $\{1, 2, \ldots\} \to [0, \infty),$ we write $f(n) \asymp g(n)$ instead of ``$f(n) \asymp g(n)$ as $n \to \infty$.''

The main estimate of Theorem \ref{thm: bdy-connects} enables us to describe certain lower tail behaviors in the critical phase. Our second result concerns the chemical distance in the critical regime.
\begin{defin}\label{defin:armsetc2}
    For $A, B\subset\Zd$, let $d_{chem}(A, B)$ denote the length --- that is, number of edges --- of the shortest open path connecting some vertex of $A$ and some vertex of $B$ if such a path exists and $\infty$ otherwise. $d_{chem}(A, B)$ is called the \emph{chemical distance} between the sets $A$ and $B$. For $x, y\in\Zd$, we write $d_{chem}(x,\cdot)$ (resp.~$d_{chem}(\cdot, y)$) to denote $d_{chem}(\{x\},\cdot)$ (resp.~$d_{chem}(\cdot, \{y\})$. If $G \subseteq \Zd$, we write $d^G_{chem}(A,B)$ for the length of the shortest open path from a vertex of $A$ to a vertex of $B$ which lies entirely in $G$, and we write $d^H_{chem} := d^{\Zd_+}_{chem}$.
    
    We denote 
    \[S_n := d_{chem}(0, \partial B(n)),\]
   the chemical distance between the origin and the boundary of the box $B(n)$.
\end{defin}

It is known  \cite{KN09, KN11, HS14} that $S_n$ is of order $n^2$ on the event that the origin has an arm  to Euclidean distance $n$. In the next theorem, we show that the lower tail of the normalized chemical distance $n^{-2}S_n$ decays exponentially.
\begin{thm}\label{thm:chem dist bd}
In the setting of critical percolation in high dimensions, there is a constant $c>0$ such that  for any $\lambda>0$
    \begin{equation}
        \label{eq:chemsmall0}
        \prob_{p_c}(S_n \le \lambda n^2\mid 0\lra \partial B(n) )\le \exp(-c\lambda^{-1}),
    \end{equation}
    and there is a constant $C >0$ such that for all $\lambda\ge Cn^{-1}$, we have:
    \begin{equation} \label{eq:chemsmall}
        \prob_{p_c}(S_n \le \lambda n^2\mid 0\lra \partial B(n) )\ge \exp(-C\lambda^{-1}).
    \end{equation}
\end{thm}

This theorem characterizes the lower tail behavior of $S_n$, with the exponential rate of decay determined up to constants. We note that on $\{0 \lra \partial B(n)\},$ we trivially have $S_n \geq n$, and so the restriction on $\lambda$ in the second part is necessary.
As a corollary of Theorem \ref{thm:chem dist bd}, we are able to derive analogous results for point-to-point chemical distances, including
\begin{equation}
  \label{eq:ptchempreview} 	\prob_{p_c}(0 \lra x,\, d_{chem}(0, x) \leq \lambda \|x\|^2) \leq C e^{-c/\lambda} \|x\|^{2-d};
\end{equation}
see Section \ref{sec:ptchem} below for this and a related statement in half-spaces.

Our third main result is the upper-tail counterpart to Theorem \ref{thm:chem dist bd}:
\begin{thm}
\label{thm:chem dist bd2}
In the setting of critical percolation in high dimensions, there is a constant $c>0$ such that  for any $\lambda>0$
    \[\prob_{p_c}(S_n \ge \lambda n^2\mid 0\lra \partial B(n) )\le \exp(-c\lambda)\ .\]
\end{thm}
Using similar but simpler arguments, we obtain the following result analogous to
 \eqref{eq:ptchempreview}, involving the upper tail of the point-to-point chemical distance within boxes
\begin{equation}
    \label{eq:ptchempreview2}
    \text{for $x \in B(n)$,} \quad  \prob_{p_c}\left(d_{chem}(0, x) > \lambda \|x\|^2 \mid 0 \sa{B(2n)} x\right) \leq \exp(-c \lambda)\ .
\end{equation}
At the end of Section \ref{sec:chemup}, we give a sketch of how to adapt the argument proving Theorem \ref{thm:chem dist bd} to prove \eqref{eq:ptchempreview2}.

Our fourth main result concerns the size of the cluster $\fC_{B(n)}(0)$ in the critical regime. It is known \cite{A97, KN11}  that $|\fC_{B(n)}(0)|$ is $O_p(n^4)$ on the event that the origin has an arm to Euclidean distance $n$. On the same event, we show that the lower tail of the normalized cluster size $n^{-4}|\fC_{B(n)}(0)|$ decays stretched-exponentially with exponent $1/3$.
\begin{thm}\label{thm:volann}
Consider critical percolation in high dimensions, and let $\alpha > 3d/2$ be fixed. There are constants $C_1, c_1 = C_1(d), c_1(d,\alpha) >0$ such that the following holds. 
    \begin{equation}
    \label{eq:volann}
    \prob_{p_c}(|\fC_{B(n)}(0)|\le \lambda n^4\mid 0\leftrightarrow \partial B(n)) \quad 
    \begin{cases}\leq \exp(-c\lambda^{-\frac{1}{3}}) & \text{ for all } \lambda> (\log n)^{\alpha} n^{-3} \\
     \geq \exp(-C\lambda^{-\frac{1}{3}}) & \text{ for all } \lambda>Cn^{-3}.
    \end{cases}
    \end{equation}
\end{thm}
 The probability appearing in \eqref{eq:volann} is zero when $\lambda < n^{-3}$, and so the theorem covers essentially the entire support of $|\fC_{B(n)}(0)|$. The interesting problem of obtaining matching constants on both sides of the inequality seems challenging, being related to well-known problems in the model.

Our fifth main result concerns the number of spanning clusters of boxes at $p = p_c$.
\begin{defin}\label{defin:armsetc3}
An open cluster $\cC$ intersecting the box $B(n)$ is called a \emph{spanning cluster} of $B(n)$ if there are vertices $x, y\in\cC$ such that $x(1)=-n$ and $y(1)=n$. We denote by $\mathscr{S}_n$ the set of spanning clusters of $B(n)$:
\[\mathscr{S}_n:= \{\fC(z), z \in B(n): \, \exists\, x, y \in \fC(z) \, \text{ such that } x(1) = -n,\, y(1)  = n \}\ .\]
\end{defin}
This quantity was analyzed by Aizenman \cite{A97}, who showed
\begin{equation}\label{eqn: manyclusters}
\prob_{p_c}(|\sS_n|\ge o(1)n^{d-6})\rightarrow 1,
\end{equation}
as $n\to\infty$. A matching upper bound ($\omega(1)n^{d-6}$)  was obtained for 
the number of spanning clusters of $B(n)$ having size $\approx n^4$. Using our estimate for the lower tail of the cluster size, we can extend the upper bound to $|\sS_n|$, which includes all spanning clusters:
\begin{thm}\label{thm:spanningtight}
In the setting of critical percolation in high dimensions, there is a constant $C >0$ such that $\E_p[|\sS_n|] \leq C n^{d-6}$. Therefore, the sequence of random variables $\{n^{6-d}|\sS_n|\}_{n=1}^\infty$ is tight.
\end{thm}
This sharpens the picture obtained in \cite{A97} for the behavior of the number of spanning clusters. Our lower tail estimates obtained in Theorem \ref{thm:volann} allows us to overcome the difficulties encountered in \cite{A97} in handling ``thin spanning clusters" having atypically small cardinality.

Our sixth and final main result, Theorem \ref{thm:scalingub}, gives bounds for the two-point function within half-spaces. We introduce some notation for this, along with the analogous notation for the two-point function in more general subgraphs, for future use.
\begin{defin}
    \label{defin:armsetc4}
     The \emph{two-point function} $\tau_p(x, y)$ denotes the  connectivity probability \[\tau_p(x,y):=\prob_p(x \lra y) = \prob_p(x \sa{\Zd} y).\] More generally, when $G \subseteq \Zd$, the \emph{two-point function restricted to $G$} is $\tau_{G,p}(x, y) = \prob_p(x \sa{G} y)$. When $G = \Zd_+$, we call $\tau_{G,p}(\cdot, \cdot)$ the \emph{half-space two-point function} and  abbreviate it to $\tau_{H,p}(\cdot, \cdot)$. We often suppress the suffix $p_c$ in $\tau_{p_c}$ and $\tau_{H,p_c}$. 
\end{defin}

\begin{thm} \label{thm:scalingub} 
There is a constant $C>0$ such that the following upper bound holds uniformly in $m \geq 0$ and $x \in \Zd_+$:
\[\tau_H(x, m\mathbf{e}_1):=\prob_{p_c}\left(x\sa{\Zd_+} m \mathbf{e}_1\right) \leq C (m+1) \|x-m\mathbf{e}_1\|^{1-d}\ . \]
There is a constant $c>0$ such that the following lower bound holds  uniformly in $m>0$, and  $x\in\Zd_+$ satisfying $x(1) \geq \frac 12\|x\|$ and $\|x\|\geq 4m$:
\[\tau_H(x, m\mathbf{e}_1)\geq c (m+1) \|x - m \mathbf{e}_1\|^{1-d}\ . \]
\end{thm}
This theorem is an extension of results of \cite{CH20}, which handled the case that at least one vertex is on the boundary of $\Zd_+$. The present theorem allows one to consider points at ``intermediate distance'' from the boundary. This is necessary for key estimates in the proofs of other theorems. We also believe it is interesting in its own right and is a potential tool for studying other properties of open clusters (see e.g.~the remark at the end of Section~3.2 of \cite{S04}).

In the high-dimensional settings of Definition \ref{defin:whatishighd}, the ``unrestricted'' two-point function $\tau(x, y) = \tau_{\Zd}(x, y)$ is asymptotic to $\|x - y\|^{2-d}$. 
Theorem 1.1(b) of \cite{CH20} shows, using this bound as input, that $\tau_H(x,y)$ is asymptotic to $\|x - y\|^{2-d}$ (resp.~$\|x - y\|^{1-d}$) if both (resp.~one of) $x$ and $y$ are macroscopically away from the boundary of $\Zd_+$ and none (resp.~one) lies on the boundary.
The asymptotic result of Theorem \ref{thm:scalingub} interpolates the above two behaviors of $\tau_H(x, y)$.

We conclude this subsection with a pair of remarks about our main results and some last definitions of important quantities in the model. The latter will be useful in the next subsection for describing past work on the model.

\begin{remk} As this work was being finalized, Hutchcroft, Michta and Slade posted a preprint \cite{HMS} proving Theorem \ref{thm: bdy-connects}, as well as an upper bound for the subcritical two-point function, along different lines from this paper. A key technical input in their proof are estimates for the expectation and tail probabilities of the volume of \emph{pioneer points} on connections to hyperplanes, using the estimates \eqref{eqn: hs-tp} of Kozma-Nachmias \cite{KN09} and the first two authors of the present paper \cite{CH20}. They use this to derive various results on percolation on high-dimensional tori of large volume, a setting we do not discuss here. Our proof of Theorem \ref{thm: bdy-connects} depends instead on some of the other results presented here, and the theorem is used to prove some others. These concern aspects of high dimensional percolation in $\mathbb{Z}^d$ at the critical point not treated in \cite{HMS}.
\end{remk}

\begin{remk}\label{remk:so}
As mentioned earlier, the above results would generalize to the spread-out lattice, where edges are placed between all vertices at $\ell^\infty$ distance at most $\Lambda$ apart (where $\Lambda \geq 1$ is an arbitrary parameter). The proofs in this paper go through with only minor modification in this case, as long as $d > 6$ and the Green's function asymptotic for the two-point function appearing in Definition \ref{defin:whatishighd} hold. These lattices hold some interest because existing methods can establish this two-point function asymptotic for the spread-out model for any $d > 6$, as long as $\Lambda$ is chosen sufficiently large. We choose to write our proofs with a focus on the hypercubic lattice purely for notational simplicity.
\end{remk}
\begin{remk}\label{remk:extend}
We believe the ideas of this paper are robust enough to extend our results to closely related cases of interest --- for instance, extending volume and chemical distance bounds to the IIC of \cite{HJ04}.
\end{remk}
\begin{defin}
    \label{defin:armsetc5}
    \begin{itemize}
    \item The \emph{density of open clusters} $\theta(p):=\prob_p(|\fC(0)|=\infty)$ denotes the probability that the origin belongs to the infinite cluster.
\item The \emph{mean finite cluster size} is denoted by $\chi(p):=\E_p[|\fC(0)|; |\fC(0)|<\infty]$.
\end{itemize}
\end{defin}

\subsection{Past work relevant for our results}\label{sec:pastwork}
Much past work has dealt with the behavior of percolation at and near criticality. By ``near critical" behavior, we mean that $p \neq p_c$ but that we consider events involving length scales at which the model looks approximately critical in some sense. While the subcritical and supercritical regimes of percolation on $\Zd$ are by now well-understood \cite{AB87} at large scales, the critical regime is only well-understood when $d=2$ and in high dimensions. The near-critical regime is fairly well-understood when $d = 2$, but less so in high dimensions (though several results, for instance the behavior of $\chi(p)$ as $p \nearrow p_c$, are known). Notably, the near-critical behavior of the one-arm probability $\pi_p$ is not yet understood in high dimensions.

Relatedly, results about certain types of connectivity events at criticality seem significantly easier to prove in two-dimensional percolation than in high dimensions. A notable example is the relation between the two-point function and one-arm probability: on $\Z^2$ at $p_c$, Kesten  \cite{K87} showed
\[\tau_{p_c}(0, n\mathbf{e}_1) \asymp \pi(n)^2 \quad \text{ as $n \to \infty$}\ . \]
This estimate is derived by connecting the clusters of $0$ and $n\mathbf{e}_1$ using the Russo-Seymour-Welsh (RSW) theorem.
The corresponding result in high dimensions, $\tau(0, n \mathbf{e}_1) \asymp n^{6-d} \pi(n)^2$ took until 2011 \cite{KN11} to establish. A main reason is the proliferation of spanning clusters in high dimensions, already noted at \eqref{eqn: manyclusters}, which prevents the use of many $d=2$ techniques based on the RSW theorem.

Bridging this gap between $d = 2$ and high dimensions is a major focus of this paper. We will put our results into context by describing past work in both of these settings. 

\subsubsection{Past relevant work in two dimensions}
At $p = p_c$, connectivity probabilities like $\pi(n)$ are believed to obey \emph{power laws}, with the powers often called \emph{critical exponents}.  The work of Kesten \cite{K87} alluded to above established a relation between the critical exponents governing $\pi$,  $\tau$, and the tail of the cluster size $|\fC(0)|$ at $p = p_c$. Remarkably, this work predated the proof of the exact values of these exponents \cite{LSW02} by about twenty years. Kesten and Zhang \cite{KZ87} built upon these ideas to show that these exponents strictly change when $\Z^2$ is replaced by a sector: if we set for $\theta > 0$
\[S_\theta := \{(r \cos \phi, r \sin \phi) \in \mathbb{Z}^2:\, r \geq 0, 0 \leq \phi < 2 \pi-\theta\} \quad \text{and } \pi(n, \theta) := \prob_{p_c}(0 \sa{S_\theta} \partial B(n)),\]
then $\pi(n, \theta) \leq n^{-\delta} \pi(n)$ for all $n \geq 1$, with $\delta$ some $\theta$-dependent constant.

In a related and important work, Kesten \cite{K87a} clarified several aspects of the near-critical behavior of percolation, showing relations between probabilities of arm events at $p_c$ (in a more general sense than that of Definition \ref{defin:armsetc}) and quantities like $\chi$, $\theta$, and $\xi$. A main and useful idea is  that $\xi(p)$ is roughly the length scale $L(p)$ at which squares become very unlikely to be crossed by a spanning cluster. This allows one to give useful bounds on near-critical connectivity probabilities: for instance
\begin{equation}
\label{eq:pitwod}    
 \text{for $p < p_c$}, \quad c_1 \exp(-C_1 k) \leq \pi_p(k L(p))/\pi_{p_c}(k L(p)) \leq C_2 \exp(-c_2 k) .\end{equation}
This can be compared to our Theorem \ref{thm: bdy-connects}.

The development of SLE \cite{Sc00} and the proof of Cardy's formula \cite{S01} allowed the computation of critical exponents for arm probabilities to be computed \cite{LSW02} on the two-dimensional triangular lattice. For instance, the one-arm probability $\pi(n) = n^{-5/48 + o(1)}$. These exponents are believed to be identical on a wide class of two-dimensional lattices, a manifestation of the \emph{universality} hypothesis. Using Kesten's results mentioned above, one can use these to compute near-critical power laws: 
\[ \theta(p)= (p-p_c)^{5/36+o(1)}, \quad
\chi(p)= |p-p_c|^{43/18+o(1)}, \quad
\xi(p)= (p-p_c)^{-4/3+o(1)}.
\]
as $p \to p_c^+$, $p \to p_c$, and $p \to p_c^-$ respectively.  SLE methods also allow computation of critical exponents for, among others, arm probabilities in the sectors $S_\theta$ defined above. Conformal invariance of the model's scaling limit makes clear how many quantities of interest vary when considering percolation on different subgraphs of the lattice.

The RSW theorem allows for a number of detailed estimates of the size of large open clusters at criticality. A recent result of this type is due to Kiss \cite{K14}, who found the sharp upper tail behavior of the size of the largest spanning cluster of a box (compare earlier results in \cite{BCKS99}). See also e.g.~\cite{vandenbergconijn} for results on the $k$th largest cluster, and \cite{garban2013pivotal} for a description of the scaling limit of the counting measure on points lying in large clusters. It is possible to prove using RSW methods and the asymptotic $\pi(n) = n^{-5/48 + o(1)}$ that 
\[-\log \prob_{p_c}(|\fC(0)| \leq \lambda n^2 \pi(n) \mid 0 \lra \partial B(n)) = \lambda^{-43/48+o(1)}\ , \]
but we have not been able to find this result in the literature.

The exponent governing the chemical distance at $p_c$ is not known on $\Z^2$ or the triangular lattice, and it appears not to be directly computable via SLE methods (see \cite{schramm2011conformally}). Aizenman-Burchard \cite{AB99} showed that chemical distances are superlinear: there is a $\delta > 0$ such that, on $\{0 \lra \partial B(n)\}$, the inequality $S_n \geq n^{1+\delta}$ holds with high probability. An upper bound for the chemical distance between sides of a box is given by the length of the lowest crossing of the box $B(n)$: on the triangular lattice, this crossing is known to have expected length $n^{4/3 + o(1)}$ \cite{morrow2005sizes}. This was improved by Damron-Hanson-Sosoe \cite{DHS21}, who showed that there also exist crossings of length at most $C n^{4/3 - \varepsilon}$; see \cite{SR21} for the case of chemical distances to the origin. Since it is not even known that $S_n = n^{s + o(1)}$ for some $s$ in dimension $d=2$, distributional results like Theorem \ref{thm:chem dist bd} on scale $n^s$ are currently out of reach.

\subsubsection{Past work in high dimensions}
The values of numerous critical exponents have been rigorously established in high dimensions, through methods very different from those available in two dimensions. A key point is that $d = 6$ is believed to be the model's \emph{upper critical dimension}, above which many critical exponents are believed to become dimension-independent, along with other aspects of the model's behavior. For $d > 6$, large open clusters should gain a degree of independence from each other --- this makes certain aspects of the model easier to understand, but also makes many RSW-type arguments inapplicable. See \cite{HH17} for an extensive review of research on high-dimensional percolation, along with related results. 

The foundational results in high dimensions are based on the Lace Expansion, adapted to percolation by Hara and Slade \cite{HS90}, who showed that $\theta(p_c) = 0$ for sufficiently large $d$. Indeed, they established the \emph{triangle condition} of Aizenman-Newman \cite{AN84}. This was extended by Hara, van der Hofstad, and Slade  \cite{HHS03} (resp.~Hara \cite{H08}), who showed the asymptotic of Definition \ref{defin:whatishighd} holds on the spread-out lattice for $d > 6$ and large $\Lambda$ (resp.~on the hypercubic lattice for $d > 19$):
\begin{equation}\label{eq:twopt}
\exists c, C >0: \quad c\|x - y\|^{2-d}\leq  \tau_{p_c}(x,y) \leq C\|x - y\|^{2-d} \text{ for all } x \neq  y \in\Zd\ .
\end{equation}
On the hypercubic lattice, the asymptotic of \eqref{eq:twopt} has so far been extended down to all $d \geq 11$ by Fitzner and van der Hofstad \cite{FH17}. It is expected to hold on the hypercubic lattice and each spread-out lattice for $d > 6$, in accord with Definition \ref{defin:whatishighd}.

In contrast to the situation on $\Z^2$, the relationships between many critical power laws took longer to establish in high dimensions. Using the triangle condition, Barsky-Aizenman showed in 1991 \cite{BA91}, 17 years before Hara's proof of \eqref{eq:twopt}, that the critical exponent for the tail of $|\fC(0)|$ is $1/2$:
\begin{equation}\label{SizeEst}
\prob_{p_c}(|\mathfrak{C}(0)| > t) \asymp t^{-1/2}.
\end{equation}
In 2011, Kozma and Nachmias \cite{KN11} computed the critical exponent governing $\pi_{p_c}(n)$:
\begin{equation} \label{eq:onearmprob}
 \pi_{p_c}(n) \asymp n^{-2}. 
\end{equation}
The proofs relating the quantities in \eqref{eq:twopt}, \eqref{SizeEst}, and \eqref{eq:onearmprob} are much more complicated than their two-dimensional analogues. We mention here also the related work \cite{KN09}, where the scaling of the \emph{intrinsic} one-arm probability was computed. We say a vertex $x$ has an intrinsic arm to distance $n$ if $x$ is the initial vertex of an open path containing at least $n$ edges. One result of \cite{KN09} is that
\begin{equation}
    \label{eq:intonearm}
    \prob_{p_c}(0 \text{ has an intrinsic arm to distance $n$}) \asymp \frac{1}{n}\ .
\end{equation}

The power laws of \eqref{SizeEst}, \eqref{eq:onearmprob}, \eqref{eq:intonearm} will be useful to us in what follows, and so we emphasize that they are shown to hold in high dimensions, in the sense of Definition \ref{defin:whatishighd}; they also hold in the spread-out model, whenever $d > 6$ and \eqref{eq:twopt} hold. 

Unlike in two dimensions, the behavior of the high-dimensional model in sectors and similar subgraphs appears to be poorly understood. The paper \cite{CH20} made advances in this direction, establishing analogues for \eqref{eq:twopt}, \eqref{SizeEst}, and \eqref{eq:onearmprob} in half-spaces. Some of these are quoted at \eqref{eqn: hs-tp} below, which says among other things that
\[\tau_{H,p_c}(0, ne_1) \asymp n^{1-d}\ . \]
 These results did not address the two-point function in the case where neither vertex is on the boundary of the half-space, which is the content of our Theorem \ref{thm:scalingub}. The paper \cite{CH20} also showed that the two-point function bound \eqref{eq:twopt} also holds in subgraphs of $\Zd$, as long as both endpoints are macroscopically far from the boundary: for each $M > 1$, there exists $c = c(M) > 0$ such that
\begin{equation}
    \label{eq:boxtwopt}
    \text{for each $n$ and all $x,y \in B(n)$,}\quad \tau_{B(Mn),p_c}(x,y) \geq c \|x-y\|^{2-d}\ .
\end{equation}

Similarly to the case of subgraphs, near-critical behavior is also less well-understood in high dimensions than on $\Z^2$, though some results are known. Notable is Hara's \cite{H90} asymptotic $\xi(p) \asymp (p_c-p)^{-1/2}$ as $p \nearrow p_c$, with $\xi$ defined in the sense of Definition \ref{defin:armsetc} so that $\pi_p(n) = \exp(-n/\xi(p) + o(n))$. Our Theorem \ref{thm: bdy-connects} sharpens this to extract the behavior of this arm probability when $n \approx \xi(p)$, giving a result analogous to \eqref{eq:pitwod}. Some other results of a near-critical type have been shown in high dimensions: for instance, the behavior of $\chi(p)$ \cite{AN84} as $p \nearrow p_c$ and $\theta(p)$ as $p \searrow p_c$ \cite{AB87} are known. The results here give less insight into the structure of open clusters than is available on $\Z^2$, where among other things it is shown that $\theta(p) \asymp \pi_p( L(p))$ as $p \searrow p_c$. Here $L(p)$ is defined for $p > p_c$ as the length scale above which the crossing of a square by a spanning cluster is very likely \cite{K87a}.

At $p_c$, exponential upper tail bounds for the cluster volume $|\fC(0)|$ conditional on $\{0 \lra \partial B(n)\}$ can be shown via the methods of Aizenman-Newman \cite{AN84} and Aizenman \cite{A97}.  The best existing upper bounds on $\prob_{p_c}(|\fC(0)| < \lambda n^4 \mid 0 \lra \partial B(n))$ appear to be of the order $\lambda^{-c}$ for some power $c$. As mentioned above Theorem \ref{thm:spanningtight}, the lower tail of $|\fC(0)|$ on $\{0 \lra \partial B(n)\}$ is related to the number of spanning clusters of a box. 
Our Theorem \ref{thm:volann} shows that this lower tail is actually stretched-exponential with power $-\frac{1}{3}$, and allows us to give a comparable upper bound to Aizenman's results on the number of spanning clusters, already mentioned at Theorem \ref{thm:volann}.

Non-optimal bounds have previously been shown for the lower tail of the chemical distance. The strongest bound to date is due to van der Hofstad and Sapozhnikov \cite{HS14}, who showed that
\[\prob_{p_c}(S_n < \lambda n^2 \mid 0 \lra \partial B(n)) \leq C \exp(-c \lambda^{-1/2}) . \]
Our Theorem \ref{thm:chem dist bd} shows that this probability is actually exponential in $\lambda^{-1}$.

A number of other recent works have studied the properties of large open clusters in high dimensions. The papers \cite{HH07, HH11, HS14} study percolation on large tori,  showing that critical percolation on such graphs mimics the critical Erd\H{o}s-R\'enyi random graph in several ways. The paper \cite{HJ04} constructs the incipient infinite cluster, an appropriately defined version of an infinite open cluster at $p_c$, and \cite{HVH14a} studies properties of this object in greater detail and from new perspectives. The paper  \cite{B15} finds the values of the ``mass dimension'' and ``volume growth exponent'' of the IIC.

\subsection{Organization of the paper, constants, and a standing assumption}

\paragraph{Organization of the paper}
The order in which we present the proofs is partially determined by dependencies between arguments.

In Section \ref{sec:further}, we define and clarify some notation and provide a few estimates which will underpin our proofs. In Section \ref{sec:HSproof}, we prove Theorem \ref{thm:scalingub}; we note this result will be invoked in several later proofs. 
In Section \ref{sec:lowertailconst}, we show the inequality \eqref{eq:chemsmall} of Theorem \ref{thm:chem dist bd}. This is by an explicit construction which forces the chemical distance to be small; this construction also guarantees that $\fC_{B(n)}(0)$ is small, and thus also proves the probability lower bound of Theorem \ref{thm:volann}.

In Section \ref{sec:bdy-connects}, we prove Theorem \ref{thm: bdy-connects} and the first inequality \eqref{eq:chemsmall0} of Theorem \ref{thm:chem dist bd}. In this argument, we make use of the inequality \eqref{eq:chemsmall} proved in Section \ref{sec:lowertailconst}.
In Section \ref{sec:chemup}, we prove Theorem \ref{thm:chem dist bd2} and sketch the proof of its point-to-point analogue \eqref{eq:ptchempreview2}. In Section \ref{sec:volnotsmall}, we prove the remaining inequality (the upper bound on the probability) of Theorem \ref{thm:volann}. Finally, in Section \ref{sec: nospanning}, we show Theorem \ref{thm:spanningtight} using Theorem \ref{thm:volann} as input.

\paragraph{Standing assumption} For the remainder of the paper, we consider subcritical and critical percolation in one of the high-dimensional settings of Definition \ref{defin:whatishighd}. We use $\prob$ (resp.~$\prob_p$) for the probability distribution of critical percolation (resp.~critical or subcritical percolation with parameter $p$). We write $\E$ (resp.~$\E_p$) for expectation with respect to $\prob$ (resp.~$\prob_p$).

\paragraph{Constants} We will generally let $c, C$ denote positive constants; $c$ will generally be small and $C$ large. These often change from line to line or within a line. All such constants will generally depend on the value of $d$ and may depend on other quantities. We will clarify the dependence of constants on other parameters when it is important and not clear from context, sometimes writing e.g.~$C = C(K)$ to indicate $C$ depends on the parameter $K$. We sometimes number constants as $C_i$, $c_i$ to refer to them locally.

\section{Further Notation and Preliminaries \label{sec:further}}
	Recall  we have introduced the $\ell^\infty$ ball or box $B(n)$. We extend the notation to boxes with arbitrary centers, writing
	\[B(x;n) = x + B(n).\]
	Similarly, we define annuli by $Ann(m, n) = B(n) \setminus B(m)$ and $Ann(x;m,n) = x + Ann(m,n)$.  Given two domains $A \subseteq D$, we write 
	\[\partial_D A = \{x \in A:\, \exists y \in D \setminus A \text{ with } \|y-x\|_1 = 1 \}.\]
	
	
	
	
	We use the symbol $\not \lra$ in the obvious way; for instance, $x \not \lra y$ means that $\fC(x) \neq \fC(y)$. When discussing a cluster $\fC_G$ or properties thereof in the case $G \neq \Zd$, we sometimes use the term \emph{restricted}; for instance, $\fC_{\Zd_+}(x) = \fC_H(x)$ is the cluster of $x$ restricted to the half-space $\Zd_+$.
	We also emphasize the slight asymmetry in the definition of restricted connections. In particular, given $D$ and $C\subseteq D$, the notation $x\sa{D\setminus C} y$ describes the event that there is an open path from $x$ to $y$ whose vertices lie in $D$ and not in $C$, with the possible exception of $x$, which is allowed to be in $C$.

	\paragraph{Correlation inequalities.}
	We recall two central correlation inequalities. An event $A$ depending on the status of the edges in $\mathcal{E}(D)$, for $D$ a subset of $\Zd$, is called \emph{increasing} if $\omega'\in A$ whenever $\omega \in \{0,1\}^{\mathcal{E}(D)}$ and $\omega \le \omega'$. The last inequality is understood componentwise, viewing $\omega$ and $\omega'$ as vectors with entries in $\{0,1\}$.
	The Harris-Fortuin-Kasteleyn-Ginibre, henceforth abbreviated as FKG, inequality states that if $A$ and $B$ are increasing events, then
	\begin{equation}\label{eqn: FKG}
	    \mathbb{P}_p(A\cap B)\ge \mathbb{P}_p(A)\mathbb{P}_p(B).
	\end{equation}
	
	For events $A$ and $B$, let $A\circ B$ denote the event of disjoint occurrence of $A$ and $B$. That is, $\omega \in A \circ B$ if there exist disjoint edge sets $E_A, E_B$ such that $\omega' \in A$ (resp.~$\omega' \in B$) whenever $\omega(e) = \omega'(e)$ for all $e \in E_A$ (resp.~for all $e \in E_B$).  The van den Berg-Kesten-Reimer inequality (or ``BK inequality'') is 
	\begin{equation}\label{eqn: BK-reimer}
	    \mathbb{P}_p(A\circ B)\le \mathbb{P}_p(A)\mathbb{P}_p(B).
	\end{equation}
	\paragraph{Russo's formula.}
	Suppose $D$ is a finite subset of $\Zd$ and $A$ is an increasing event depending on the status of edges in $\mathcal{E}(D).$ An edge $e$ is said to be pivotal for $A$ in the outcome $\omega \in \{0,1\}^{\mathcal{D}}$ if $\mathbf{1}_A(\omega) \neq \mathbf{1}_A(\omega')$, where $\omega'$ is the outcome which agrees with $\omega$ on all edges except $e$ and has $\omega'(e) = 1 - \omega(e).$
	Russo's formula \cite[Section 2.4]{G99} says that
	\begin{equation}\label{eqn: russo}
	\frac{\mathrm{d}}{\mathrm{d}p} \prob_p(A) = \sum_{e \in \mathcal{E}(D)} \prob_p(e \text{ is pivotal for $A$}).
	\end{equation}
	
	\paragraph{Cluster tail estimate.}
	
	We record a simple consequence of the estimate \eqref{SizeEst} here:
	\begin{lem}\label{eqn: cluster-est}
        There is a constant $C$ such that, uniformly for $r\ge 1 $ and $x_1,\ldots, x_r\in \mathbb{Z}^d$ and $\mu>0$, we have:
        \[\mathbb{P}(\cup_{j=1}^r \fC(x_j)>\mu r^2)\le C\mu^{-1/2}.\]
        \begin{proof}
            Write
            \begin{align*}
                \mathbb{P}\big(|\cup_{j=1}^r \fC(x_j)|>\mu r^2) &\le \sum_{j=1}^r \mathbb{P}(|\fC(x_j)|>\mu r^2\big)\\
                &+\mathbb{P}\left(\sum_{j=1}^r|\fC(x_j)|>\mu r^2,\,\text{but } |\fC(x_\ell)|\le \mu r^2 \text{ for all }  1\le \ell \le r \right).
            \end{align*}
            The first term on the right is bounded directly using \eqref{SizeEst} and a union bound, yielding
            \[Cr (\mu r^2)^{-1/2}\le C\mu^{-1/2}.\]
            For the second term with $\mu r^2 \ge 2$, Markov's inequality yields the bound
            \begin{align*}
                &(\mu r^2)^{-1}\times r\times \mathbb{E}[|\fC(0)|; |\fC(0)|\le \mu r^2]\\
                \le C &\mu^{-1}r^{-1}\int_1^{\mu r^2} t^{-1/2}\, \mathrm{d} t\le  C\mu^{-1/2}.
            \end{align*}
         \end{proof}
	\end{lem}
	
	\paragraph{A lemma on extending clusters.}
	Let $x \in D \subseteq \Zd$. We use the notation $\E[\cdot \mid \fC_D(x)]$ as an abbreviation for $\E[\cdot \mid \sigma(\fC_D(x))]$. Here $\sigma(\fC_D(x))$ is the sigma-algebra generated by $\fC_D(x)$, where we consider this cluster as a random variable taking values in $\{V \subseteq D:\, x \in V, \, |V| < \infty \}$. We extend the notation to the conditional probability $\prob(\cdot \mid \fC_D(x))$ in the obvious way.
	
	The following result appears in \cite[Lemma 3.2]{CH20}.
	\begin{lem}\label{lem: bdy-extension}
	    Let $A_0\subset A_1\subset \mathbb{Z}^d$ be arbitrary finite vertex sets with $z\in A_0$. Let $B\subset \partial A_1$ be a distinguished portion of the boundary of $A_1$ and suppose that the $\ell^\infty$ distance from $A_0$ to $B$ is $\lambda$. Then for all $M>0$, we have 
	    \[\mathbb{P}(z\sa{A_1} B\mid \fC_{A_0}(z))\le M\pi(\lambda)\]
	    almost surely, on the event $\{|\{y\in \partial_{A_1}A_0:z\sa{A_0} y\}|=M\}$.
	\end{lem}
	
	A typical application of this lemma is to estimate the probability that the cluster of $z=0$ contains too few sites on $\partial B(n/2)$ given $0\leftrightarrow \partial B(n)$. Let
	\[X = |\fC_{B(n/2)}(0)\cap \partial B(n/2)|.\]
	By \eqref{eq:onearmprob}, we have $\mathbb{P}(X>0)\le \pi(n/2)\le Cn^{-2}$. Applying Lemma \ref{lem: bdy-extension} with $A_0=B(n/2)$, $A_1=B(n)$, and $B=\partial B(n)$, and using \eqref{eq:onearmprob} again, we have
	\begin{equation}\label{eqn: X-eq}
	    \begin{split}
	    &\mathbb{P}\big(X\le \varepsilon n^2\mid   0\leftrightarrow \partial B(n)\big) \\
	    =~& \mathbb{P}\big(  0\leftrightarrow \partial B(n)\mid 0 < X\le \varepsilon n^2 \big)\cdot  \frac{\mathbb{P}(0<X\le \varepsilon n^2)}{\mathbb{P}(0\leftrightarrow \partial B(n))}\\
	    \le~& C \varepsilon n^2 \pi(n/2)\\
	    \le~& C\varepsilon.
	    \end{split}
	\end{equation}
	As an immediate consequence of \eqref{eqn: X-eq}, we have the existence of a constant $c>0$ such that
	\begin{equation}\label{eqn: X-low}
	    \mathbb{P}(X\ge cn^2)\ge c\pi(n)\ge cn^{-2}.
	\end{equation}
	
	\paragraph{Half-space two-point estimate.}\label{sec:pasthighd}
	We recall the following estimates of Chatterjee and Hanson for the two-point function in various regimes, where $K > 0$ is arbitrary and fixed:
	\begin{equation}\label{eqn: hs-tp}
	    \tau_H(x,y)\asymp \begin{cases}
	    \|x-y\|_\infty^{2-d}& \quad \text{ in } \{(x,y): 0<\|x-y\|_\infty < K \min\{ x(1),y(1)\}\};\\
	    \|x-y\|_\infty^{1-d}&  \quad \text{ in } \{(x,y): x(1)=0, \ 0<\|x-y\|_\infty < Ky(1)\};\\
	    \|x-y\|_\infty^{-d}& \quad \text{ in } \{(x,y): x\neq y, x(1)=0, y(1)=0\}. 
	
	\end{cases}
	\end{equation}
	Here the symbol $\asymp$ means that the left-hand side is bounded above and below by positive constant multiples of the right-hand side, uniformly in pairs $(x,y)$ of vertices lying in the specified regions.

\section{Half-space two-point bound near the boundary \label{sec:HSproof}}

	\subsection{Cluster boundaries and regularity .}
	To prove Theorem \ref{thm:scalingub}, we will need to use tools from \cite{KN11} to extend the cluster of a vertex from a region $D\subset \Zd$ across its boundary. We will use adaptations of these tools in some later arguments (though with differences in definitions depending on the needs of the specific problem). For this reason, we describe the setup somewhat generally here.
	
	Let $D$ be some region to which we wish to restrict connections. 
	Given such a region $D$, we denote by $Q$ a portion of its vertex boundary (possibly relative to another set --- for instance, if we are considering connections in $\Zd_+$ and $D=B_H(n)$, we might set $Q = \partial_{\Zd_+} B_H(n)$). A typical setup has us condition on the status of edges in $D$, then for a particular open cluster $\cC$ of $D$, using vertices of some such $Q$ to construct an extension of $\cC$ into a portion of $\Zd \setminus D$.
	
	\begin{defin}
	    For $K > 0$ an integer, we define
	\begin{itemize}
		\item the (random) set $\mathrm{EREG}_D(K)$ to consist of all $z \in D$ such that
		\[\E[|\fC_{\Zd}(z) \cap B(z;\ell)| \mid \fC_{D}(z)]  < \ell^{9/2}\quad \text{ for all } \ell \geq K\ ; \]
		\item The set $\mathrm{EREG}_D(A,K)$ to consist of all $z \in D$ such that $z \in \mathrm{EREG}_D(K)$ and such that
		\[\E[|\fC_{\Zd}(z) \cap B(y;\ell)| \mid \fC_{D}(z)]  < \ell^{9/2}\quad \text{ for all } \ell \geq K \text{ and $y \in A$}\ . \]
		With mild abuse, we write $\mathrm{EREG}_D(y,K)$ for $\mathrm{EREG}_D(\{y\},K)$.
		\item The set
		\begin{equation*}
		\Xi_{D,Q}(x) := Q\cap \fC_{D}(x)\ .
		\end{equation*}
		We abbreviate $X_{D,Q} := |\Xi_{D,Q}|$.
		Similarly, we let
		\[\Xi_{D,Q}^{\mathrm{EREG}}(x) := \Xi_{D,Q}^{\mathrm{EREG}}(x, m; K) = \Xi_{D,Q} \cap \mathrm{EREG}_{D}(\{0,\,m\mathbf{e}_1\}, K)\ , \]
		and $X_{D,Q}^{\mathrm{EREG}}(x):=| \Xi_{D,Q}^{\mathrm{EREG}}(x)|$.            
		\end{itemize}
	\end{defin} 
	\subsection{Regularity}
		Consider the half-space $\Zd_+$, and let $n \geq 4m \geq 4$. We assume
		\begin{equation}\label{eqn: x-cond}
		    \|x\| = n \text{   and   } x(1) \geq n/2,	
		    \end{equation}
		where the fraction $1/2$ is arbitrary and could be replaced by any fixed number in $(0, 1)$. Our main result, Theorem \ref{thm:scalingub}, will be uniform in such $x$ and in $m, n$ as above.
	 We  decompose the connection $x\leftrightarrow m\mathbf{e}_1$ into a connection from $x$ to $B_H(2m)$ lying entirely in $\Zd_+ \setminus B_H(2m)$ and then a further connection from some point of $\partial B_H(2m)$ to $m \mathbf{e}_1$. We thus introduce the following notation:
	 \begin{figure}
    \centering
    \includegraphics[scale=0.35]{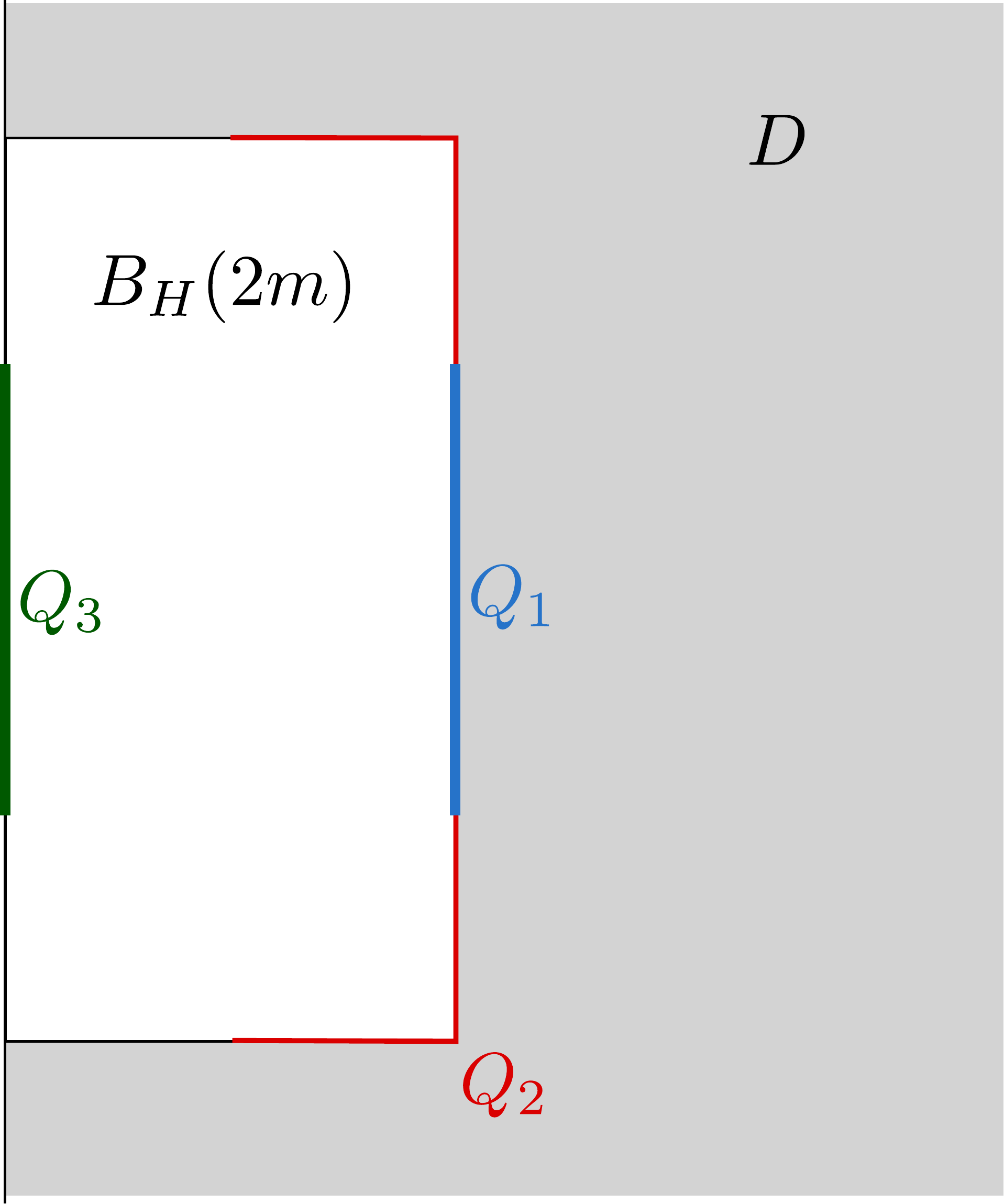}
    \caption{Geometry of the definitions in \eqref{eq:hsgeom}}
    \label{fig:halfspace}
\end{figure}

	\begin{equation}\label{eq:hsgeom}
	\begin{split}
	 D &= \Zd_+ \setminus B_H(2m);  \quad Q_1 = \{2m+1\} \times [-m,m]^{d-1},\\ 
	 \ Q_2& = [\partial_{\Zd_+}D] \cap [m \mathbf{e}_1 +\Zd_+], \quad  Q_3=\{0\}\times [-m,m]^{d-1}.
	 \end{split}
	 \end{equation}
	 See Figure \ref{fig:halfspace}. Our goal in this section is to check the show that vertices  $z\in Q=Q_1$ on the boundary of $D$ are regular in the sense of the previous section.
	
	We recall here two results which are useful for our purposes. 
	\begin{lem}[\cite{A97}; \cite{KN11}, Lemma 4.4]\label{lem:aiznew}
	There are constants $c, C$ such that, for all $r \geq 1$ and all $\lambda \geq 1$,
	\[\prob\left(\max_{y \in B(r)} |\fC(y) \cap B(r)| > \lambda r^4  \right) \leq C r^{d-6} \exp(- c \lambda)\ . \]
		\end{lem}
	\begin{lem}[\cite{KN11}, Lemma 1.1]\label{lem:knapriori}
	Uniformly in $m$ and in $w_1, w_2 \in B(r)$, we have
		\begin{equation}
		\label{eq:knapriori}
		\prob(w_1 \sa{B(r)} w_2) \geq c \exp(-C \log^2 r)\ .
		\end{equation}
		In particular,
			\begin{equation}
		\label{eq:knapriori2}
		\prob(w_1 \sa{B_H(r)\setminus B_H(2m)} w_2) \geq c \exp(-C \log^2 r)
		\end{equation} 
		uniformly in $r \geq 4m$, and in $w_1$, $w_2 \in B_H(r) \setminus B_H(2m)\ $.
		\end{lem}	
	
	We now prove a regularity lemma similar in flavor to  \cite[Theorem 4]{KN11}. It is weaker than theirs in one sense: it only controls the probability that a given vertex is regular, rather than trying to control the total number of regular vertices. On the other hand, it is slightly stronger in the sense that we control regularity ``at an arbitrary base point'': roughly speaking, conditional on part of $\fC(z)$, the remaining portion of $\fC(z)$ is not likely to be too dense near a fixed vertex $y$.
	
	\begin{lem}\label{lem:easyreg}
		There exist constants $c, C>0$ such that the following holds uniformly in $m$, in $k \geq 1$, in $\lambda \geq 1$, in $x \in \Zd_+ \setminus B_H(4m)$, in $y \in B_H(2m)$ and in $z \in Q_2$:
	\begin{align}\label{eq:easyreg}
	\prob\left(\left|\fC(z) \cap B(y;k) \right| > \lambda k^{4} \log^5 (k) \mid z \sa{D} x\right)\leq C \exp(-c \sqrt{\lambda} \log^3 k)\ .
	\end{align}
	In particular, there exists a $K_0 > 0$ such that (uniformly as above), for all $K > K_0$, 
	\begin{equation}
	\label{eq:easyrreg2}
	    \prob(z \notin \mathrm{EREG}_{D}(y;K) \mid z \sa{D} x) \leq C \exp(-c K^{1/4})\ .
	    \end{equation}
	\end{lem}
\begin{proof}
		We begin by proving \eqref{eq:easyreg}. For this, it suffices to prove the following slight modification of the claim in the lemma:
		\begin{equation}
		\label{eq:flatmodify}
		\begin{gathered}
		\text{Given $k$ as in the statement of the lemma, there exists }\\
		\text{$k' \in [k, 4^d k]$ such that \eqref{eq:easyreg} holds with $k$ replaced by $k'$.}
		\end{gathered}
		\end{equation}
		Indeed, given \eqref{eq:flatmodify}, the statement of the lemma follows by noting that for such $k,k'$, 
		\[|\fC(z)\cap B(y;k)| \leq |\fC(z)\cap B(y;k')|.\] The reason to prove \eqref{eq:flatmodify} is due to a minor technicality which will become clear at the end of the lemma. For most of what follows, we endeavor to prove that the bound of \eqref{eq:easyreg} holds for all $k$, and we will discover that we  have to prove \eqref{eq:flatmodify} to dispose of some ``exceptional'' values of $k$.
		
		If $\|x-y\| \leq k^d$, then we have $\|x -z\| \leq 4 k^d$ and so by  \eqref{eq:knapriori2} we have 
		\[\prob(z \sa{\Zd_+ \setminus B_H(2m)} x) \geq c \exp(-C \log^2 k).\]  
		In this case, we can upper-bound \eqref{eq:easyreg} by
	\begin{align*} C \exp(C \log^2 k) \prob(\left|\fC(z) \cap B(y;k) \right| > \lambda k^{4} \log^5 (k) ) \leq C \exp(-c \lambda \log^3 k)
	\end{align*}
	where we have used the tail result of Lemma \ref{lem:aiznew}. 
		
		We now treat the case that $k$ is small, that is $x \notin B(y;k^d)$.
	Let
		\[ A_k := \{\text{for each cluster $\cC$ of $B(y; k^d)$, we have $|\cC \cap B(y;k)| \leq \sqrt{\lambda} k^4 \log^3 k$} \}\ ,\]
		where $\cC$ being a cluster of $B(y;k^d)$ means considered as a component of the open subgraph of $B(y;k^d)$ (no connections outside this box are allowed).
		We also let
		\[ A'_k := \{\text{there are no more than $\sqrt{\lambda} \log^2 k$ disjoint connections from $B(y;k)$ to $\partial B(y;k^d)$} \}.\]
		We can bound each of these events' probabilities, using the one-arm probability asymptotic \eqref{eq:onearmprob}, the BK inequality \eqref{eqn: BK-reimer}, and the cluster tail bound of Lemma \ref{lem:aiznew}: for each $\lambda \geq 1$,
		\begin{equation}
		\begin{split}
		\label{eq:preedmund}
		\prob(A_k) &\geq 1 - \exp(-c\sqrt{\lambda} \log^3 k);\\
		\prob(A_k') &\geq 1 - (C k^d \times k^{-2d} )^{\sqrt{\lambda}\log^2 k} \geq 1 - \exp(-c \sqrt{\lambda}\log^3 k)\ .
		\end{split}
		\end{equation}

		We note that if there are at most $\ell$ disjoint crossings of $B(y;k^d) \setminus B(y;k)$, then 
		\[ \fC(z) \cap B(y;k)\subset \cup_{\cC} [\cC \cap B(y;k)],\] where the union is over at most $ \ell+1$ clusters $\cC$ of $B(y;k^d)$.

		 In particular on the event $A_k \cap A_k'$, we have  
		 \[|\fC(z) \cap B(y;k)| \leq \lambda k^4 \log^5 k.\]
		It therefore suffices to show, for $x \notin B(y; k^d)$,
		\begin{equation}
		\label{eq:edmundfitzgerald}
		 \prob(A_k \cap A_k' \mid x \sa{D} z) \geq 1 - \exp(-c \sqrt{\lambda} \log^3 k)\ .
		 \end{equation}
		 We do this by conditioning on the cluster outside $B(y;k^d)$, noting that $A_k$ and $A_k'$ are independent of the status of edges outside $B(y;k^d)$. We write
		 \begin{align}
		 \prob( \{ x \sa{D} z\} \setminus [A_k \cap A_k']) &\leq \sum_{\cC} \prob(\fC_{D\setminus B(y;k^d)} (x) = \cC) [1 - \prob(A_k \cap A_k')]  \nonumber\\
		 &\leq C \exp(-c \sqrt{\lambda} \log^3 k) \sum_{\cC} \prob(\fC_{D \setminus  B(y;k^d)} (x) = \cC)\ , \label{eqn: C-sum}
		 \end{align}
		 where the sum is over $\cC$ compatible with the event $\{x \sa{D} z\}$ (in other words, such that $\prob(x \sa{D} z \mid \fC_{D \setminus  B(y;k^d)} (x) = \cC)$ is nonzero) and we have used \eqref{eq:preedmund}. To show \eqref{eq:edmundfitzgerald}, we need to compare the sum on the right to $\prob(x \sa{D} z)$. We will show that each term of that sum is at most $\exp(C \log^2 k) \prob(\fC_{D \setminus B(y;k^d)} (x) = \cC,\,z \sa{D} x)$.
		 
		 For a cluster $\cC$ as in \eqref{eqn: C-sum} to be compatible with $\{x \sa{D} z\}$, there are two possibilities: either $x$ is connected to $z$ in $\cC$, or it is possible to build an open connection from $x$ to $z$ which passes through $B(y;k^d)$. In the former case, we have 
		 \[\prob(\fC_{D \setminus B(y;k^d)} (x) = \cC) =  \prob(\fC_{D \setminus B(y;k^d)} (x) = \cC,\,z \sa{D} x)\ .\]
		 In the latter case we can measurably choose two disjoint open connections in $\cC \cup \{x,z\}$, one from $x$ to $B(y; k^d)$ in $\Zd_+$ and one from $z$ to $B(y; k^d)$ in $\Zd_+$. If $z \in B(y; k^d)$, the latter ``connection'' consists of the vertex $z$, considered as a trivial open path. Given such disjoint connections to $B(y;k^d)$, we denote by  $\zeta_x$ the endpoint on $\partial B(y;k^d)$ of the connection started from $x$, and by $\zeta_z$ the endpoint of the connection started from $z$. The vertex $\zeta_z$ lies in $\partial B(y;k^d)$ unless $z \in B(y;k^d)$, in which case $\zeta_z = z$.
		 
		 If $\fC_{D \setminus B(y;k^d)}(x) = \cC$ and if $\zeta_x \sa{B(y;k^d)\cap D} \zeta_z$, then $x \sa{D} z$. The former two events depend on different edge sets and are hence independent. Therefore, as long as 
		 \begin{equation}
		 \label{eq:tinyglue}
		 \prob(\zeta_x \sa{B(y;k^d)\cap D} \zeta_z) \geq \exp(-c \log^2 k)\ ,
		 \end{equation}
		 we can upper bound each term of \eqref{eqn: C-sum} by
		 \[\prob(\fC_{D\setminus B(y;k^d)} (x) = \cC) \leq \exp(C \log^2 k) \prob(\fC_{D \setminus B(y;k^d)} (x) = \cC,\,z \sa{D} x)\ . \]
		 Plugging this back in, we find in this case that
		 \[\prob(\{ x \sa{D} z\}\setminus[A_k \cap A_k']) \leq C \exp(-c \sqrt{\lambda} \log^3 k)\ .\]
		 Combining the two cases, \eqref{eq:edmundfitzgerald} and hence \eqref{eq:easyreg} follows. 
		 
		 So it remains to finally argue for \eqref{eq:tinyglue}. We note that $D \cap B(y;k^d)$ is a union of at most $4^d$ rectangles. As long as none of these rectangles is too ``thin'', that is does not have the ratio of its longest sidelength to its smallest sidelength larger than $10$, then \eqref{eq:tinyglue} follows easily from Lemma \ref{lem:knapriori}. In case at least one such rectangle is thin, for instance if $y$ has distance $k^d - 1$ from $D$, so that one rectangle has smallest sidelength $1$, it is easy to see that there exists some $k' \in [k, 4^d k]$ such that no rectangles making up $B(y; k^d ) \cap D_2$ are thin. Again for this $k'$ \eqref{eq:tinyglue} follows, and so we have established \eqref{eq:flatmodify}. This establishes \eqref{eq:easyreg}.
		 
		 We will conclude the proof by showing \eqref{eq:easyrreg2}. Successively conditioning in \eqref{eq:easyreg}, we have
		 \begin{align*}
		     \E\left[ \prob( |\fC(z) \cap B(y;k)| > k^{9/2}/2 \mid \fC_D(z)) \,\Big|\, z \sa{D} x\right] \leq \exp\left(-c k^{1/4} \log^{1/2} k \right)\ .
		 \end{align*}
		 Using Markov's inequality, we see
		 \begin{equation}
		     \label{eq:edend}
		\prob\left( \prob( |\fC(z) \cap B(y;k)| > k^{9/2}/2 \mid \fC_D(z)) \geq \exp(-k^{1/4}) \,\Big|\, z \sa{D} x\right) \leq \exp(-c k^{1/4})\ .
		 \end{equation}
		 
		 Noting that
		 \[\E\left[ |\fC(z) \cap B(y;k)| \mid  \fC_D(z) \right] \leq \frac{k^{9/2}}{2} + k^d \prob\left(|\fC(z) \cap B(y;k)| > \frac{k^{9/2}}{2} \,\Big |\,  \fC_D(z)\right) \]
		 and applying \eqref{eq:edend}, we find for all large $k$
		 \begin{align*}
		     &\prob\left( \E[|\fC(z) \cap B(y;k)| \mid \fC_D(z)] > k^{9/2}\, \Big |\, z \sa{D} x\right)\\ \leq &\prob\left( \prob\left(|\fC(z) \cap B(y;k)| > \frac{k^{9/2}}{2} \,\Big |\,  \fC_D(z)\right) \geq k^{9/2-d} \,\Big |\, z \sa{D} x\right)\\
		     &\leq \exp(-c k^{1/4})\ .
		 \end{align*}
		 The bound \eqref{eq:easyrreg2} and hence the lemma now follow by choosing $K_0$ sufficiently large.

		

		\end{proof}
		A direct consequence of the above is the following lower bound on the size of $\mathrm{EREG}_{D}$. 
	\begin{lem}\label{lem:topreg}
		There exist constants $K_0, c > 0$ such that the following holds uniformly in $m$, in $x$ satisfying \eqref{eqn: x-cond}, in $z \in Q_1$, and in $K > K_0$:
		
		\[\prob\left(z\in \mathrm{EREG}_{D}(\{0,\, m \mathbf{e}_1\},K), \, z \sa{\Zd_+} x\right) \geq c n^{1-d}\ . \]
		\end{lem}
		\begin{proof}
                  Applying the half-space two-point function bound \eqref{eqn: hs-tp} and Lemma \ref{lem:easyreg}, we bound uniformly in $m$, $x$, $z$ as above and uniformly in $K$:
                  \begin{align*}
                  \prob(z \in \mathrm{EREG}_{D}(\{0, m\mathbf{e}_1\}), z \sa{D} x) &\geq c n^{1-d}[1 -  \prob(z \notin \mathrm{EREG}_{D}(y;K) \mid z \sa{D} x)]\\
                  &\geq c n^{1-d}[1 -  C \exp(-c K^{1/4})]\ .
                  \end{align*} 
                  The result follows by enlarging $K_0$ from Lemma \ref{lem:easyreg} if necessary.
			\end{proof}
			
		\subsection{Gluing}	
		We have already shown a lower bound for $\E[X_{D,Q_1}]$ in Lemma \ref{lem:topreg}.
		Our goal now is to upper bound $\E[X_{D,Q_2}]$. This subsection provides the groundwork for this by showing that in a sense, most vertices of $\Xi_{D, Q_2}$ have conditional probability $m^{2-d}$ to connect to $m \mathbf{e}_1$ in $\Zd_+$ and similarly have conditional probability $m^{1-d}$ to connect to $0$ in $\Zd_+$.

			\begin{defin}
			    For each $z \in Q_2$, we choose a deterministic neighbor $z' \in \Zd_+ \setminus D = B_H(2m)$. 
			For each $K$ and for any $y \in B_H(2m)$, we let $Y(y) = Y(y,m,x;K)$ be the (random) number of $z \in  Q_2$ satisfying the following properties:
			\begin{enumerate}
				\item $z \in \Xi_{D, Q_2}^{\mathrm{EREG}}(x, m; K)$;
				\item The edge $\{z, z'\}$ is open and pivotal for the event $\{x \sa{ \Zd_+} y \}$.
			\end{enumerate}
			We will ultimately choose a large nonrandom $K$, fixed relative to $m$ and $x$.
			\end{defin}
			
			The following facts relate $Y(y)$ to the cluster of $x$.
			\begin{prop}\label{prop:competingglue}
				For each $m$ and $K$, and any $x \in \Zd_+ \setminus B_H(4m)$, $y \in B_H(2m)$, we have
				\begin{equation}
				\label{eq:topivglue}
				\prob(x \sa{\Zd_+} y) \geq \prob(Y(y)> 0) = \sum_{z \in Q_2} \prob(z \in Y(y))\ .
				\end{equation}
				We also have
				\[\prob(x \sa{\Zd_+} 0 ) \leq C \|x\|^{1-d}\]
				and so
				\begin{equation}
				\label{eq:Yysumbd}
				\sum_{z \in Q_2} \prob(z \in Y(0) ) \leq C \|x\|^{1-d}\ .
				\end{equation}
			\end{prop}
			\begin{proof}
				The first inequality of \eqref{eq:topivglue} is a consequence of the definition of $Y$, so we begin by proving the subsequent equality. This equality follows immediately once we establish that $\{Y(y) > 0\}$ is equal to the disjoint union $\cup_{z \in Q_2} \{z \in Y(y)\}$ --- in other words, $Y(y)$ is either empty or a singleton. 
				
				To show this, we fix an outcome and suppose that $z_1$ and $z_2$ are two distinct elements of $Y(y)$ --- since $x \sa{\Zd_+} y$ when $Y(y)$ is nonempty, there is some open self-avoiding path $\gamma$ connecting $x$ to $y$ in $\Zd_+$. By the pivotality condition in the definition of $Y(y)$, it follows that this path must pass through both $\{z_1, z_1'\}$ and $\{z_2, z_2'\}$. Suppose, relabeling if necessary, that $\gamma$ passes first through $\{z_1, z_1'\}$; letting $\tilde \gamma$ be the terminal segment of $\gamma$ beginning with the edge $\{z_2, z_2'\}$, we have $z_1 \notin \tilde \gamma$.
				
				Now we produce a new open path $\widehat \gamma$ by appending a path from $x$ to $z_2$ lying entirely in $D$ to the path $\tilde \gamma$. Then $\widehat \gamma$ connects $x$ to $y$ in $\Zd_+$, and it avoids the edge $\{z_1, z_1'\}$, since $\tilde \gamma$ does, and since $\{z_1, z_1'\}$ does not have both endpoints in $D$. This contradicts the fact that $\{z_1, z_1'\}$ is open and pivotal (even when we close this edge, the path $\widehat \gamma$ still connects $x$ to $y$), and so we have shown the claim about $Y(y)$ and hence \eqref{eq:topivglue}.
				
				The inequality above \eqref{eq:Yysumbd}  is a consequence of \eqref{eqn: hs-tp}, and then \eqref{eq:Yysumbd} follows by an application of the already-proved \eqref{eq:topivglue}.
				\end{proof}
			
			We now show that for typical $z \in Q_2$, the conditional probability 
			\[\prob(z \in Y(y) \mid z \in \Xi_{D, Q_2}^{\mathrm{EREG}}(x))\] 
			is at least order $m^{2-d}$ when $y = m\mathbf{e}_1$ and at least order $m^{1-d}$ when $y \in Q_3$. In fact, we prove the former bound on average, for vertices within order constant distance of $m \mathbf{e}_1$.
			\begin{prop}\label{prop:yprop}
				We have the following bounds on the expectation of $Y(y)$, covering the cases of $y \in Q_3$ and $y \in B(m\mathbf{e}_1; K)$. These hold uniformly in $m \geq 1$, in $x \notin B_H(4m)$, with $K$ fixed relative to $x, m, n, N$ but larger than some constant $K_1 > K_0$ (uniform in $x, m, n, N$).
				\begin{itemize}
				\item There exists a constant $c > 0$ such that
				\[ \sum_{y \in Q_3}\E[Y(y); X_{D,Q_2}^{\mathrm{EREG}} = N] \geq c N \prob(X_{D,Q_2}^{\mathrm{EREG}} = N)\ .\]
				\item  There exists a constant $c  > 0$ such that
				\[\sum_{y \in B(m\mathbf{e}_1;K)}\E[Y(y); X_{D,Q_2}^{\mathrm{EREG}} = N] \geq c N m^{2-d} \prob(X_{D,Q_2}^{\mathrm{EREG}} = N)\ . \]
				\end{itemize}
				\end{prop}
			\begin{proof}
				This is a now-familiar extension argument originating in Kozma-Nachmias \cite{KN11}, with adaptations to half-spaces from Chatterjee-Hanson \cite{CH20}. We define three families of events, indexed by vertices of the lattice: 
				\begin{align*}
				\cE_1(z) &= \left\{z \in \Xi_{D,Q_2}^{\mathrm{EREG}}(x), \, X_{D,Q_2}^{\mathrm{EREG}}(x) = N  \right\}\ ;\\
				\cE_2(z, z^*,y) &= \left\{z^* \sa{\Zd_+ \setminus \fC_{D}(z) } y\right\}\ ;\\
				\cE_3(z, z^*) &= \left\{\fC(z) \cap \fC(z^*) = \varnothing  \right\}\ .
				\end{align*}
				Here the variable $z$ ranges over $Q_2$ and, for a given value of $z$, the variable $z^*$ ranges over the set
				\[\Delta_K(z) := [B(z; 2K) \setminus B(z; K)] \cap B_H(2m)\ , \]
				noting that $|\Delta_K(x)| \geq  (K-1)^d$ for all $x \in Q_2$, and all $K_0 < K<  m/8 < n/2$. The variable $y$ is an element of $B_H(2m)$, though we will specialize to $y \in Q_3$ or $y \in B(m\mathbf{e}_1;K)$.
				
				Our goal is to show that $\cE_2$ and $\cE_3$ have appropriately large probability, given $\cE_1$. That is, we hope to show:
				\begin{lem}\label{lem:Estack}
					There exists a  constant $K_1 > K_0$ such that, for each $K_1 < K < m/8$  there is a  $c = c(K) > 0$ such that, for  each $x \notin B_H(2m)$, the following hold. 
					\begin{enumerate}
					\item For each $z \in Q_2$, there exists $z^* \in \Delta(z)$ such that
					\[\sum_{y \in Q_3}\prob(\cE_1(z) \cap \cE_2(z,z^*,y) \cap \cE_3(z, z^*)) \geq c \prob(\cE_1(z))\ . \]
					\item For each $z \in Q_2$, there exists $z^* \in \Delta(z)$ such that
					\begin{equation}
					\label{eq:e3smallbox} \sum_{y \in B(m\mathbf{e}_1; K)} \prob(\cE_1(z) \cap \cE_2(z,z^*,y) \cap \cE_3(z, z^*)) \geq c m^{2-d} \prob(\cE_1(z))\ .\end{equation}
					\end{enumerate}
					\end{lem}
					\begin{proof}
					We first show an analogous statement involving just the first two events: for each large $K$, there exists $c = c(K) > 0$ such that
						\begin{equation}\begin{gathered}
						\label{eq:e2first}
						\text{for $m > 8K$, for $z \in Q_2$, for each $z^* \in \Delta(z)$ and for $y \in B(m\mathbf{e}_1;K)$ or $y \in Q_3$,}\\\prob(\cE_1(z) \cap \cE_2(z, z^*, y)) \geq c \prob(z^* \sa{\Zd_+} y) \prob(\cE_1(z)).
						\end{gathered}
						\end{equation}
						To see this, we note that $\cE_1(z)$ is measurable with respect to the sigma-algebra generated by $\fC_{D}(z)$, and we write
						\begin{align*}
						\prob(\cE_2(z,z^*,y)\cap \cE_1(z)) = \sum_{\cC \in \cE_1(z)}\prob(\cE_2(z,z^*,y) \mid \fC_{D}(z) = \cC) \prob(\fC_{D}(z) = \cC),
						\end{align*}
						where the sum is over $\cC$ such that $\cE_1(z)$ occurs when $\fC_{D}(z) = \cC$.
						
						Now, for each $\cC$ as above,
						\begin{align}
						\prob(\cE_2(z,z^*,y) \mid \fC_{D}(z) = \cC) = \prob(z^* \sa{\Z^d_+\setminus \cC} y)\ ,\label{eq:armyblock}
						\end{align}
					where we can now treat $\cC$ as a deterministic vertex set. Taking a union bound, the probability in \eqref{eq:armyblock} is at least
						\begin{align*}
						&\prob(z^* \sa{\Z^d_+}y)
						- \sum_{\zeta \in \cC}\prob\left(\left\{ z^* \lra \zeta \right\} \circ \left\{ \zeta\sa{\Zd_+} y\right\}\right)\\
					\geq  &\prob(z^* \sa{\Zd_+} y )
					- \sum_{\zeta \in \cC}\prob\left( z^* \lra \zeta \right)\prob\left( \zeta \sa{\Zd_+} y \right)\ .
						\end{align*}
	Because $\zeta \notin B_H(2m)$, the final factor appearing above is at most $C m^{2-d}$ (in case $y \in B(m\mathbf{e}_1;K)$) or $C m^{1-d}$ (in case $y \in Q_3$). On the other hand, we have identical (up to constant factors) lower bounds for  $\prob(z^* \sa{\Zd_+} y)$. We thus obtain the lower bound 
						\[\prob(z^* \sa{\Zd_+} y)
						- C \prob(z^* \sa{\Zd_+} y)\sum_{\zeta \in \cC} \prob(z^{*} \lra \zeta) \]
						for the expression appearing in \eqref{eq:armyblock}.
						
						We now use the fact that (on $\fC_{D}(z) = \cC$) the vertex $z \in \Xi_{D,Q_2}^{EREG}(x,m;K)$ to upper bound the sum appearing in the last expression:
						\begin{align*}
						\sum_{\zeta \in \cC} \prob(z^{*} \lra \zeta) &\leq C \sum_{\ell \geq \log_2{K/2}} 2^{(2 - d)\ell} [\cC \cap B(z^*,2^\ell)]\\
						&\leq C \sum_{\ell \geq \log_2{K/2}} 2^{(13/2 - d)\ell}\\
						&\leq C K^{13/2 - d}\ .
						\end{align*}
					Our shorthand in the limits of summation means $\ell$ is summed over integers satisfying the specified inequality. Inserting the above bounds into the left-hand side of \eqref{eq:e3smallbox} and summing over $\cC$ shows \eqref{eq:e2first}.
						
						We next argue that
						\begin{equation}
						\begin{gathered}
						\label{eq:e3first}
						\text{For large $K$, there is a $c = c(K)>  0$ such that, for } K < m/8 < n/2 \text{ and } z \in Q_2, \text{ there is}\\
						\text{a $z^* \in \Delta(z)$ such that }
						 \sum_{y \in A} \prob(\cE_2(z,z^*,y) \cap\cE_3(z,z^*) \mid \cE_1(z)) \geq \begin{cases} c, \, &A = Q_3;\\
						 	c m^{2-d},\, &A = B(m\mathbf{e}_1;K)\ .
						 	\end{cases}
						\end{gathered}
						\end{equation}
						To show \eqref{eq:e3first}, we again condition on $\fC_{D}(z) = \cC$ for a $\cC$ such that $\cE_1(z)$ occurs; we will upper bound
						\begin{equation}\label{eq:toe3it}
						|\Delta(z^*)|^{-1} \sum_{y \in A}\sum_{z^* \in \Delta(z)} \prob(\cE_2(z,z^*,y) \setminus \cE_3(z,z^*) \mid \fC_{D}(z) = \cC)
						\end{equation}
						by a quantity smaller than that appearing in \eqref{eq:e2first}. From this and \eqref{eq:e2first}, it follows that the bound on the right-hand side of \eqref{eq:e3first} holds for a uniformly chosen random $z^* \in \Delta(z)$, hence for some particular value of $z^*$.
						
					    Given $\fC_D(z) = \cC$, the event $\cE_2(z,z^*,y) \setminus \cE_3(z,z^*)$ implies the following disjoint occurrence happens:
						\begin{equation}
						\label{eq:threesummer}
						\bigcup_{\zeta \notin \cC}\{\cC \lra \zeta\} \circ \{z^* \sa{\Zd_+ \setminus \cC}  \zeta \} \circ \{\zeta \sa{\Zd_+ \setminus \cC} y \}\ ;
						\end{equation}
						here the event $\{\cC \lra z\}$ is interpreted with $\cC$ treated as a deterministic vertex set (and so this is an upper bound---in fact, the connection from $\cC$ to $\zeta$ is in $\Zd_+ \setminus D$).
						Applying the BK inequality and summing, we see the probability of the event in \eqref{eq:threesummer} is at most
						\begin{align}\label{eq:totwoterms}
                                                   \sum_{\zeta \notin \cC} \prob(\zeta \lra \cC\mid \fC_{D}(z)=\mathcal{C}) \prob( z^* \lra \zeta) \prob(\zeta \sa{\Zd_+} y)
                       \end{align}
                       In other words, we have shown that
                       \begin{equation}
                       \label{eq:e3itagain}
                       \eqref{eq:toe3it} \leq |\Delta(z^*)|^{-1}  \sum_{y \in A, \, z^* \in \Delta(z)}\sum_{\zeta \notin \cC} \prob(\zeta \lra \cC\mid \fC_{D}(z) =\cC) \prob(\zeta \lra z^*) \prob(\zeta \sa{\Zd_+} y)\ .
                       \end{equation}
                       
                        The precise bound we find for \eqref{eq:e3itagain} depends on whether we are summing over $y \in Q_3$ or $y \in B(m\mathbf{e}_1;K)$, though the structure is similar in both cases. 
                        \paragraph{Case $A = Q_3$.} We  bound the sums appearing in \eqref{eq:e3itagain} by
                       \begin{align}
                        \eqref{eq:e3itagain} \leq C K^{-d} \sum_{y \in Q_3, \, z^* \in \Delta(z)} \sum_{\zeta \notin \cC} \prob(\zeta \lra \cC\mid \fC_{D}(z)=\mathcal{C}) \prob(\zeta \lra z^*)\|\zeta - y\|^{1-d}\label{eq:stripdelta}
                       \end{align}
                       We have used the fact that $|\Delta(z)| \geq cK^d$ and the two-point function bound \eqref{eqn: hs-tp}.
                       
                       We further decompose the sum in \eqref{eq:stripdelta} depending on whether $\zeta \in B_H(3m/2)$ or $\zeta \notin B_H(3m/2)$. In the former case, we use the uniform upper bound
                       \begin{equation}
                           \label{eq:constlog}
                       \max_{\zeta \in \Zd} \sum_{y \in Q_3} \|\zeta - y\|^{1-d} \leq C \log m 
                       \end{equation}
                       to bound the $y$ sum for fixed $\zeta$, $z^*$. Moreover, for each such $\zeta$ we have $\prob(\zeta \lra z^*) \leq C m^{2-d}.$ Pulling these together, the portion of \eqref{eq:stripdelta} where $\zeta$ is summed over $B_H(3m/2)$ is bounded by
                       \begin{equation}
                           \label{eq:zetaclose}
                       C m^{2-d} \log m \sum_{\zeta \in B_H(3m/2)} \prob(\zeta \lra \cC \mid \fC_D(z) = \cC) \leq C m^{13/2 - d} \log m\ , 
                       \end{equation}
                       where we have used the fact that $z \in \Xi_{D,Q2}^{\mathrm{EREG}}(x).$
                       
                       To bound \eqref{eq:stripdelta} for $\zeta \notin B_H(3m/2)$, we perform the $y$ sum using the following replacement for \eqref{eq:constlog}:
                       \[ \max_{\zeta \in \Zd_+ \setminus B_H(3m/2)} \sum_{y \in Q_3} \|\zeta - y\|^{1-d} \leq C \ . \]
                       The remaining sum can be dealt with  by decomposing based on $\|\zeta - z^*\|$. This leads to the sequence of bounds
                       \begin{align}
                      &\sum_{z^* \in \Delta(z)} \sum_{\zeta \notin \cC}\prob(\zeta \lra \cC\mid \fC_{D_2}(z)=\cC) \prob(\zeta \lra z^*)\label{eq:scalede}\\
                       &\leq C \sum_{z^* \in \Delta(z)}\sum_{\ell \geq\log_2 K/2}^\infty \E[\fC(z) \cap B(z^*; 2^{\ell})\mid \fC_{D}(z)=\mathcal{C}] 2^{\ell(2-d)}\\
                       &\quad + C \sum_{z^* \in \Delta(z)} \sum_{\zeta \in B(z^*;K)} \prob(\zeta \lra \cC\mid \fC_{D}(z)= \mathcal{C}) \|\zeta - z^*\|^{2-d}\nonumber\\
                       &\leq C K^d\sum_{\ell \geq \log_2 K/2} 2^{(13/2-d)\ell} + K^2 \E[\fC(z) \cap B(z, 4K) \mid \fC_D(z) = \cC]\leq C K^{13/2} \ .\nonumber
                       \end{align}
                       Applying this and \eqref{eq:zetaclose} in \eqref{eq:stripdelta}, we produce an upper bound of the form
                       \[\text{for $A = Q_3$, } \eqref{eq:toe3it} \leq C K^{13/2-d}\ . \]
                      We compare this to \eqref{eq:e2first}, noting that the sum of the right-hand side of that equation is bounded below by $c \prob(\cE_1(z))$. We see there is some $K_1$ large such that for each $K > K_1$, there is a $c = c(K)$ with
                       \[ |\Delta(z)|^{-1}\sum_{z^* \in \Delta(z)}\sum_{y \in A} \prob(\cE_3(z,z^*) \mid \cE_1(z) \cap \cE_2(z,x^*,y)) \geq c\ ,
                       \]
                       and \eqref{eq:e3first} follows for $A = Q_3$.
					
				    \paragraph{Case $A = B(m\mathbf{e}_1;K)$.} We decompose the sum of \eqref{eq:totwoterms} into two sums, one over $\zeta \in B(m\mathbf{e}_1; m/8)$ and the other over the remaining values of $\zeta$. The first sum is slightly more complicated (involving the more stringent regularity notion of $\mathrm{EREG}$), so we treat it in detail. We write, performing first the sum over $z^*$:
					\begin{align}
					&\sum_{z^* \in \Delta(z)} \sum_{\zeta \in B(m\mathbf{e}_1; m/8)}\sum_{y \in B(m\mathbf{e}_1;K)} \prob(\zeta \lra \cC \mid \fC_{D}(z)=\mathcal{C}) \prob(\zeta \lra z^*) \prob(\zeta \sa{\Zd_+} y)\nonumber\\
					\leq &C m^{2-d}|\Delta(z)| \sum_{\zeta \in B(m\mathbf{e}_1; m/8)}\sum_{y \in B(m\mathbf{e}_1;K)} \prob(\zeta \lra \cC \mid \fC_{D}(z)=\mathcal{C}) \prob(\zeta \sa{\Zd_+} y).\label{eq:breakthek}
					\end{align}
					We now further decompose the sum over $\zeta$ in \eqref{eq:breakthek} into terms with $\zeta \in B(m\mathbf{e}_1; 2K)$ and $\zeta \notin B(m\mathbf{e}_1; 2K)$. For the former case, we bound
					\begin{align}
					\sum_{\zeta \in B(m\mathbf{e}_1; 2K)}\sum_{y \in B(m\mathbf{e}_1;K)} \prob(\zeta \lra \cC \mid \fC_{D}(z)=\mathcal{C}) \prob(\zeta \sa{\Zd_+} y) &\leq CK^2 \E[|\fC(z)\cap B(m
					\mathbf{e}_1;2K)| \mid \fC_D(z) = \cC
					 ]\nonumber\\
					 &\leq C K^{13/2}\ ,\label{eq:zetanear1}
					\end{align}
					where we have used the fact that $z \in \Xi_{D,Q_2}^{\mathrm{EREG}}(x)$ in the last line. To bound \eqref{eq:breakthek} when $\zeta \notin B(m\mathbf{e}_1; 2K)$, we decompose based on scale as in the bounds at \eqref{eq:scalede}, arriving as before at the bound
					\begin{align}
					\sum_{\zeta \notin B(m\mathbf{e}_1; 2K)}\sum_{y \in B(m\mathbf{e}_1;K)} \prob(\zeta \lra \cC \mid \fC_{D}(z)=\mathcal{C}) \prob(\zeta \sa{\Zd_+} y) &\leq  K^{13/2}\ .\label{eq:zetanear2}
					\end{align}
				
				The bounds \eqref{eq:zetanear1} and \eqref{eq:zetanear2} together show that 
				\begin{equation}\label{eq:almostallzeta} \eqref{eq:breakthek} \leq C K^{13/2+d} m^{2-d}\ , \end{equation}
				 and this controls the terms of \eqref{eq:totwoterms} involving $\zeta \in B(m\mathbf{e}_1; m/8)$. The contribution to \eqref{eq:totwoterms} from $\zeta \notin B(m \mathbf{e}_1;m/8)$ can be controlled in a similar but simpler way; a main difference is that instead of uniformly bounding $\prob(\zeta \lra z^* )$ as in \eqref{eq:breakthek}, we can instead bound $\prob(\zeta \lra y)$. 
				 
				 We arrive at the bound
				 \[\text{when $A = B(m\mathbf{e}_1;K)$, } \eqref{eq:toe3it} \leq C m^{2-d} K^{13/2}\ .  \]
				 For comparison, summing \eqref{eq:e2first} over $y \in B(m\mathbf{e}_1;K)$ and using the fact that $\prob(z^* \sa{\Zd_+} y) \geq c m^{2-d}$ uniformly in $z^* \in \Delta(z)$ and $y \in B(m\mathbf{e}_1;K)$ gives
				 \[	|\Delta(z^*)|^{-1} \sum_{y \in A}\sum_{z^* \in \Delta(z)} \prob(\cE_2(z,z^*,y)  \mid \fC_{D}(z) = \cC) \geq c m^{2-d} K^d\ . \]
				 
				Comparing the last two displays and recalling the uniform bound $|\Delta(z)| \geq c K^d$ completes the proof of \eqref{eq:e3smallbox} and the lemma.
						\end{proof}
		It now remains to use the above lemma to lower-bound $Y$ and complete the proof of Proposition \ref{prop:yprop}. As in \eqref{eq:topivglue}, we write
		\[\sum_{y \in A} \E[Y(y); X_{D,Q_2}^{\mathrm{EREG}} = N] = \sum_{y \in A}\sum_{z \in Q_2} \prob(z \in Y(y), X_{D,Q_2}^{\mathrm{EREG}} = N). \]
		To lower-bound the right-hand side of the above, we use a crucial fact: fixing $K > K_1$ as in Lemma \ref{lem:Estack}, there is a uniform constant $c = c(K)$ such that
			\begin{equation}
			\label{eq:edgemodpiv}\prob(z \in Y(y), X_{D,Q_2}^{\mathrm{EREG}} = N) \geq c \prob(\cE_1(z) \cap \cE_2(z,z^*,y) \cap \cE_3(z,z^*))\end{equation}
			uniformly in $m$, $x$, $y$, $z$, and $z^*$ as in Lemma \ref{lem:Estack}.
		This is a standard edge modification argument (see \cite[Lemma 5.1]{KN11} or the argument in Step 5 of the proof of Lemma \ref{lem:toinduct} below), so we do not give a full proof. In outline: one must open a path with length of order $K$ from $z$ to $z^*$ lying in $\Zd_+ \setminus D$, thereby ensuring that $z$ is connected to $y$, while potentially closing some edges to ensure that the edge $\{z, z'\}$ is pivotal as the definition of $Y(y)$.
		
		Applying \eqref{eq:edgemodpiv}, we see that
		\begin{align*}
		\sum_{y \in B(m\mathbf{e}_1;K)} \E[Y(y); X_{D,Q_2}^{\mathrm{EREG}} = N] &\geq c \sum_{z \in Q_2}\sum_{y \in B(m\mathbf{e}_1;K)}  \prob(\cE_1(z) \cap \cE_2(z,z^*,y) \cap \cE_3(z,z^*))\\
		(\text{by Lemma \ref{lem:Estack}})	&\geq c m^{2-d} \sum_{z \in Q_2} \prob(\cE_1(z))\\
		&\geq c N m ^{2-d}\prob(X_{D,Q_2}^{\mathrm{EREG}} = N)\ .
		\end{align*}
		This proves Proposition \ref{prop:yprop} for the case of $y \in B(m\mathbf{e}_1;K)$. A similar calculation to the previous display establishes the case of $y \in Q_3$, completing the proof of the proposition.
	\end{proof}

	
	\subsection{Two-point function asymptotics}
	In this section, we state and prove asymptotics for $\tau_{\Zd_+}(x,m\mathbf{e}_1)$, completing the proof of Theorem \ref{thm:scalingub}. The proofs build on the estimates obtained in the previous sections. We first prove asymptotics for $\E [X_{D, Q_1}^{EREG}]$ and $\E [X_{D, Q_2}^{EREG}]$. Since an open path from $n \mathbf{e}_1$ to $m \mathbf{e}_1$ with $2m < n$ (for instance) must pass through $\partial B_H(2m)$, these asymptotics are related to those for $\tau_{\Zd_+}$ itself.
	\begin{cor}\label{cor:topreg}
		For each $K > K_1$, there exists a $c = c(K)$ such that the following holds uniformly in $m > 2K$, and in $x$ with $\|x\| > 4m$:
		\[\E[X_{D,Q_1}^{\mathrm{EREG}}] \geq c \E[X_{D,Q_1}] \geq c(m/\|x\|)^{d-1}\ . \]
	\end{cor}
	\begin{proof}
		We can write, using Lemma \ref{lem:topreg},
		\begin{align*}
		\E[X_{D,Q_1}^{\mathrm{EREG}}] &= \sum_{z \in Q_1} \prob(z \in \mathrm{EREG}_{D}(\{0,m\mathbf{e}_1\},K) \mid z \sa{D} x)\prob(z \sa{D} x)\\
		&\geq c  \sum_{z \in Q_1}  \prob(z \sa{D} x) = c \E[X_{D, Q_1}]\ .
		\end{align*}
		We now use the two-point function asymptotic \eqref{eqn: hs-tp} to complete the proof:
		\begin{align*}
		 \E[X_{D, Q_1}] = \sum_{z \in Q_1}  \prob(z \sa{D_1} x) \geq \sum_{z \in Q_1}  \prob(z \sa{2m \mathbf{e}_1 + \Zd_+} x) \geq c (m/\|x\|)^{d-1}\ .
		\end{align*}
	\end{proof}

The next lemma provides an upper bound on the quantity $\E X_{D,Q_2}^{\mathrm{EREG}}$ (itself an upper bound for $\E X_{D,Q_2}^{\mathrm{EREG}}$) which matches that of Corollary \ref{cor:topreg} up to a constant factor.
	\begin{lem}\label{lem:regd2upper}
	 For each $K > K_1$, there exists a $c = c(K)$ such that the following holds uniformly in $m > 2K$, and in $x$ with $\|x\| > 4m$:
	\[C^{-1} \E[X_{D, Q_2}] \leq \E [X^{\mathrm{EREG}}_{D, Q_2}] \leq C (m/\|x\|)^{d-1}\ .\]
\end{lem}	
\begin{proof}
	The key ingredient of the proof is Proposition \ref{prop:competingglue}, and so we use the notation of that proposition. Indeed, fixing a $K$ large enough and then summing the bound of the proposition, we find
	\[\sum_{y \in Q_3}\E[Y(y)] \geq c \E[X_{D,Q_2}^{\mathrm{EREG}}]\ , \]
	uniformly in $x$ and $m$. 
	On the other hand, as observed in Proposition \ref{prop:competingglue}, the left-hand side of the above is at most 
	\[\sum_{y \in Q_3} \prob(x \sa{\Zd_+} y) \leq C m^{d-1} \|x\|^{1-d}\ , \]
	where in the last inequality we used the two-point function bound \eqref{eqn: hs-tp}. 
	
	This completes the proof of the second inequality. The first follows using Lemma \ref{lem:topreg} as in the proof of Corollary \ref{cor:topreg}.
\end{proof}

We are now equipped to prove the asymptotics for the two-point function in $\Zd_+$.
\begin{proof}[Proof of Theorem \ref{thm:scalingub}]
					We prove the upper bound first. It is helpful to introduce a domain $D^*$ consisting of $\Zd_+$ with a ``flattened version'' of  $B_H(4m)$ removed:
					\[ D^* :=  \Zd_+ \setminus \big([0,2m] \times [-4m,4m]^{d-1}\big); \quad Q_4 := \partial_{\Zd_+} (\Zd_+ \setminus D^*)\ . \] 
					If $x \sa{\Zd_+} m \mathbf{e}_1$, then there exists a $z \in Q_4$ such that
					\[\{x \sa{D^*} z\} \circ \{z \lra m \mathbf{e}_1\}\ . \]
					Using the BK inequality, then:
					\begin{align}
					\prob(x \sa{\Zd_+} m\mathbf{e}_1)&\leq \sum_{z \in Q_4} \prob(x \sa{D^*} z) \prob(z \lra m \mathbf{e}_1)\nonumber\\
					&\leq C m^{2-d} \sum_{z \in Q_4} \prob(z \sa{D^*} x) \leq C m^{2-d} \sum_{z \in Q_4} \prob(z \sa{-2m\mathbf{e}_1 +[\Zd_+ \setminus B_H(4m)]} x)\ .\label{eq:todilate}
					\end{align}
					The box $-2m\mathbf{e}_1 +[\Zd_+ \setminus B_H(4m)]$  is  a translate of the analogue of $D$ with $m$ replaced by $2m$. In particular, we can use Corollary \ref{prop:competingglue} to upper bound the quantity in the last display:
					\[\eqref{eq:todilate} \leq C \|x\|^{1-d} m^{d-1} \times m^{2-d} \]
					and the upper bound of the theorem follows.

					We turn to the lower bound on $\tau_H$. As in the previous part, we build our connection from $x$ to $m\mathbf{e}_1$ by first connecting $x$ to the boundary of a box and then extending. By Corollary \ref{cor:topreg}, we can choose a large constant $K$ so that					
					\[ \E[ X_{D, Q_1}^{\mathrm{EREG}}] \geq c (m/\|x\|)^{d-1} \quad \text{ uniformly in $x, m$ as claimed in Theorem \ref{thm:scalingub}}.\]
					Applying the bound of Proposition \ref{prop:yprop} and summing over $N$ gives 
					\[\sum_{y \in B(m\mathbf{e}_1; K)}\E[Y(y)] \geq c m \|x\|^{1-d}\ .\]
					Using Proposition \ref{prop:competingglue}, this implies
					\[\text{for $x$, $m$ as above, there exists $y \in B(m\mathbf{e}_1;K)$ such that $\tau_H(x,y) \geq c m \|x\|^{1-d}$}. \]
					
					With $x$, $m$, and $y$ as in the last display, we can write
					\[\tau_H(x, m \mathbf{e}_1) \geq \prob(x \sa{\Zd_+}y, y \sa{\Zd_+} m\mathbf{e}_1) \geq c m \|x\|^{1-d}\ , \]
					by the previous display, the FKG inequality \eqref{eqn: FKG}, and the fact that $\|y - m\mathbf{e}_1\| \leq K$. The theorem follows.
			\end{proof}

\section{Lower bounds for the chemical distance and cluster size}\label{sec:lowertailconst}
In this section, we show the inequality \eqref{eq:chemsmall} of Theorem \ref{thm:chem dist bd} and the probability lower bound of Theorem \ref{thm:volann}. The main portion of the argument is Lemma \ref{lem:toinduct} below, where we lower-bound the probability of a sequence of events whose occurrence guarantees that the cluster of the origin is small but that the origin is connected to the boundary of a box by a sufficiently small-length path. We start with some definitions and preliminary estimates.

For a rectangle $D=\prod_{i=1}^d [a_i,b_i]$, we define its ``right boundary''
\[
\partial_R\left[\prod_{i=1}^d [a_i,b_i]\right]:= \left\{x \in D: \{x,y\} \text{ is an edge with }  y\cdot e_1 > b_1 \right\}.
\]
We will also use the notation 
\[\partial_W D= \partial D\setminus \partial_R D.\]
For positive integers $\alpha$, we also define
\begin{equation}\label{eqn: C-aspect}
\mathrm{Rect}^{(\alpha)}(n) =[-\alpha n,n]\times [-\alpha n,\alpha n]^{d-1},
\end{equation}
and the shifted version
\[\mathrm{Rect}^{(\alpha)}(x;n):=x+ \mathrm{Rect}^{(\alpha)}(n).\]
For notational simplicity, we introduce the convention that $\mathrm{Rect}^{(\alpha)}(n) = \varnothing$ when $n < 0$.
By \eqref{eq:onearmprob}, we have, for each $\varepsilon>0$, an $\alpha = \alpha(\varepsilon)>0$ large enough such that
\begin{equation}
    \label{eq:armtorect}
\mathbb{P}(0\sa{\mathrm{Rect}^{(\alpha)}(n)} \partial_W \mathrm{Rect}^{(\alpha)}(n))\le \varepsilon n^{-2}.
\end{equation}
We introduce some notation that is reminiscent of the definitions in Section \ref{sec:HSproof}, with some adaptations to the geometry in this section. Since the pertinent definitions from Section \ref{sec:HSproof} will not appear in this section, there is no risk of confusion. For an integer $n$, we define
 \begin{align*}
     \Xi_n(x)&:=\{y\in \partial_R \mathrm{Rect}^{(\alpha)}(x;n): y \sa{\mathrm{Rect}^{(\alpha)}(x:n )} x\},\\
     X_{n}(x)&:=|\Xi_{n}(x)|.
 \end{align*}
We denote
\[\Xi_{n}:=\Xi_{n}(0), \quad X_{n}:=X_{n}(0).\]
The above notation suppresses the dependence on $\alpha$ because we will fix a particular value of $\alpha$, to be denoted $\alpha^*$,  in Lemma \ref{lem:xshort}. We will use this $\alpha^*$ for the rest of this section. Once we fix $\alpha^*$, we will further abbreviate $\mathrm{Rect}^{(\alpha^*)}(n)$ by $\mathrm{Rect}(n)$, with a similar abbreviation for $\mathrm{Rect}^{(\alpha^*)}(x;n)$.

We now fix an integer $m\ge 4$ and $\ell \ge 1$.
\begin{defin}
The random set $\mathrm{SREG}(x;\ell,m,K)$ consists  of all $y\in \partial \mathrm{Rect}^{(\alpha)}(x;\ell m)$ such that
\[\mathbb{E}[|\fC(y)\cap B(y;r)\setminus \mathrm{Rect}^{(\alpha)}(x;(\ell-1/2)m)|\mid \fC_{\mathrm{Rect^{(\alpha)}}(x;\ell m)}(y)]<r^{\frac{9}{2}}\]
for all $r\ge K$.
\end{defin}
When $x=0$, we omit it from the notation. See Figure \ref{fig:tunnel} for a schematic depiction. We write (again omitting the argument when $ x = 0$)
\begin{align}
    \Xi^{\mathrm{SREG}}_{\ell, m}(x)&:= \Xi_{\ell m}(x) \cap \mathrm{SREG}(x;\ell,m, K)\\
    X^{\mathrm{SREG}}_{\ell, m}(x)&:= |\Xi_{\ell, m}^{\mathrm{SREG}}(x)|.
\end{align}
We also introduce a version of $\Xi_{\ell m}$ restricted to vertices connected to $x$ through ``short paths". Let $\rho>0$ and define
\[\Xi_{\ell, m}^{\rho\text{-short}}(x)=\Xi_{\ell, m}^{\mathrm{SREG}}(x)\cap \{y\in \partial_R \mathrm{Rect}^\alpha(x;\ell m): d_{chem}^{\mathrm{Rect}^{(\alpha)}(x; \ell m)}(x,y)\le \rho \ell m^2\}.\]
Similarly, we write $X_{\ell, m}^{\rho\text{-short}}(x) = |\Xi_{\ell, m}^{\rho\text{-short}}(x)|$.
\subsection{Estimates}
We first obtain a lower bound on the quantity $\Xi_{\ell m}$. The following is Lemma \ref{lem:easyreg} with minor modifications for this context.
\begin{lem} \label{lem:SREG}
There are constants $n_0,\, c,\, C>0$ such that, uniformly in $n\ge n_0$, in $k \geq 1$, and in $\lambda\ge 1$, we have
	\[
	\prob\left(\left|\fC(y) \cap B(y;k) \right| > \lambda k^{4} \log^5 (k) \mid 0 \sa{\mathrm{Rect}(\ell m)} y\right)\leq C \exp(-c \sqrt{\lambda} \log^3 k)\ .\]
	Thus, as in Lemma \ref{lem:easyreg},  there exists a $K_0 > 0$ such that uniformly in $\ell\ge 1$ and $m\ge m_0$, for all $K > K_0$: 
	\begin{equation}
	\label{eq:easyreg2}
	    \prob(y \notin \mathrm{SREG}(\ell,m,K) \mid 0 \sa{\mathrm{Rect}(\ell m)} y) \leq C \exp(-c K^{1/4})\ .	    
	    \end{equation}
\end{lem}

Applying Lemma \ref{lem:SREG} and \eqref{eqn: hs-tp}, we see
\begin{align*}
    \mathbb{E}[|\Xi_{\ell m}\setminus \Xi_{\ell, m}^{\mathrm{SREG}}|]&=\sum_{y\in \partial_R \mathrm{Rect}(\ell m)}\mathbb{P}(y \notin  \mathrm{SREG}(\ell,m,K) \mid 0\sa{\mathrm{Rect}(\ell m)} y)\mathbb{P}(0\sa{\mathrm{Rect}(\ell m)} y)\\
    &\le \exp(-cK^{1/4})(\alpha \ell m)^{d-1}(\ell m)^{-d+1}\le C\alpha^{d-1}\exp(-cK^{1/4}).
\end{align*}
Thus by Markov's inequality, we have, for each $\delta > 0$,
\begin{equation}\label{eqn: apply-markov}
\mathbb{P}(|\Xi_{\ell m}\setminus \Xi_{\ell, m}^{\mathrm{SREG}}|\ge \delta n^2)\le C\delta^{-1}(\ell m)^{-2}\alpha^{d-1}\exp(-cK^{1/4}).
\end{equation}

The following lemma will serve as the base case in an induction appearing in Lemma \ref{lem:toinduct}.
\begin{figure}
    \centering
    (A)\includegraphics[width=12cm, height=8cm]{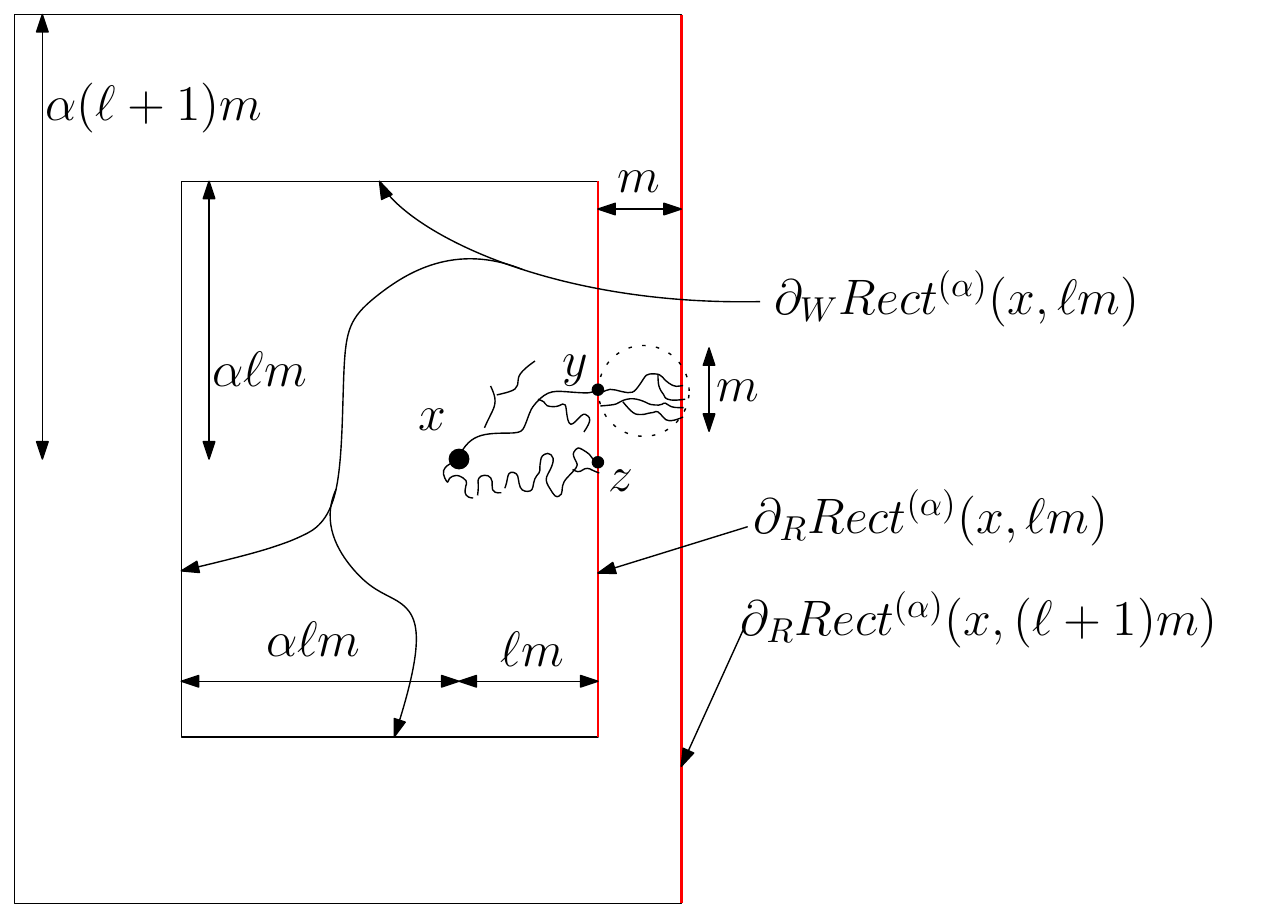}
    (B)\includegraphics[width=4cm, height=5cm]{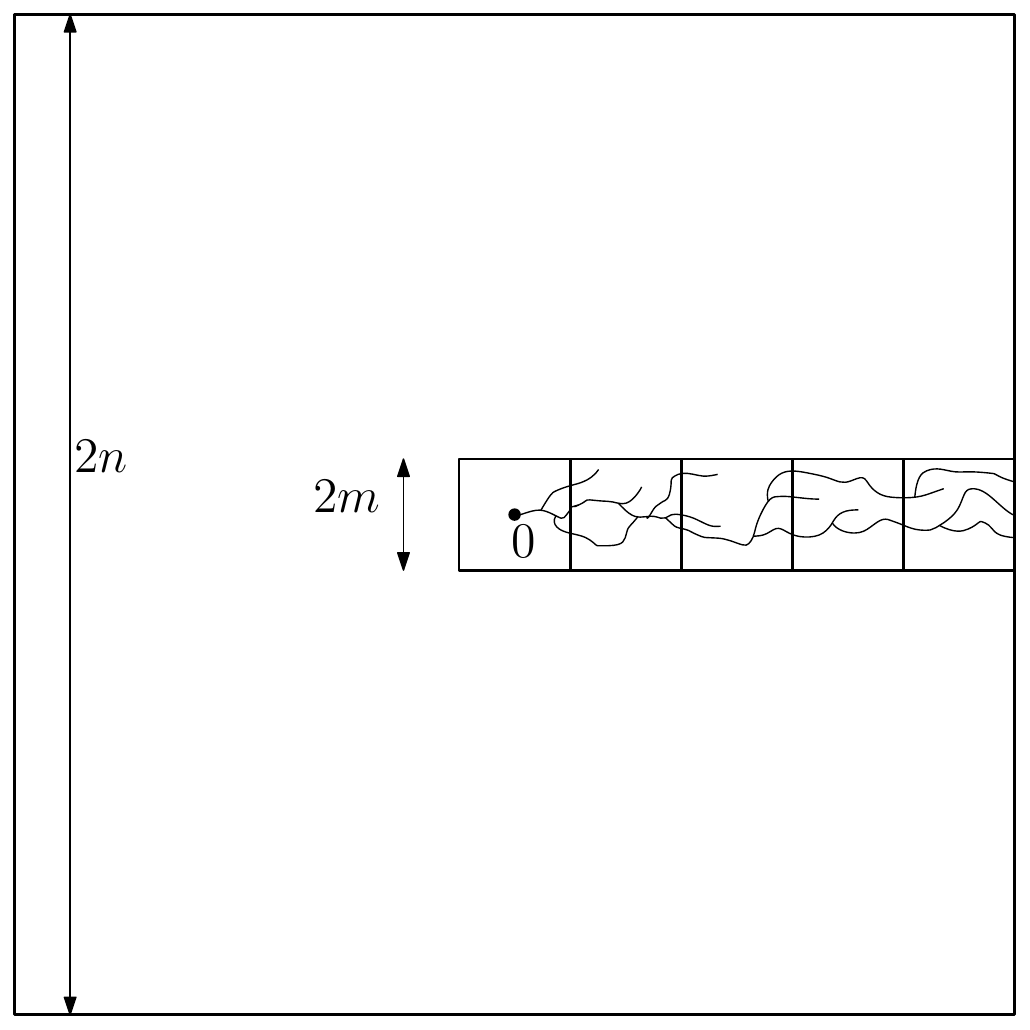}
    \caption{(A) Schematic representation of $Rect^{(\alpha)}(x,\ell m), Rect^{(\alpha)}(x,(\ell+1) m)$, and $\Xi_{\ell m}(x)$. For a typical regular boundary vertex $y\in\Xi_{\ell, m}^{\mathrm{ SREG}}(x)$ of $Rect^{(\alpha)}(x,\ell m)$, the volume of the extended cluster (encircled region)  within $B(y,m)$  and the chemical distance between $y$ and $\partial B(y,m)$ within this scales as $O(m^4)$ and $O(m^2)$ respectively. $z\in\Xi_{\ell, m}^{\rho-short}(x)$ if the chemical distance $d_{chem}^{Rect^{(\alpha)}(x,\ell m)}(x,z)\leq \rho\ell m^2$. (B) Schematic representation of the kind of cluster that suffices for the inductive lower bound argument to work.}
    \label{fig:tunnel}
\end{figure}

\begin{lem}\label{lem:xshort}
For each choice of $\alpha \geq 1$ from \eqref{eqn: C-aspect}, there is a constant  $c>0$ and large constants $1\le \rho <\infty$ and $K_1 \geq K_0$ depending only on $\alpha$ and the dimension $d$ such that, if $K\ge K_1$,
 \begin{equation}\label{eq:firstxshort}
     \mathbb{P}(X_{1,m}^{\rho\text{-short}}\ge cm^2)\ge cm^{-2}
 \end{equation}
 for all $m\ge m_0$. In particular, there is some choice of integer $\alpha$, henceforth denoted by $\alpha^*$, and some $K_1=K_1(\alpha^*) > K_0$ such that for some $c_{\partial}$, $C_\partial<\infty$, we have
 \begin{equation}
     \label{eq:secondxshort}
\mathbb{P}\left(\left\{ C_\partial m^2 > X_{m}(0) \geq  X_{1,m}^{\rho\text{-short}}(0)\ge c_\partial m^2\right\} \setminus \left\{ 0 \sa{\mathrm{Rect}(m)} \partial_W \mathrm{Rect}(m)\right\}\right)\ge cm^{-2}
 \end{equation}
 for all $K\geq K_1$ and $m\ge 1$.
 
 
 \end{lem}
 \begin{proof}
	We first recall the bound \eqref{eqn: X-low}, which implies
	\begin{equation*}
	\text{uniformly in $n \geq 1$, \quad }\prob(X_n \geq c_1 n^2) \geq c_1 n^{-2}
	\end{equation*}
	for some uniform $c_1 > 0$ independent of $\alpha$ as long as $\alpha \geq 1$. Now, using \eqref{eq:armtorect}, we can find a $\alpha^*$ large and a constant $c_2 > 0$ uniform in $n$ such that
    \begin{equation}
        \label{eq:Xnmustbe2}
    \text{with $\alpha = \alpha^*$,}  \quad  	\prob(\{X_n \geq c_2 n^2\}\setminus \{0 \sa{\mathrm{Rect}^{(\alpha^*)}(n)} \partial_W \mathrm{Rect}^{(\alpha^*)}(n) \}) \geq c_2 n^{-2}\ .
    \end{equation}
	We henceforth fix $\alpha^*$ as in \eqref{eq:Xnmustbe2}.
	
	Using Markov's inequality as in \eqref{eqn: apply-markov}, we can choose $K_1=K_1(\alpha^*) > K_0$ such that, for $K \geq K_1$ and for all $m$,
	\begin{equation}
	    \label{eq:toSREG}
	\mathbb{P}(X_m- X_{1,m}^{\mathrm{SREG}}\ge c_2m^2/4)\le c_2m^{-2}/4.
	\end{equation}
	
	We estimate the expected number of edges on a path from $0$ to a vertex $y \in \Xi_m$. Let $M(0,y;m)$ denote the number of edges on the shortest open path from $0$ to $y$ in $\mathrm{Rect}(m)$, with the convention that $M(0,y;m) = 0$ when there is no such path. We have
	\begin{align}
	\E\left[M(0,y;m) \right] &\leq 2d \sum_{z \in \mathrm{Rect}(m)} \prob(\{0 \sa{\mathrm{Rect}(m)} z\} \circ \{z \sa{\mathrm{Rect}(m)} y \})\nonumber\\
	&\leq  2d \sum_{z \in \mathrm{Rect}(m)} \prob(0 \sa{m \mathbf{e}_1 - \Zd_+} z) \prob(z \sa{m \mathbf{e}_1 - \Zd_+} y)\nonumber\\
	&\leq C_1 m^{3-d}\ , \label{eq:longbd}
	\end{align}
	where we have used the two-point function asymptotic of Theorem \ref{thm:scalingub}.
	
	For each $\rho > 0$, with $c_2$ as in \eqref{eq:Xnmustbe2},
	\[\text{on the event $\{ X_{1,m}^{\mathrm{SREG}} - X_{1,m}^{\rho\text{-short}} \geq c_2 m^2 / 2 \}$, we have }\sum_{y \in \partial_R\mathrm{Rect}(m)}M(0,y;m) \geq c_2 \rho m^4 / 2\ ; \]
	the constant $c_2$ in this display is independent of $\rho$. Taking expectations, we find
	\[\E\left[\sum_{y \in \partial_R \mathrm{Rect}(m)}M(0,y;m)\right] \geq c_2 \rho m^4/2\, \prob\left(  X_{1,m}^{\mathrm{SREG}} - X_{1,m}^{\rho\text{-short}} \geq c_2 m^2 / 2 \right)\ .  \]
	Contrasting the last display with \eqref{eq:longbd}, we see that we can make a choice of $\rho$ independent of $m$ such that
	\begin{equation}
	    \label{eq:xtoshort}
	    \prob\left(  X_{1,m}^{\mathrm{SREG}} - X_{1,m}^{\rho\text{-short}} \geq c_2 m^2 / 2 \right) \leq c_2 m^{-2} / 2\ .
	\end{equation} 
	Finally, using \eqref{eq:xtoshort} in conjunction with \eqref{eq:toSREG}, we find (with $\rho$ as in \eqref{eq:xtoshort})
	\begin{equation}
	    \label{eq:xtoshort2}
	     \prob\left(  X_m - X_{1,m}^{\rho\text{-short}} \geq 3c_2 m^2 / 4 \right) \leq 3c_2 m^{-2} / 4\ .
	\end{equation}

	Comparing \eqref{eq:xtoshort2} with \eqref{eq:Xnmustbe2} completes the proof of \eqref{eq:firstxshort} and an analogue of  \eqref{eq:secondxshort} where we do not demand $X_m(0) \leq C_\partial m^2$. To impose this condition, we note that
	\[\E[X_m(0)] \leq \sum_{x \in \partial_R Rect(m)} \prob(0 \sa{m \mathbf{e}_1 - \Zd_+} x) \leq C\ , \]
	and we apply Markov's inequality to see $\prob(X_m \geq C_\partial m^2) \leq c_2 m^2/  8$ for sufficiently large $C_\partial$.
	This completes the proof of \eqref{eq:secondxshort}, concluding the proof of the lemma.
	\end{proof}

\begin{lem}\label{lem:toinduct}
	Let $\rho, \, C_{\partial}, \,c_{\partial}$ be as in the statement of Lemma \ref{lem:xshort}. There exist constants $C_{vol} < \infty$ and $m_1 > m_0$ such that the following holds. Defining, for each pair of integers $\ell \ge 1$ and $m \geq m_1$, the event
	\begin{equation}
	    G(\ell,m) := A(\ell,m)\cap B(\ell,m), 
	\end{equation}
	where
	\begin{align}
	    A(\ell,m)&=\left\{C_{\partial} m^2 > X_{\ell m} \geq X_{\ell, m}^{ 2\rho \text{-short}} \geq \frac{c_{\partial} m^2}{2}\right\}
	\setminus \left\{0 {\sa{\mathrm{Rect}(\ell m)}} \partial_W \mathrm{Rect}(\ell m)  \right\}\label{eq:Gdeff1}\\
	&B(\ell,m)=\left\{|\fC_{\mathrm{Rect}(\ell m)}(0)| < C_{vol} \ell m^4 \right\}\nonumber \\
	 &\cap\left\{\text{for each $0 \leq i \leq \alpha$, } |\fC_{\mathrm{Rect}(\ell m)}(0)\cap \mathrm{Rect}((\ell-i) m) \setminus \mathrm{Rect}((\ell-i-1)m)| < C_{vol} i m^4
	 \right\}\label{eq:Gdeff2},
	\end{align}
	then  we have $\prob(G(\ell, m)) \geq c^{\ell + 1} m^{-2}$ for a constant $c$ uniform in $\ell \geq 1$ and $m \geq m_1$.
	\end{lem}
	We comment briefly on the definition of $B(\ell, m)$. The first event appearing in the intersection in its definition is in some sense the operative one: it bounds the size of $C_{\mathrm{Rect}(\ell m)}(0)$, which is our main goal. The second event appears for technical reasons, essentially serving as an accessory to regularity. See \eqref{eq:twoscales} and the following for how this condition is applied, and see the end of Step 5 below for discussion of why we did not try to impose a version of this condition as part of the definition of $\mathrm{SREG}.$
	
\begin{proof}
	The proof is by induction on $\ell$ for fixed $m$. The base case $\ell = 1$ is almost furnished by Lemma \ref{lem:xshort}; all that remains to prove is that the bound on $|\fC_{\mathrm{Rect}(m)}(0)|$ in \eqref{eq:Gdeff2} can be imposed without changing the order of the probability bound in that lemma. To do this, we simply apply a moment bound. Indeed,
	\[\E[|\fC_{\mathrm{Rect}(m)}(0)|] \leq \E[|\fC(0) \cap \mathrm{Rect}(m)|] = \sum_{x \in \mathrm{Rect}(m)} \tau(0,x) \leq C m^2\ .\]
	Applying Markov's inequality and a union bound shows the claim of the lemma for $\ell = 1$, for all sufficiently large values of $C_{vol}$.
	
	We now prove the inductive step. We write
	\begin{align}
	\prob(G(\ell+1, m)) &\geq \prob(G(\ell+1,m) \cap G(\ell, m))\nonumber \\
	&= \sum_{\cC} \prob(G(\ell+1,m) \mid \fC_{\mathrm{Rect}(\ell m)}(0) = \cC)\label{eq:inductclust} \prob(\fC_{\mathrm{Rect}(\ell m)} = \cC),
	\end{align}
	where in \eqref{eq:inductclust} the sum is over realizations $\cC$ of $\fC_{\mathrm{Rect}(\ell m)}$ such that $G(\ell, m)$ occurs (this event being measurable with respect to $\fC_{\mathrm{Rect}(\ell m)}$). Similarly, the sets $\Xi_{\ell m}$, $\Xi_{\ell,  m}^{\rho\text{-short}}$, and their cardinalities are functions of $\fC_{\mathrm{Rect}(\ell m)}(0)$; we write (for instance) $X_{\ell m}(\cC)$ to denote the (deterministic) value of $X_{\ell m}$ that obtains when $\fC_{\mathrm{Rect}(\ell m)}(0) = \cC$.
	
	The remainder of the proof will provide a uniform lower bound on the conditional probability appearing in \eqref{eq:inductclust}. We do this by successive conditioning, bounding the probability cost as we impose the conditions of $G(\ell+1, m)$. For clarity of presentation, we organize this into steps.  In what follows, $\cC$ will be a fixed but arbitrary value of $\fC_{\mathrm{Rect}(\ell m)}(0)$ appearing in \eqref{eq:inductclust}. Before starting the first step of the proof, we make some definitions to allow us to notate events occurring off of $\cC$ more easily.

	 \begin{defin}
	 \begin{itemize}
	    \item 	 $\widetilde \Zd \subseteq \Zd$ is the vertex set $[\Zd \setminus \cC] \cup \Xi_{\ell, m}^{2\rho\text{-short}}$. With some abuse of notation, we use the same symbol for $\widetilde \Zd$ and the graph with vertex set $\widetilde \Zd$ and with edge set $\mathcal{E}(\widetilde \Zd)$ defined by
	\[\{\{x, y\} \in \mathcal{E}(\Zd): \,x \in \Xi_{\ell m}, y \in \Zd \setminus \mathrm{Rect}(\ell m) \} \cup \{\{x, y\} \in \mathcal{E}(\Zd): \,x, \, y \in \Zd \setminus \cC \}\ . \]
	     \item We denote the conditional percolation measure $\mathbb{P}(\, \cdot \, \mid \fC_{\mathrm{Rect}(\ell m)} = \cC)$ on $\widetilde \Omega := \{0,1\}^{\mathcal{E}(\widetilde \Zd)}$ by $\widetilde \prob(\cdot)$. Similarly, we write $\widetilde d_{chem}$ for the chemical distance on the open subgraph of $\widetilde \Zd$.
	 \end{itemize}
	 \end{defin}
	  Conditional on $\{\fC_{\mathrm{Rect}(\ell m)} = \cC\}$, the distribution of $\omega_e$ for edges $e$ of $\widetilde \Zd$ is the same as their unconditional distribution: i.i.d.~Bernoulli$(p_c)$. Indeed, when $\cC$ is such that $G(\ell, m)$ occurs, $\mathcal{E}(\widetilde \Zd)$ is exactly the set of edges in $\mathcal{E}(\Zd)$ which are not examined to determine $\fC_{\mathrm{Rect}(\ell m)}(0) = \cC$.  So the measure $\widetilde \prob$ is just a  projection of $\prob$ onto a subset of the edge variables of our original lattice.
	 
	We note that the restriction on $m$ appearing in the statement of the lemma will arise through the arguments below. Like in Section \ref{sec:HSproof}, we will need to introduce an auxiliary parameter $K$ which will be chosen large in order to make various error terms involving cluster intersections small. All bounds will be uniform as long as $m \geq m_0+ 4K$, and so the ultimate value of $m_1$ will be $m_0+4K$ for the choice of $K$ made at \eqref{eq:D3un}. We will also potentially need to enlarge the value of $C_{vol}$ below in Step 6, but not any other constants (and the value of $C_{vol}$ will be manifestly independent of $m$ and $\ell$).
	 \paragraph{Step 1.} In what follows, we let $K = 2^k \geq 1$ be a constant  larger than the $K_1$ from Lemma \ref{lem:xshort}, to be fixed shortly at \eqref{eq:D3un}.  For each $x \in \Xi_{\ell, m}^{2\rho\text{-short}}$, we define the following events on the space of edge variables on $\widetilde \Zd$.
	 
	 \begin{itemize}
	     \item $D_1(x)$ is the event that
	     \begin{enumerate}[a.]
	         \item $|\{y \in \partial_R \mathrm{Rect}(x; m)\cap  \Xi_{\ell+1,m}^{\mathrm{SREG}}: \widetilde d_{chem}(x,y) \leq 2\rho m^2 \}| \geq c_{\partial} m^2$,
	         \item $\left|\left\{y \in \partial \mathrm{Rect}((\ell+1) m): y \sa{\mathrm{Rect}((\ell+1) m) \setminus \cC} x \right\}\right| < C_{\partial} m^2,$
	         \item $\{x, x+\mathbf{e}_1\}$ is pivotal for $\Xi_{\ell m} \sa{\widetilde \Zd \cap \mathrm{Rect}((\ell+1) m)} \partial \mathrm{Rect}((\ell +1/2) m)$,
	         \item but we do not have $x \sa{\mathrm{Rect}((\ell+1)m)} \partial_W \mathrm{Rect}((\ell+1)m)$.
	     \end{enumerate}
	     \item  $D_1$ is the event $\cup_{x \in \Xi_{\ell m}^{\rho\text{-short}}} D_1(x)$.
	 \end{itemize}

	 We note that the conditional probability of the event
	 \begin{equation}
	 \label{eq:D1def2}
	 \left\{C_{\partial} m^2 > X_{(\ell+1) m} \geq X_{\ell+1, m}^{{ 2\rho \text{-short}}} \geq c_{\partial} m^2\right\}
	\setminus \left\{0 {\sa{\mathrm{Rect}((\ell+1) m)}} \partial_W \mathrm{Rect}((\ell+1) m)  \right\}
	 \end{equation}
	 conditioned on $\fC_{\mathrm{Rect}(\ell m)}= \cC$ is bounded below by $\widetilde \prob(D_1)$, and we turn to lower-bounding $\widetilde \prob(D_1)$.
	 
	 The pivotality in the definition of $D_1(x)$ guarantees that $D_1(x_1) \cap D_1(x_2) = \varnothing$ for $x_1 \neq x_2$; in particular,
	 \begin{equation}
	 \label{eq:unexpected}
	 \widetilde \prob(D_1) = \sum_{x \in \Xi_{\ell, m}^{2\rho\text{-short}}} \widetilde \prob(D_1(x))\ .
	 \end{equation}
	 In light of \eqref{eq:unexpected} and \eqref{eq:D1def2}, Steps 2--5 are devoted to establishing a uniform lower bound on $\widetilde \prob(D_1(x))$.

	 \paragraph{Step 2.} For each $x$ as in \eqref{eq:unexpected}, we set $x^* = x + K \mathbf{e}_1$. For use in this step, we introduce  notation for the analogues of $X_r$ and $X_{1,r}$ (for $r \geq 1$) when connections are forced not to intersect $\cC$. Namely,
	 \[\widetilde X_r(x^*) := |\{y \in \partial_R \mathrm{Rect}(x^*;r): \, y \sa{\mathrm{Rect}(x^*; r) \setminus \cC} x^* \}|\ , \]
	 with the analogous definition for $\widetilde X_{1,r}^{\rho\text{-short}}.$ Here we note that $K$ plays both the role of the shift of $x^*$ and the implicit regularity parameter for $\widetilde X_{1,r}^{\rho\text{-short}}.$ 
	 
	 We begin by arguing a probability lower bound for a modification of the event appearing in \eqref{eq:secondxshort} but centered at $x^*$:
	 \begin{equation}
	 \begin{split}
	 \label{eq:D2def}
	 D_2(x^*) := \left\{ C_\partial (m-K)^2 > \widetilde X_{m-K}(x^*) \geq  \widetilde X_{1,m-K}^{\rho\text{-short}}(x^*)\ge c_\partial (m-K)^2\right\}\\\setminus \left\{ x^* \sa{\mathrm{Rect}(x^*; m-K) \setminus \cC} \partial_W \mathrm{Rect}(x; m-K)\right\}
	 \end{split}
	 \end{equation}
	  Using a union bound, we find
	 \begin{equation}
	 \begin{split}
	\widetilde \prob\left(D_2(x^*) \right)\geq  &\prob\left( D_2(x^*)\right)\\
	   - &\prob\left(\begin{array}{c}
	 \exists z \in \cC \cap \mathrm{Rect}(x^*; m - K): \\ \{z \lra \partial_R \mathrm{Rect}(x; m-K)\}
	  \circ \{z \lra x^*\} \text{ occurs}\end{array}\right)\ .\label{eq:justoneY}
	 \end{split}
	 \end{equation}
	 
	  It follows that the second term in \eqref{eq:justoneY} is bounded by
	 \begin{align}
	 &\sum_{z \in  \cC \cap \mathrm{Rect}(x^*; m - K)} \mathbb{P}(z\lra \partial \mathrm{Rect}((\ell+1)m)) \prob(x^* \lra z)\nonumber\\
	 \leq & C m^{-2}\sum_{z \in  \cC \cap \mathrm{Rect}(x^*; m - K)} \prob(x^* \lra z)\ .\label{eq:twoscales}
	 \end{align}
	 The sum over $z$ in the last term can be further subdivided into the case that $z$ also lies in $\mathrm{Rect}((\ell-1/2)m)$ and the case that $z$ lies outside of $\mathrm{Rect}((\ell-1/2)m)$. In the latter case, we apply  the facts that $x \in \Xi_{\ell, m}^{\mathrm{SREG}}$ and that  $x^*$ is distance $K$ from $x$. In the former, we use the fact that in this regime $\prob(x^* \lra z) \leq C m^{2-d}$ and the fact that $B(\ell, m)$ occurs, which implies that the number of $z$ terms appearing in the sum is at most $C m^4$.
	 
	 Using these two bounds, we see
	 \begin{align*}
	 \eqref{eq:twoscales} \leq &~ C m^{-2} \left[ m^{6-d} +  \sum_{s = k}^\infty 2^{\frac{9}{2}s}  2^{(2-d)s} \right]\\ \leq &~ Cm^{-2} K^{\frac{13}{2}-d}.
	 \end{align*}
	 It remains to give a lower bound for the first term of \eqref{eq:justoneY}. Indeed, this is almost the content of Lemma \ref{lem:xshort} (specifically \eqref{eq:secondxshort}) with $m$ replaced by $m - K$, except for the appearance of the set $\cC$ in the portion of $D_2(x^*)$ involving connections to $\partial_W\mathrm{Rect}(x;m-K)$. This restriction only makes $\prob(D_2(x^*))$ higher than the probability appearing in \eqref{eq:secondxshort}. As long as $m \geq m_0 + 4K$, we can apply the bound of \eqref{eq:secondxshort} in \eqref{eq:justoneY}. We see there exists a $K_2 > K_1$ and a $c$ such that, for all $K > K_2$ and $m \geq m_0 + 4K$,
	 \begin{equation}
	 \label{eq:shortmod}\widetilde \prob\left(D_2(x^*)  \right) \geq c m^{-2} \quad \text{uniformly in $\ell, \cC, x$}\ .
	 \end{equation}
	 
     \paragraph{Step 3.} We now upgrade the above, demanding further that $x^*$ not be in the same cluster as any element of $\Xi_{\ell m}(\cC).$ We define
	 \[D_3(x^*):= D_2(x^*) \setminus \{\exists z \in \Xi_{\ell m}: \,  z \sa{\widetilde \Zd \cap \mathrm{Rect}((\ell + 1) m)} x^* \} \subseteq \widetilde \Omega\ . \]
	 We note for future reference that
	 \begin{equation}
	     \label{eq:noarm}
	     \begin{gathered}
	     \text{when $D_3(x^*)$ and $\{\fC_{\mathrm{Rect}(\ell m)}(0) = \cC\}$ occur, then we do \textbf{not} have}\\
	     x^* \sa{\mathrm{Rect}(x^*; m-K)} \partial_W \mathrm{Rect}((\ell + 1)m).
	     \end{gathered}
	 \end{equation}
	 This follows from \eqref{eq:D2def}, which ensures $x^*$ has no connection to $\partial_W \mathrm{Rect}((\ell+1) m)$ off $\cC$, and the definition of $D_3(x^*)$, which ensures $x^*$ has no connection to $\cC$.
	 
	We can lower bound the probability of $D_3(x^*)$ similarly to the argument establishing \eqref{eq:e3first} in the proof of Lemma 9:
	\begin{align}
	&\widetilde \prob(D_3(x^*)) \geq \widetilde \prob(D_2(x^*))\nonumber\\
	- &\sum_{y \in \mathrm{Rect}((\ell+1)m) \setminus \cC}  \prob\left(\left.\begin{array}{c}\{\Xi_{\ell m}\sa{\mathrm{Rect}((\ell+1)m)} y\} \circ \{y \lra x^* \}\\ \circ \{y \lra \partial_R \mathrm{Rect}((\ell +1) m)\}\end{array} \right| \fC_{\mathrm{Rect}(\ell m)}(0) = \cC\right)\ .\label{eq:twolayers}
	\end{align}
	
	We bound the sum in \eqref{eq:twolayers} by decomposing the sum into three terms: a) a term corresponding to $y \in \mathrm{Rect}((\ell-1/2)m)$, b) a term corresponding to $y \in \mathrm{Rect}((\ell+1/2)m) \setminus \mathrm{Rect}((\ell-1/2)m)$, and c) a term corresponding to $y \notin \mathrm{Rect}((\ell+1/2)m)$.
	In case a), we use the BK inequality to upper bound the sum by (letting $\ell m - y(1) = r$)
	\begin{equation}
	    \label{eq:circleback}
	    \begin{split}
	        C \pi(m/2) \times |\Xi_{\ell m}(\cC)| \times \sum_{r= m/2}^\infty r^{d-1} r^{4-2d} \leq C m^{4-d} = C m^{-2} (m^{6-d})\ .
	    \end{split}
	\end{equation}
    Case c) is similar to a) but slightly more complicated. We use Theorem \ref{thm:scalingub} to control the connection probability between $x^*$ and $y$, since $y$ is close to  $\partial_R Rect((\ell+1)m)$. We obtain the upper bound (letting $\max\{(\ell +1)m - y(1),1\}= r$)
    \begin{equation}\label{eq:circleback2}
        \begin{split}
            C |\Xi_{\ell m}(\cC)|  \sum_{r = 1}^{m/2-1} r^{d-1} \times (r m^{1-d})^2 \times r^{-2} \leq C m^{-2}(m^{6-d})\ .
        \end{split}
    \end{equation}

	Finally, the term corresponding to case b) can be bounded similarly to \eqref{eq:e3smallbox} using the BK inequality and the fact that $x \in \Xi_{\ell. m}^{2\rho\text{-short}}(\cC).$ We find, for $K > K_2$ and $m \geq m_0 + 4K$,
	\begin{equation*}
	\prob(D_3(x^*)) \geq c m^{-2} - C m^{-2} K^{13/2-d} \text{ uniformly in $\ell ,\cC, x$}\ .
	\end{equation*}
	Thus, there exists a $K_3 > K_2$ and a $c > 0$ such that, uniformly in $K \geq K_3$ and $m \geq m_0+4K,$
	\begin{equation}
	\label{eq:D3un}
	\prob(D_3(x^*)) \geq c m^{-2} \text{ uniformly in $\ell ,\cC,$ and $x$}\ .
	\end{equation}
	From here on, we fix $K = K_3$, and assume $m \geq m_1 = m_0 + 4K_3$.
	
	\paragraph{Step 4.} We define one final subevent of $D_3(x^*)$, imposing the additional restriction that no vertex of $\Xi_{\ell m}(\cC)$ have an arm to $\partial \mathrm{Rect}((\ell +1/2)m)$:
	\begin{equation}
	\label{eq:D4def}
	D_4(x,x^*) = D_3(x^*) \setminus \{ \exists z \in \Xi_{\ell m}: \, z \sa{\widetilde \Zd \cap \mathrm{Rect}((\ell + 1/2) m)} \partial \mathrm{Rect}((\ell +1/2) m)\}\ . 
		\end{equation}
	We lower-bound $\widetilde \prob(D_4(x,x^*)).$ To do this, we condition further on $\fC_{\widetilde \Zd \cap \mathrm{Rect}((\ell+1)m)}(x^*)$, noting that $D_3(x^*)$ is measurable with respect to the sigma-algebra on $\widetilde \Omega$ generated by this cluster:
	\begin{align}
\widetilde \prob(D_4(x,x^*)) = 	\sum_{\cC'}\widetilde \prob(D_4(x,x^*) \mid \fC_{\widetilde \Zd \cap \mathrm{Rect}((\ell+1)m)}(x^*) = \cC') \prob(\fC_{\widetilde \Zd \cap \mathrm{Rect}((\ell+1)m)}(x^*) = \cC')\ .
\label{eq:condthenbk}
	\end{align}
	
	On $D_3(x^*)$, we have $\Xi_{\ell m} \cap  \fC_{\widetilde \Zd \cap \mathrm{Rect}((\ell+1)m)}(x^*) = \varnothing$, and so the conditional probability in \eqref{eq:condthenbk} is bounded by
	\begin{align*}
	1 - \widetilde \prob(\exists z \in \Xi_{\ell m}: \, z \lra \partial \mathrm{Rect}((\ell +1/2) m) ) &=  \widetilde \prob\left( \forall z \in \Xi_{\ell m}: \, z \not \lra  \partial \mathrm{Rect}((\ell +1/2) m)  \right)\\
	\text{(by FKG)}\qquad &\geq \prod_{z \in \Xi_{\ell m}} \prob(z \not \lra \partial \mathrm{Rect}((\ell +1/2) m))\\
	 &\geq(1- c m^{-2})^{C m^2} \geq c\ .
	\end{align*}
	In the second line, in addition to the FKG inequality, we used the fact that conditioning on $\fC_{\mathrm{Rect}(\ell m)} = \cC$ can only decrease the probability that $\Xi_{\ell m}(\cC)$ is connected to $\partial \mathrm{Rect}((\ell+1/2)m)$.
	 The above bound is uniform in $\cC'$, so reinserting into \eqref{eq:condthenbk}, we find 
	 \begin{equation}
	 \label{eq:D4un}
	 \prob(D_4(x,x^*)) \geq c m^{-2} \text{ uniformly in $m \geq m_1$ and $\ell ,\cC, x$}\ .
	 \end{equation}
    
     \paragraph{Step 5.} We now turn \eqref{eq:D4un} into the estimate
	 \begin{equation}
	 \label{eq:shortmod2}
	 \widetilde \prob(D_1(x)) \geq  c m^{-2} \quad \text{uniformly in $m \geq m_1$ and in $\ell, \cC, x$}\ 
	 \end{equation}
	 by an edge modification argument.
	 Let us write $\omega$ for a typical configuration in $D_4(x,x^*)$, considered as an element of $\Omega$. That is, we say $\omega \in \Omega$ is an element of $D_4(x,x^*)$ if $\omega \in \{\fC_{\mathrm{Rect}(\ell m)} = \cC\}$ and if the restriction of $\omega$ to $\widetilde \Omega$ is an element of $D_4(x,x^*)$.
	 We write $\omega'$ for the modification of $\omega$ produced as follows. We close all edges of $\mathcal{E}(\widetilde \Zd)$ with an endpoint in $\widetilde \Zd \cap B(x;2K)$ except those in $\fC_{\widetilde \Zd}(x^*)$. We then open edges of the form $\{x + n\textbf{e}_1, x + (n+1) \textbf{e}_1\}$ for $0 \leq n < K$ one by one, until the first time that $x$ and $x^*$ have an open connection in $\mathrm{Rect}((\ell+1)m)$ (at which time we stop opening edges).
	 
	  Then in $\omega'$, we still have $\fC_{\mathrm{Rect}(\ell m)}(0) = \cC$, since we have not opened or closed an edge with both endpoints in $\mathrm{Rect}(\ell m)$. Moreover, the vertices $y$ counted by the $\widetilde X$ variables from \eqref{eq:D2def} are now in $\Xi_{(\ell + 1)m}(x)$ in $\omega'$. In addition, each such $y$ has 
	  \[d_{chem}(x,y) \leq \rho m^2 + K \leq 2 \rho m^2\] (where the last inequality uses $m \geq m_1$). 
	  
	  To show that $\omega' \in D_1(x)$, we show pivotality --- that every connection from $\Xi_{\ell m}$ to $\partial \mathrm{Rect}((\ell + 1) m)$ in $\omega'$ passes through $\{x, x+\textbf{e}_1\}$ --- and then that the cluster of $x$ in the modified configuration $\omega'$ inherits the appropriate properties from the cluster of $x^*$ in the original configuration $\omega$.
	 To show pivotality, suppose $\gamma$ is an open path in $\omega'$ from $\Xi_{\ell m}$ to $\partial \mathrm{Rect}((\ell+1) m)$. Then $\gamma$ must use one of the edges opened in the mapping $\omega \mapsto \omega'$, since $\omega \in D_4(x,x^*)$. Letting $e$ be the first such edge, if $e$ is not $\{x, x+\textbf{e}_1\}$, then the edge of $\gamma$ just before $e$ must terminate at some vertex $x + i \mathbf{e}_1, \, 1 \leq i \leq K$. But this edge would have been closed by the mapping $\omega \mapsto \omega'$ unless it were an edge of $\widetilde \fC(x^*)$, implying that $x \sa{\mathrm{Rect}((\ell+1) m)} x^*$ in $\omega$, a contradiction. 
	 
	 By pivotality and the fact that the mapping $\omega \mapsto \omega'$ modifies only edges of $\mathrm{Rect}((\ell+1/2)m)$, we have
	 \begin{equation}\label{eq:localblock}
	 \fC_{\mathrm{Rect}((\ell+1)m)}(x)[\omega'] \setminus \mathrm{Rect}((\ell+1/2)m) = \fC_{\mathrm{Rect}((\ell+1)m)}(x^*)[\omega] \setminus \mathrm{Rect}((\ell+1/2)m) 
	 \end{equation}
	 and in particular that  $\Xi_{(\ell+1)m}(0)[\omega']$ is $\widetilde \fC(x^*)[\omega]$. The definition \eqref{eq:D2def} of $D_2(x^*)$ then implies that in $\omega'$, we have $X_{(\ell+1)m} < C_\partial m^2$; the fact \eqref{eq:noarm} implies $x$ does not have a connection to $\partial_W \mathrm{Rect}((\ell+1)m)$. To complete the proof that $\omega' \in D_1(x)$, all that remains is to show that each $y$ counted in $\widetilde X^{\rho \text{-short}}_{1,m-K}$ in $\omega$ satisfies  $y \in \Xi_{(\ell+1),m}^{2\rho\text{-short}}[\omega']$. 
	 
	 To show first that $y \in \mathrm{SREG}(0;\ell+1,m, K)[\omega']$, let $r \geq K$; we compute 
	 \begin{align}
	 &\E[|\fC(y) \cap B(y;r) \setminus \mathrm{Rect}((\ell+1/2)m)| \mid \fC_{\mathrm{Rect}((\ell+1) m)}(y)]\nonumber\\
	 = &\sum_{z \in B(y;r) \setminus \mathrm{Rect}((\ell+1/2)m)} \prob(y \lra z \mid  \fC_{\mathrm{Rect}((\ell+1) m)}(y) )\quad \text{ on $\omega'$.}\label{eq:onomegaprime}
	 \end{align}
	 	 Fix $z \in B(y;r) \setminus \mathrm{Rect}((\ell+1/2)m)$. Consider a realization $\omega''$ having the same value of $\fC_{\mathrm{Rect}((\ell+1)m)}(y)$ as in $\omega'$, and suppose that $z \in \fC(y)$. There are two possibilities:
	 	 \begin{enumerate}
	 	 	\item $z \in \fC_{\mathrm{Rect}((\ell+1)m)}(y)[\omega''] =\fC_{\mathrm{Rect}((\ell+1)m)}(y)[\omega'] $. In this case, by \eqref{eq:localblock}, we actually have that $z \in \fC_{\mathrm{Rect}((\ell+1)m)}(x^*)[\omega].$
	 	 	\item Otherwise, there is an open path from some element of $\Xi_{(\ell+1)m}[\omega']$ to $z$ which avoids $\fC_{\mathrm{Rect}((\ell+1)m)}(x)[\omega']$ (and hence $\fC_{\mathrm{Rect}((\ell+1)m)}(x^*)[\omega]$).
	 	 \end{enumerate}
	 	 In either case, using \eqref{eq:localblock}, the conditional probability of the connection from $y$ to $z$ is at most
	 	 \[\prob(y \lra z \mid \fC_{\mathrm{Rect}((\ell+1)m)}(x^*))[\omega]\ . \]
	 	Since $y$ is counted in $\widetilde X^{\rho \text{-short}}_{1,m-K}$ in $\omega$, we can use the last display to bound the sum in \eqref{eq:onomegaprime} by $C m^{9/2}$.
  As noted at \eqref{eq:D1def2} and \eqref{eq:unexpected}, this shows that there is a constant $c_1>0$ such that
	 \begin{equation}
	 \label{eq:endstep1}
	 \widetilde \prob(A(\ell,m)) \geq \widetilde \prob(D_1) \geq c_1 \text{ uniformly in $m \geq m_1, \, \ell,\, \cC$}.
	 \end{equation}
	 
	 We return briefly to the issue of the definition of $B(\ell, m).$ We note that the above argument only gives effective control of the cluster of $x$ outside of $\mathrm{Rect}((\ell+1/2)m)$. In principle, there could be many other vertices of $\Xi_{\ell m}$ whose clusters span part of $\mathrm{Rect}((\ell+1/2)m) \setminus \mathrm{Rect}(\ell m)$. Without controlling the number of vertices contained in such ``partial spanning clusters'', we would not be able to adequately bound \eqref{eq:twoscales}. The definition of $B(\ell, m)$ is designed to provide the necessary control.

\paragraph{Step 6.}
Let $c_1$ be the constant in \eqref{eq:endstep1}. We show that there is a choice of $C_{vol}$ as in the definition of $G(\ell,m)$ sufficiently large such that 
\begin{equation}
\label{eq:volatlast}
\prob(|\fC_{\mathrm{Rect}((\ell+1) m)}(0)\setminus\fC_{\mathrm{Rect}(\ell m)}(0) | < C_{vol} m^4 \mid G(\ell, m)) > 1 - c_1/2.
\end{equation}
for all $\ell$ and $m$.

Given \eqref{eq:volatlast},  $\prob(B(\ell+1, m) \mid G(\ell, m)) > 1-c_1/2$ trivially follows.
This proves the lower bound on $\prob(G(\ell+1,m))$ and completes the induction; indeed,
\begin{align*}
\prob(A(\ell+1,m) \cap B(\ell+1,m) \mid G(\ell, m))
&\geq \prob(A(\ell+1,m) \mid G(\ell, m)) + \prob(B(\ell+1,m) \mid G(\ell, m)) - 1\\
&\geq c_1 + 1 - c_1/2 - 1\\
&= c_1 / 2\ ,
\end{align*}
where we have used \eqref{eq:endstep1} and \eqref{eq:volatlast}.

We now show \eqref{eq:volatlast}, using the decomposition in \eqref{eq:inductclust}. It will suffice to show
\begin{equation}
\label{eq:clusttild}
\widetilde \prob\left( \left| \bigcup_{x \in \Xi_{\ell m}(\cC)} \widetilde \fC(x) \right| > C_{vol} m^4 \right) < c_1 / 2
\end{equation}
for a large $C_{vol}$, uniformly in $m$ and $\ell$ and in $\cC$. Of course, the clusters $\widetilde \fC(x)$ above are stochastically dominated by the corresponding clusters in $\Zd$, and so we can use the Aizenman-Barsky tail asymptotic \eqref{SizeEst} for $\Zd$ cluster sizes.

Indeed, we can upper-bound the the left-hand side of \eqref{eq:clusttild}, with $C_{vol}$ replaced by an arbitrary parameter $\tau > 0$, as follows:
\begin{align*}
\widetilde \prob\left( \left| \bigcup_{x \in \Xi_{\ell m}(\cC)} \widetilde \fC(x) \right| > \tau m^4 \right) \leq  \prob\left( \left| \bigcup_{x \in \Xi_{\ell m}(\cC)}  \fC(x) \right| > \tau m^4 \right)\ .
\end{align*}
Recalling that $X_{\ell m}(\cC) \leq C_\partial m^2$ and using Lemma \ref{eqn: cluster-est}, we see that right-hand side of the last display is at most $C \tau^{-1/2}$ uniformly in $m$, $\cC$, and $\ell$; in particular, there is a large constant $C_{vol}$ such that \eqref{eq:clusttild} holds uniformly in the same parameters. This completes the proof of Lemma \ref{lem:toinduct}.
\end{proof}

\subsection{Proof of lower bounds in Theorems \ref{thm:chem dist bd} and \ref{thm:volann}}
We first prove the lower bound of Theorem \ref{thm:volann}. Recalling the constant $m_0$ from Lemma \ref{lem:toinduct}, we assume $\lfloor \lambda^{1/3} n\rfloor \geq m_0$; this is where the constraint on $\lambda$ arises. We fix $m =  \lfloor\lambda^{\frac{1}{3}}n\rfloor $ and set $\ell = \lceil n / m \rceil$. By Lemmas \ref{lem:xshort}, \ref{lem:toinduct} and the one-arm probability \eqref{eq:onearmprob}, we see
\begin{align*}
  \prob(|\fC(0)| \leq \lambda n^4 \mid 0 \lra \partial B(n)) &\geq c n^2 \prob(|\fC(0)| \leq \lambda n^4, 0 \lra \partial B(n))\\
                                                           &\geq c n^2 \prob(G_\ell) \geq n^2 m^{-2} c^{\ell+1} \\
  &\geq c^{\ell + 1} \geq  c \exp(-C \lambda^{-1/3})\ .
\end{align*}

Similarly, to prove \eqref{eq:chemsmall} from Theorem \ref{thm:chem dist bd}, we take $m = \lfloor \lambda n\rfloor$ (assuming that this is at least $m_0$) and again set $\ell = \lceil n/m \rceil$. We note
\begin{align*}
  \prob(S_n \leq \lambda n^2 \mid 0 \lra \partial B(n)) &\geq c n^2 \prob(S_n < \lambda n^2, 0 \lra \partial B(n))\\
                                                           &\geq c n^2 \prob(G_\ell) \geq n^2 m^{-2} c^{\ell+1} \\
  &\geq c^{\ell + 1} \geq  c \exp(-C \lambda^{-1})\ .
\end{align*}
The lower bounds are proved.\qed

\section{Proof of Theorem \ref{thm: bdy-connects} and of \eqref{eq:chemsmall0} from Theorem \ref{thm:chem dist bd}}
\label{sec:bdy-connects}


We recall the correlation length $\xi(p)$ introduced for $p < p_c$ in Definition \ref{defin:armsetc}. The lower tail of the critical chemical distance will be related to the behavior of $\pi_p(n)$ with $n$ of order $\xi(p)$. We introduce a quantity to be denoted $L_\gd(p)$ which is related to $\xi(p)$ and which will play the role of $L(p)$ from $\Z^2$ appearing in \eqref{eq:pitwod}.
For each finite vertex set $D \subseteq \Z^d$  satisfying $0 \in D$, we write, similar to notation of Section \ref{sec:HSproof},
\begin{equation}	\label{eq:XforD}	
X_D := \{x \in \partial D: 0 \sa{D}x\} = \mathfrak C_D(0) \cap \partial D\ .	
\end{equation}	
For any $n\in\dN, \gd>0$ and $p\in(0,  p_c)$, we define	
\[	
\sD(n) := \{D\subset\dZ^d: 0\in D \text{ and } \sup\{\|x\|_\infty: x\in D\}\le n\},\] 
and
\begin{equation}
L_\gd(p) := \inf \left\{n \geq 1: \inf_{D\in\sD(n)}\E_p[|X_D|] \leq \delta\right \}\ .	 \label{eq:Lpdef}
\end{equation}
See \cite{duminil2016new}, where a related quantity was used to provide a new proof of the fact that $\xi(p) < \infty$ whenever $p < p_c$. See also \cite{G99} for exposition of earlier proofs of this fact. As a consequence of $\xi(p) < \infty$, we have $L_\gd(p) < \infty$ for any $p < p_c$. Moreover, $L_\gd(p) \uparrow \infty$ as $p \nearrow p_c$ with $\gd>0$ held constant.

\subsection{Upper bound on $\pi_p(n)$ from Theorem \ref{thm: bdy-connects}}
The upper bound on $\pi_p(n)$ from Theorem \ref{thm: bdy-connects} follows by combining Lemma \ref{lem:useLp} and Lemma \ref{lem:Lasymp} stated below.
\begin{lem}\label{lem:useLp}	
There is a constant $C>0$ (depending on $d$ only) such that uniformly  in $n$, $\gd\in(0, \min\{C^{-1}, e^{-4}/2^{8}\})$, and $p < p_c$, 
\beq \label{eq:useLp}
\prob_p(0 \lra \partial B(n)) \leq C n^{-2} \exp(-n/ L_\gd(p))\ . \eeq	
\end{lem}

\begin{lem}\label{lem:Lasymp}		
For $\gd$ as in the statement of Lemma \ref{lem:useLp}, there are constants $c(\gd), C(\gd)>0$ such that 
\[c (p_c - p)^{-1/2} \leq L_\gd(p) \leq C (p_c - p)^{-1/2}\] uniformly in $p \in(0, p_c)$.	
\end{lem}	
We recall that the asymptotic behavior of $\xi(p)$ as $p\nearrow p_c$ is known \cite{H90}, namely $\xi(p)\asymp (p_c - p)^{-1/2}$. Lemma \ref{lem:Lasymp} shows that identical asymptotic behavior holds for $L_\gd(p)$. 

\begin{proof}[Proof of Lemma \ref{lem:useLp}] We will use the following claim, whose proof is given after the proof of the lemma.				
\begin{clam} \label{ExpDecay}	
There is a constant $c_1(d)$ such that $\E_p[X_{B(k L_\delta(p))}] \leq \delta^{k/4}$ for all $\gd<c_1$, $p<p_c$, and integers $k\geq 4$.						
\end{clam}	
Claim \ref{ExpDecay} is related to Theorem 2 of \cite{H57} or Lemma 1.5 of \cite{duminil2016new}.
Given Claim \ref{ExpDecay}, we prove the lemma using an induction argument. 
For $\ell\in\dN$, our $\ell$th induction hypothesis is that the inequality in \eqref{eq:useLp} holds for all $n\le 2^\ell L_\gd(p)$ and $p<p_c$, where $C:= \max\{Ae^8, c_1^{-1}\}$, for $c_1$ as in Claim \ref{ExpDecay} and 
where $A$ is the implicit constant in the upper bound in \eqref{eq:onearmprob}. To prove our hypothesis for $\ell\leq 3$ we
use \eqref{eq:onearmprob} and the monotonicity property of $\prob_p(\cdot)$ in $p$ to see
\beq \label{eq:BaseCase}
\prob_p(0 \lra \del B(n)) \le  \prob_{p_c}(0 \lra \del B(n))  \le Cn^{-2}e^{-n/L_\gd(p)}
 \eeq 
 for all $p<p_c$ and $n\le 8L_\gd(p)$. \eqref{eq:BaseCase} proves our induction hypothesis for $\ell\leq 3$.

\begin{figure}
    \centering
    \includegraphics{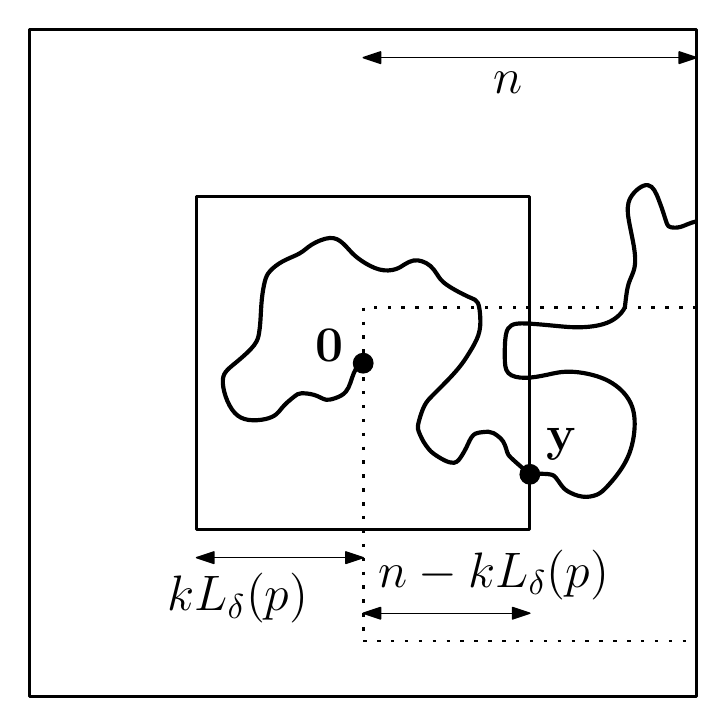}
    \caption{Geometry in the proof of Lemma \ref{lem:useLp}}
    \label{fig:useLp}
\end{figure}

Assuming that the $\ell$th induction hypothesis is true, we now prove the $(\ell+1)$st hypothesis. Without loss of generality, we can take $n\in(2^{\ell}L_\gd(p), 2^{\ell+1} L_\gd(p)]$, as all $n\le 2^{\ell}L_\gd(p)$ are covered in the $\ell$th hypothesis.  We take $k:=\lfloor n/(2L_\gd(p))\rfloor$. If $\{0\lra\del B(n)\}$ occurs, then there must be a $y\in\del B(kL_\gd(p))$ such that $\{0\sa{B(kL_\gd(p))} y\}$ and $\{y\lra \del B(y; n-kL_\gd(p))\}$ occur disjointly. See Figure \ref{fig:useLp} for an illustration. So, using a union bound, the BK inequality, and our $\ell$th induction hypothesis,
\begin{align*}							
\prob_p(0 \lra \partial B(n)) &\leq \sum_{y \in \partial B(k L_\gd(p))} \prob_p(0\sa{B(k L_\gd(p))} y) \prob_p(y \lra \partial B(y; n-kL_\gd(p)))\\	
&\leq C (n-k L_\gd(p))^{-2} \exp\left(-\frac{n-kL_\gd(p)}{L_\gd(p)}\right) \sum_{y \in \partial B(k L_\gd(p))} \prob_p(0 \sa{B(k L_\gd(p))} y)\\	
&\leq C (n/2)^{-2} e^{k-n/L_\gd(p)} \E_p[X_{B(kL_\gd(p))}],			
\end{align*}
as $n-kL_\gd(p)\ge n/2$.
Finally, note that $\E_p[X_{B(kL_\gd(p))}]\le\gd^{k/4}$ by Claim \ref{ExpDecay}, and $4e\gd^{1/4}<1$.  So the RHS of the last display is $\le Cn^{-2}e^{-n/L_\gd(p)}$, which proves the $(\ell+1)$st induction hypothesis. This completes the proof of the induction argument and the lemma.
\end{proof}

\begin{proof}[Proof of Claim \ref{ExpDecay}] 
We abbreviate $m = k L_\delta(p)$. Let $D$ be the infimizing set appearing in the definition \eqref{eq:Lpdef} of $L_\delta(p)$.
We expand the expectation:
\begin{equation}
\label{eq:basicdecomp}
\E_p[X_{B(m)}] = \sum_{z \in \partial B(m)} \tau_{B(m),p}(0, z)\ .
\end{equation}
 Consider an outcome in  $\{0 \sa{B(m)} z\}$, where $z \in \partial B(m)$. In this outcome, we can decompose the connection into segments which extend roughly distance $L(p)$. We let $y_1$ be the first vertex of $\partial D$ encountered by some open path from $0$ to $z$, then let $y_2$ be the first vertex on $y_1 + \partial D$ encountered by this path after $y_1$, and so on. Proceeding in this way, we see there is a sequence $0 = y_0, y_1, \ldots, y_{r}$ of vertices of $B(m)$ with $r = \lfloor k/2 \rfloor$, such that $y_{\ell+1} \in [y_\ell + \partial D]$ for each $\ell \leq r-1$, such that $\|y_{r}-z\| \geq m/2$, and such that the following disjoint connection event occurs:
\[ \{0\sa{D} y_{1}\} \circ \{y_1 \sa{y_1 + D} y_2\} \circ \ldots \circ \{y_{r-2} \sa{y_{r-2} + D} y_{r-1} \} \circ \{y_{r} \sa{B(m)} z\}\ .\]

We apply the BK inequality and sum over the $y_\ell$'s. Each term has a factor of the form $\tau_{B(m),p}(y_r, z)$; this is at most $\tau_{ B(m),p_c}(y_r, z)$ and so is uniformly bounded by $C m^{1-d}$ using \eqref{eqn: hs-tp}. This leads us to the estimate
\[\tau_{B(m),p}(0,z) \leq Cm^{1-d} \sum_{y_1 \in \partial D} \sum_{y_2 \in [y_1 + \partial D]} \ldots \sum_{y_r \in [y_{r-1}+\partial D]} \tau_{D,p}(0, y_1) \ldots, \tau_{y_{r-1} +  D,p}(y_{r-1}, y_r)\ . \]
Evaluating the $y_\ell$ sums and using the definition of $D$, the above is bounded by
\[C m^{1-d} \delta^r\ . \]
Finally, we sum over $z \in \partial B(m)$ to find
\[\E_p[X_m] \leq C \delta^{k/2-1} \leq \delta^{k/4} \]
for all $\delta$ smaller than some $d$-dependent constant and all $k \geq 4$. This proves the claim.
\end{proof}	

\begin{figure}
    \centering
    \includegraphics[scale=0.6]{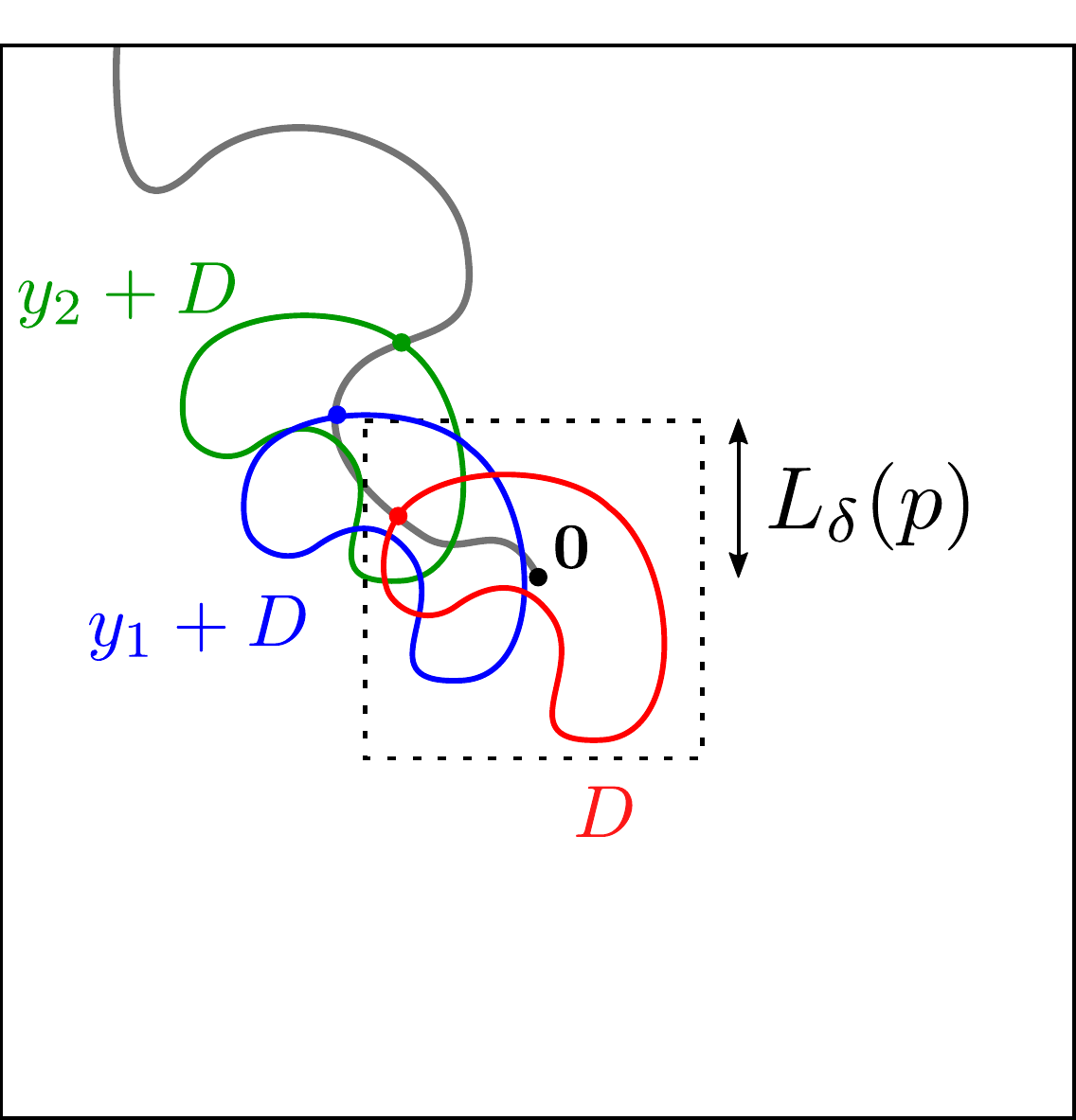}
    \caption{Geometry in the proof of Claim \ref{ExpDecay}: the red dot represents $y_1$, the blue dot is $y_2$, $y_3$ is green.}
    \label{fig:iteratebox}
\end{figure}

\begin{proof}[Proof of Lemma \ref{lem:Lasymp}]		
To prove the upper bound for $L_\gd(p)$, first we recall the following bound from \cite[(1.3)]{duminil2016new}:		
\begin{equation}	
\label{eq:dtsharp}		\frac{\mathrm{d}}{\mathrm{d} p} \prob_p(0 \lra \partial B(n)) \geq \frac{1}{p(1-p)}[\prob_p(0 \not \lra \partial B(n))] \inf_{D \in \sD(n)} \E_p[|X_D|].		\end{equation}
Since $p_c\le 1/2$ and $\prob_p(0\not\lra\partial B(n))$ is decreasing (resp.~increasing) in $p$ (resp.~$n$),
\[ \frac{1}{p(1-p)}[\prob_p(0 \not \lra \partial B(n))]  \geq \frac{1}{p_c(1-p_c)}[\prob_{p_c}(0 \not \lra \partial B(1))]  =\frac{(1-p_c)^{2d-1}}{p_c}=: c_0\] 
for all $n\ge 1$ and $p<p_c$. 
Combining the last two displays,	
we arrive at the following bound.		
\begin{equation}	
\label{eq:dtsharp2}
\frac{\md}{\md p} \prob_p(0 \lra \partial B(n)) \geq c_0 \inf_{D \in \sD(n)} \E_p[|X_D|]\ , \quad \text{uniformly in $n \geq 1,  p < p_c$}\ .		
\end{equation}		
Next, we integrate both sides of the above inequality from $p$ to $p_c$ (using the continuity of $\prob_p(E)$ for each cylinder event $E$) to see
\begin{align}		
0 \leq \prob_p(0 \lra \partial B(n)) &\leq \prob_{p_c}(0 \lra \partial B(n)) - c_0 \int_{p}^{p_c} \inf_{D\in\sD(n)} \E_q[|X_D|]\, \md q\nonumber\\		
&\leq C n^{-2} - c_0 \int_{p}^{p_c} \inf_{D\in\sD(n)} \E_q[|X_D|]\, \md q\ ,\label{eq:dtsharp3}		
\end{align}		where in the last line we used \eqref{eq:onearmprob}. Clearly, $\E_q[X_D]$ is increasing in $q$ for each fixed $D$; we can therefore bound the right-hand side of \eqref{eq:dtsharp3} by taking $q = p$ inside the integral, and obtain the inequality		
\[C n^{-2} \geq c_0(p_c - p) \inf_{D\in\sD(n)} \E_p[|X_D|]\ , \]
uniformly in $n \geq 1$ and $p < p_c$. Now, choosing $p_0\in(0,p_c)$ such that $p>p_0$ implies $L_\gd(p)\ge 2$, and taking $n = L_\delta(p)-1$, we have		
\[C (L_\gd(p)-1)^{-2} \geq c_0 \delta (p_c - p) \text{ for all } p\in(p_0,p_c)\ .\]
This proves the upper bound for $L_\gd(p)$.

To prove the lower bound for $L_\gd(p)$, recall that (see \cite{H90})
\beq \label{xi est}
 \left[\lim_{n\to\infty} \frac{-\log \prob_p(0\lra n\ve_1)}{n}\right]^{-1}:=\xi(p) \asymp (p_c-p)^{-1/2}.
\eeq

Also, $\prob_p(0\lra n \ve_1)\le \prob_p(0\lra \del B(n))\le \E_p X_{n} \le \gd^{n/4 L_{\delta}(p)}$ for $n = k L_\delta(p)$ with $k \geq 4,$ by Claim \ref{ExpDecay}. Using this last display, and looking at the limit as $k\to\infty$ after taking the $n$-th root of both sides of the last inequality, we see that $c_1\xi(p)\le L_\gd(p)$ for some constant $c_1$. This together with \eqref{xi est} proves the lower bound for $L_\delta(p)$.
\end{proof}	

\subsection{Lower bound for the subcritical one arm probability}
For $\lambda \geq 0$, define
\[ \pi_p(n; \lambda) = \prob_p(\sA_{n,\lambda}), \text{ where } \sA_{n,\lambda}:=\{0 \lra \partial B(n), S_n < \lambda n^2\}. \]
The goal is to use the Russo's formula to compute the derivative of the above and show that $\pi_p(n ; \lambda)$ is not too small for a ``good choice of $\lambda$".
Using Russo's formula \eqref{eqn: russo},
\[\frac{\mathrm{d}}{\mathrm{d}p} \pi_p (n; \lambda) = \E_pN_{n,\lambda}, \text{ where } N_{n,\lambda}:= \sum_{e \in \mathcal{E}(B(n))} \mathbf 1_{\{e \text{ is pivotal for the event } \sA_{n,\lambda}\}}. \]
It is easy to see that if 
\[N'_{n,\lambda}:= \sum_{e \in \mathcal{E}(B(n))} \mathbf 1_{\{e \text{ is open and is pivotal for the event } \sA_{n,\lambda}\}},\] 
then  $N'_{n,\lambda}\leq \lambda n^2\mathbf 1_{\sA_{n,\lambda}}$ and $\E N_{n,\lambda}= p^{-1}\E N'_{n,\lambda}$. It follows that
\[\frac{\mathrm{d}}{\mathrm{d}p} \pi_p (n; \lambda) \leq p^{-1}\lambda n^2 \pi_p(n;\lambda)\ .\]
Therefore, for any $p_0\in (0, p_c)$ and $p\in(p_0, p_c)$, there is a constant $C(p_0)$ such that
\[\frac{\mathrm{d}}{\mathrm{d}p} \log\left( \pi_p(n ; \lambda) \right) \leq C \lambda n^2\ . \]
Integrating both sides of the above inequality from $p$ to $p_c$,
\[\log\left( \frac{\pi_{p_c}(n; \lambda)}{\pi_p(n;\lambda)} \right)  \leq C (p_c - p) \lambda n^2, \]
which is equivalent to
\[ \frac{\pi_{p_c}(n; \lambda)}{\pi_p(n;\lambda)} \leq \exp(C (p_c - p) \lambda n^2)\ .\]
In other words,
\[\pi_p(n;\lambda) \geq \exp( -C (p_c - p) \lambda n^2) \pi_{p_c}(n; \lambda). \]
Using the lower bound for $\pi_{p_c}(n;\lambda)$ from Theorem \ref{thm:chem dist bd}, we obtain
\[\pi_p(n;\lambda) \geq \exp( -(p_c - p) \lambda n^2) \exp(- C / \lambda) n^{-2}\ . \]
Now we choose $\lambda$ to optimize the RHS of the above display. Choosing $\lambda = [n \sqrt{p_c - p}]^{-1}$, we ger
\[\pi_p(n;\lambda) \geq  \exp(- C n \sqrt{p_c - p}) n^{-2}\ . \]
This completes the proof of the lower bound.

\subsection{Upper bound for the critical chemical distance}
We will employ the usual coupling of the measures $\prob_p$ for different values of $p$. Let $(\go_e)_e$ be i.i.d.~$\text{Uniform}(0,1)$,  $\Go^n=(\go_e: \text{both endpoints of $e$ are in $B(n)$})$, and $\prob_{\Go^n}$ denote the distribution of $\Go^n$. An edge $e$ is called $p$-open  if $\go_e\le p$. A path is called $p$-open if all the edges on that path are $p$-open. Let $S_n(p)$ denote the smallest number of edges on any $p$-open path connecting $0$ and $\del B(n)$. Also let $\{0\lra_p A\}$ denote the event that there is a $p$-open path connecting $0$ and $A$.

We use the following inequality, which has been used in the first display of \cite[Section 2]{KZ93}.
\begin{align}
\prob_{\Go^n}(0 \lra_p \partial B(n) , |S_n(p)| =\ell)&\ge \prob_{\Go^n}(S_n(p_c)=\ell \text{ and  the optimal path is }p\text{-open})\nonumber\\
&\ge \left(\frac{p}{p_c}\right)^\ell \prob_{\Go^n}(0\lra_{p_c} \partial B(n), |S_n(p_c)|=\ell). \label{eq:firstppc}
\end{align}
Summing over $\ell\leq k$ and dividing both sides by $\prob_{\Go^n}(0\lra_{p_c}\del B(n))$, 
\[\prob_{p_c}(|S_n|\le k\mid 0\leftrightarrow \partial B(n))\le C\left(\frac{p_c}{p}\right)^k \frac{\prob_p(0\leftrightarrow \partial B(n))}{\prob_{p_c}(0\leftrightarrow \partial B(n))}. \]
Using the inequality $\log(x)\le x-1$ for all $x>1$,
\[\left(\frac{p_c}{p}\right)^k=\exp\big(k(\log p_c-\log p)\big)\leq \exp\left(k\frac{p_c-p}{p}\right) \text{ for all }p<p_c.\]
Combining the last two estimates, using the upper bound on the subcritical one-arm probability given in Theorem \ref{thm:scalingub}, and applying the lower bound in \eqref{eq:onearmprob}, there are constants $c,C>0$ such that 
\begin{align*}
\prob_{p_c}(|S_n|\le k\mid 0\leftrightarrow \partial B(n))\le
C\exp\left(k\frac{p_c-p}{p} -c n\sqrt{p_c-p}\right).
\end{align*}
Replacing $k$ by $\gl n^2$ and $p$ by $p_c-\frac{1}{C_0\gl^2n^2}$,
\begin{align*}
\prob_{p_c}(|S_n|\le k\mid 0\leftrightarrow \partial B(n))\le
\exp\left(-\lambda^{-1}\left[\frac{c}{C_0}-\frac{2}{C_0^2}\right]\right).
\end{align*}
Choosing $C_0>2/c$ we get the desired upper bound.

\subsection{Point-to-point corollaries \label{sec:ptchem}}
In this section, we prove the corollary stated at \eqref{eq:ptchempreview} and a related extension to half-spaces. These will also be useful in the proof of Theorem \ref{thm:chem dist bd2}. We state the results here formally:
		\begin{cor}\label{cor:chempt}
			There exist constants $C, c>0$ such that the following bounds on the lower tail of the point-to-point chemical distance hold:
			\begin{align*}
		\text{for all $x \in \Zd$, }	\prob(0 \lra x, d_{chem}(0, x) \leq \lambda \|x\|^2) &\leq C e^{-c/\lambda} \|x\|^{2-d};\\
		\text{for all $x \in \Zd_+$, }\,\prob(m \mathbf{e}_1 \sa{\Zd_+} x ,\, d_{chem}^H(m\mathbf{e}_1, x) \leq \lambda \|x - m\mathbf{e}_1\|^2) &\leq C e^{-c/\lambda} m \|x -m\mathbf{e}_1\|^{1-d}\ .
			\end{align*}
			\end{cor}

We recall that $d_{chem}^H$ is the analogue of $d_{chem}$ for percolation restricted to the half-space $\Zd_+$. To prove the corollary, we need an intermediate lemma relating point-to-box chemical distances to point-to-point chemical distances.
For $\lambda > 0$, let 
\[\widehat X_{B(n)}^k = \#\{x \in \partial B(n): \, x \sa{B(n)}0 \text{ by a path of fewer than $k$ edges}\}\ .\]
	In other words, $\widehat X^k_{B(n)}$ is the number of vertices $x \in \partial B(n)$ having $d_{chem}^{B(n)}(0, x) \leq k$.
	
	\begin{lem}\label{lem:lfind}
		There is a uniform constant $C$ such that, for each $n\geq 1$ and each $\lambda > 0$, there is an $\ell \leq n/2$ with
		\[\E_{p_c}[\widehat X_{B(\ell)}^{\lambda n^2}] \leq C \exp(-(C \lambda)^{-1})\ . \]
		\end{lem}
	\begin{proof}
		We fix $\delta$ small as in Lemma \ref{lem:useLp} and Claim \ref{ExpDecay}. We will assume $n \geq 8$; the extension to smaller values of $n$ is trivial.  The parameter $p < p_c$ will be chosen later such that $L_\delta (p) \leq n/2$; we set $k = \lfloor n/2L_\delta(p)\rfloor$.  Our ultimate choice of $p$ will depend on $\lambda$ and $n$, and we will need $\lambda$ smaller than some uniform constant to ensure $L_\delta(p) \leq n/2$; we assume this in what follows, since we can handle larger $\lambda$ by adjusting constants.
		
		  Similarly to \eqref{eq:firstppc}, we see that for each $y \in \partial B(n)$ and each $\lambda >0$,
		\[\prob_{p_c}(y \text{ is counted in $\widehat X_{B(k L_\delta (p))}^{\lambda n^2}$}) \leq \left(\frac{p_c}{p}  \right)^{\lambda n^2} \prob_p(y \text{ counted in } X_{B(k L_\delta(p))})\ . \]
		Summing the last inequality over $y \in \partial B(n)$, we find
		\begin{align*}
		\E_{p_c}[\widehat X_{B(k L_\delta(p))}^{\lambda n^2}] &\leq \left(\frac{p_c}{p}  \right)^{\lambda n^2} \delta^{k/4} \leq \left(\frac{p_c}{p}  \right)^{\lambda n^2} e^{-C n (p_c - p)^{-1/2}}\\
		&\leq \exp\left(\lambda n^2\, \frac{p_c-p}{p} - C n (p_c-p)^{-1/2} \right)\ .
		\end{align*}
		where we have used Claim \ref{ExpDecay} and then Lemma \ref{lem:Lasymp}. The constant here is uniform in $n$ and $p$ as above. 
		
		We set $p_c - p = (C_1\lambda^2 n^2)^{-1}$ for a suitably large uniform $C_1 > 0$.  The last estimate becomes
		\[\text{For all $n$ and $\lambda$, }\quad \E_{p_c}[X_{k L_\delta(p)}^{\lambda n^2}] \leq C \exp(-c/\lambda)\ . \]
		 Since $k L_\delta(p) \leq n/2$ for $\lambda$ small relative to our constant $C_1$, the proof is complete with $\ell = k L(p)$. 
		 	\end{proof}

		\begin{proof}[Proof of Corollary \ref{cor:chempt}]
			We prove only the second inequality. The first is simpler to show, and the argument requires only minor modifications.
			
			We find an $\ell$ as in Lemma \ref{lem:lfind} (with the role of $n$ played by $\|x - m\mathbf{e}_1\|/2$). Then, on the event under consideration, we can find a $y \in B(x; \ell)$ such that
			\[\{y \sa{B(x;\ell)} x, \, d_{chem}^{B(x;\ell)}(x, y) \leq \lambda \|x - m\mathbf{e}_1\|^2\}\circ \{y \sa{\Zd_+} m\mathbf{e}_1\}\]
			occurs. 
			Summing over $y$ and applying the BK inequality and Theorem \ref{thm:scalingub}, we find
			\begin{align*}
			\prob(m \mathbf{e}_1 \sa{\Zd_+} x ,\, d_{chem}^H(m\mathbf{e}_1, x) \leq  \lambda \|x - m\mathbf{e}_1\|^2) &\leq C m \|x - m\mathbf{e}_1\|^{1-d} \E[X_\ell^{\lambda (\|x - m \mathbf{e}_1\|/2)^2}]\\
			&\leq C e^{-c/\lambda} m  \|x - m\mathbf{e}_1\|^{1-d}\ ,
			\end{align*}
			as claimed.
			\end{proof}

\section{Chemical distance upper tail \label{sec:chemup}}
In this section, we prove Theorem \ref{thm:chem dist bd2}. We actually show something stronger; namely, that the length of the \emph{longest} self-avoiding path from $0$ to $\partial B(n)$ has exponential upper tail on scale $n^2$. In Section \ref{sec:up1st}, we make some necessary definitions and then perform a first moment calculation. In Section \ref{sec:up2nd}, we compute higher moments and conclude the proof. We then comment briefly on how to show \eqref{eq:ptchempreview2} using similar ideas.
\subsection{First moment bound \label{sec:up1st}}

 Given a vertex $y \in \Zd_+$, let $\longestH(y)$ be the length of the longest self-avoiding open path from $y$ to $\partial \Zd_+$, if such a path exists. Otherwise we set $\longestH(y) = 0$. This convention will be useful for avoiding expressions such as $\longestH(y) \mathbf{1}_{\{y \lra \partial \Zd_+\}}$.
 
 We let $\beta(y)$ denote a measurably chosen maximizer in the definition of $\longestH(y)$, with $\beta(y) = \varnothing$ if no path from $y$ to $\partial \Zd_+$ exists. Then $\E[\longestH(y)] = \E[|\beta(y)|]$ by definition, where we interpret $\beta(y)$ as a sequence of vertices when computing the cardinality. We provide a uniform upper bound on the expectation:
\begin{equation}
\label{eq:longestind}
\sup_{y \in \Zd_+} \E[\longestH(y)] < \infty.
\end{equation}

In what follows, we consider a fixed vertex $x$ in $\Zd_+$ and then provide an upper bound on $\E[\longestH(x)]$ which will be seen to be uniform in $x$.
 For ease of notation, we let $\delta = x(1)$ denote the distance of our vertex from $\partial \Zd_+$. Keeping track of $\delta$-dependence will allow us to make sure our constant upper bound is indeed uniform.
 
 We first peel off an inconsequential piece of the expectation:
 \begin{equation}
 \label{eq:incons}
 \E[\longestH(x); \longestH(x) \leq \delta^2] \leq \delta^2 \prob(x \lra \partial B(x;\delta)) \leq C\ ,
 \end{equation}
 where in the last inequality we used the one-arm probability bound \eqref{eq:onearmprob}. The constant here is uniform because it is just the constant appearing in that upper bound on $\pi(n)$. On the event that $\longestH(x) > \delta^2$, we have to do significantly more work. We let $\beta'(x)$ denote the ``first half'' of $\beta(x)$ --- in other words, the segment of $\beta(x)$ beginning at $x$ and terminating after $\lfloor |\beta(x)|/2 \rfloor$ edges. Of course, $\E[|\beta(x)|] \leq 2 \E[|\beta'(x)|]+1$, so if we can show 
  \begin{equation}
 \label{eq:cons}
 \E[|\beta'(x)|; |\beta(x)| > \delta^2] \leq C,
 \end{equation}
then the proof of \eqref{eq:longestind} will be complete.

We first sum over $B(x;\delta)$. Let $A(z;r)$ denote the event that a vertex $z$ has an intrinsic arm to distance $r$, as defined at \eqref{eq:intonearm}. If $z \in \beta'(x) \cap B(x;\delta)$ and $\longestH(x) > \delta^2$, then $\{x \lra z\} \circ A(z;\delta^2/2)$ occurs. Using the BK inequality, we see
\begin{align}
\E[|\beta'(x) \cap B(x;\delta)|; \longestH(x) > \delta^2] &\leq \sum_{z \in B(x;\delta)} \tau(x,z) \prob(A(z;\delta^2/2))\nonumber\\
& \leq C \delta^{-2} \sum_{z \in B(x;\delta)} \tau(x,z) \leq C\ ,\label{eq:firstbox}
\end{align}
where we have used the intrinsic one-arm probability upper bound \eqref{eq:intonearm}.

 To count the remaining portion of $\beta'(x)$, we will replicate the calculation leading to \eqref{eq:firstbox} by summing over scales --- here we are more careful and exploit the fact that the $\tau$ from \eqref{eq:firstbox} could actually be taken as a $\tau_H$. The more rapid decay of $\tau_H$, from Theorem \ref{thm:scalingub}, will be necessary to show the sum converges. Let us abbreviate $\mathfrak{A}_k = Ann(x; \delta 2^{k}, \delta 2^{k+1})$. Then
 \begin{equation}
 \begin{split}
 \label{eq:twolong}
 \E[|\beta'(x) \cap \mathfrak{A}_k|; \longestH(x) > \delta^2] =  &\E[|\beta'(x) \cap \mathfrak{A}_k|;\, \longestH(x) > 2^{3k/2}\delta^2] \\
 + &\E[|\beta'(x) \cap \mathfrak{A}_k|;\, 2^{3k/2} \delta^2 \geq \longestH(x) > \delta^2]\ .
 \end{split}
 \end{equation}
 We bound each of the terms on the right-hand side of \eqref{eq:twolong} by different methods.
 
 For the first term, we note that when $\longestH(x) > 2^{3k/2} \delta^2$, each $z \in \beta'(x) \cap \mathfrak{A}_k$ must satisfy
 \[\{z \sa{\Zd_+} x\} \circ A(z; 2^{3k/2} \delta^2/2)\ . \]
 Applying the BK inequality and summing, we find
 \begin{align*}
\E[|\beta'(x) \cap \mathfrak{A}_k|; \longestH(x) > 2^{3k/2}\delta^2] &\leq \sum_{z \in \mathfrak{A}_k} \tau_H(x,z) \prob(A(z;2^{3k/2}\delta^2/2))\nonumber\\
 & \leq C \delta^{-2} 2^{-3k/2} \sum_{z \in \mathfrak{A}_k} \tau_H(x,z)\\
 &\leq C \delta^{-2} 2^{-3k/2} \times (\delta 2^{k})^d \times \delta \times  (\delta 2^k)^{-(d-1)}\\
 &\leq C2^{-k/2}\ .
 \end{align*}
For the second term of \eqref{eq:twolong}, we use Corollary \ref{cor:chempt}:
\begin{align*}
 &\E[|\beta'(x) \cap \mathfrak{A}_k|; 2^{3k/2} \delta^2 \geq \longestH(x) > \delta^2] \\
\leq~ &\sum_{z \in \mathfrak{A}_k}\prob\left(x \sa{\Zd_+}z, \, d_{chem}^H(x,z) < 2^{-k/2} (\delta 2^k)^2\right) \prob(A(z; \delta^2/2))\\
\leq~ & C \delta^{-2} \times (\delta 2^k)^d \times e^{-c 2^{k/2}} \times \delta \times (\delta 2^k)^{1-d} \leq C 2^k e^{-c 2^{k/2}}\ .
\end{align*}

In both cases, all constants arise from the estimates on the one-arm probability, the chemical distance lower tail, or the asymptotics for $\tau_H$. In particular, these constants are uniform in $k$ and $x$. Combining the two estimates, we get that the left-hand side of \eqref{eq:twolong} is bounded uniformly by
\[ C 2^{-k/2}\ .\]
Summing over $k$ shows \eqref{eq:cons}, and recombining this with \eqref{eq:incons} completes the proof.
 
 \subsection{Higher moments of path length \label{sec:up2nd}}
 Let $\longest_n$ denote the length of the longest self-avoiding open path from $0$ to $\partial B(n)$ which lies entirely within $B(n)$. As before, we set $\longest_n = 0$ if no open arm from $0$ to $\partial B(n)$ exists. We now show the following result, which implies Theorem \ref{thm:chem dist bd2} via the trivial inequality $S_n \leq \longest_n$ on $\{0 \lra \partial B(n)\}$.
 \begin{prop}
 There exists a constant $C_1$ such that, for all integers $n, k \geq 1$,
 \[\E[\longest_n^k \mid 0 \lra \partial B(n)] \leq k! (C_1n^2)^k\ .  \]
 In particular, there is a constant $C_2$ such that
 \[\prob(\longest_n \geq \lambda n^2 \mid 0 \lra \partial B(n)) \leq C_2 \exp\left( \frac{-\lambda}{C_2} \right)\ . \]
 \end{prop}

	\begin{proof}
	The second claim follows by using the first to bound the moment generating function of $\longest_n / n^2$. It therefore suffices to bound the moments of $\longest_n$.  
	Similarly to before, we let $\beta_n$ denote a measurably chosen self-avoiding open path from $0$ to $\partial B(n)$ of maximal length. By expanding $\longest_n$ into a sum of indicators and using \eqref{eq:onearmprob}, we find
	\begin{equation}
	    \label{eq:momentsum}
	    \E[\longest_n^k \mid 0 \lra \partial B(n)] \leq Cn^2 \sum_{z_1, \ldots, z_k \in B(n)} \prob(z_1, \ldots, z_k \in \beta_n, \, 0 \lra \partial B(n))\ . 
	\end{equation}
	
	Since $\beta_n$ is self-avoiding, the vertices $z_1, \ldots, z_k$ appear in a well-defined order along this path. We abbreviate ``$w$ and $y$ lie on $\beta_n$ with $w$ appearing before $y$ in order starting at $0$'' by $w \prec y$. Then
	\begin{align}
	    \eqref{eq:momentsum} &= (C n^2)(k!)\sum_{z_1, \ldots, z_k \in B(n)}  \prob(z_1 \prec z_2 \prec \ldots \prec z_k, \, 0 \lra \partial B(n))\nonumber\\
	    &=(C n^2)(k!) \sum_{z_1, \ldots, z_{k-1} \in B(n)} \E\left[|\{y \in \beta_n:\, z_{k-1} \prec y \}|\mathbf{1}_{z_1 \prec \ldots \prec z_{k-1}}\right] \label{eq:momentinter}\ .
	\end{align}
	We would like to evaluate the expectation in \eqref{eq:momentinter}, and so we need some way to decouple the variables there. To make the notation for this step easier, we abbreviate
	\[V = V(z_1, \ldots, z_{k-1}) :=  \mathbf{1}_{z_1 \prec \ldots \prec z_{k-1}}; \quad W = W(z_{k-1}) = |\{y \in \beta_n:\, z_{k-1} \prec y \}| \ . \]
	
	Consider an outcome $\omega \in \{VW \geq \lambda\}$  for some real number $\lambda > 0$. We see that 
	\begin{equation*}
	\omega \in\{0 \lra z_1\} \circ \ldots \circ \{z_{k-2} \lra z_{k-1}\} \circ \{\exists \text{ open path of length $\geq \lambda$ in $B(n)$ from $z_{k-1}$ to $\partial B(n)$} \}\ . 
	\end{equation*}
	Indeed, disjoint witnesses for the events above are provided by disjoint segments of $\beta_n$.
	Letting the length of the longest open path from $z_{k-1}$ to $\partial B(n)$ which lies entirely in $B(n)$ be denoted by $W'$ and using the BK inequality, we bound
	\begin{align*}
	    \E[VW] &= \int_{0}^\infty \prob(VW \geq \lambda) \mathrm{d}\lambda\\
	    &\leq \tau(0,z_1) \ldots \tau(z_{k-2},z_{k-1}) \int_{0}^\infty \prob(W'\geq \lambda) \mathrm{d}\lambda\\
	     &= \tau(0,z_1) \ldots \tau(z_{k-2},z_{k-1}) \E[W']\ .
	\end{align*}
	
	Any open path in $B(n)$ from $z_{k-1}$ to $\partial B(n)$ is also an open path to one of the $2d$ hyperplanes containing one of the $2d$ sides making up $\partial B(n)$, with this open path lying entirely on one side of the hyperplane. In other words, $\E[W']$ is bounded above by a sum of $2d$ terms of the form $\E[\longestH(y_i)]$ for $y_i$'s appropriately chosen depending on $z_{k-1}$. Applying \eqref{eq:longestind}, we see there is a  $C$ uniform in $n$ and $z_1, \ldots, z_{k-1}$ such that
	\begin{equation}
	    \label{eq:sumVW}
	    \E[VW] \leq C \tau(0,z_1) \ldots \tau(z_{k-2},z_{k-1})\ .
	\end{equation}
   Inserting the bound of \eqref{eq:sumVW} into \eqref{eq:momentinter} and summing over $z_1$ through $z_{k-1}$, we see
   \[\E[\longest_n^k \mid 0 \lra \partial B(n)]  \leq C^k n^2 (k!) n^{2(k-1)} = k! (C n^{2})^k \ .\]
   Because $k$ was arbitrary and the constant $C$ is uniform in $n$ and $k$, the moment bound is proved.
	\end{proof}
	
	We now briefly describe how to show \eqref{eq:ptchempreview2}. Considering a shortest self-avoiding open path from $0$ to $x$, we can upper bound the $k$th moment of $d_{chem}(0,x)$ on $\{0 \sa{B(2n)}x\}$ by an expression like \eqref{eq:momentsum}. The main differences are that the probability on the right-hand side no longer includes the event $\{0 \lra \partial B(n)\}$, and that the prefactor is $\|x\|^{d-2}$ instead of $n^2$. (Here we use \eqref{eq:boxtwopt}.) Fixing an ordering as in \eqref{eq:momentinter} gives rise to an analogous prefactor of $k!$. Finally, we are left to sum an expression of the form
	\[\sum_{z_1, \ldots, z_k} \tau(0, z_1) \tau(z_1, z_2), \ldots, \tau(z_k, x)\ . \]
	
	This sum can be upper-bounded by $C^{k-1} \|x\|^{2k+2 -d}$ using standard methods. Pulling this factor together with the previous ones, we find
	\[\E[d_{chem}(0,x)^k \mid 0 \sa{B(2n)} x] \leq k! C^k \|x\|^{d-2} \|x\|^{2k+2-d} = k! (C \|x\|^2)^k\ ,  \]
	completing the proof.

\section{Proof of upper bound from Theorem \ref{thm:volann} \label{sec:volnotsmall}}
In this section, we prove the inequality ``$\leq$'' from \eqref{eq:volann}. We wish to bound the probability, conditional on $0 \lra \partial B(n)$, that $|\fC_{B(n)}(0)| \leq \lambda n^4$. 
As in the statement of Theorem \ref{thm:volann}, we fix a  value of $\alpha > 3d/2$ and will consider only values of $\lambda > (\log n)^{\alpha} / n^3$.
We set $\kappa=\lceil \lambda^{-1/3} \rceil$; this parameter will be more directly useful than $\lambda$ in our arguments, and  most of our estimates going forward are more naturally phrased in terms of $\kappa$. 
We divide up the annulus $Ann(n/2,n)$ into $\kappa$ annuli
\[A_j=Ann\left(\frac{n}{2}+\frac{nj}{2\kappa}, \frac{n}{2}+\frac{n (j+1)}{2\kappa}\right), \quad j=0,\ldots, \kappa-1,\]
with associated boxes 
\[B_j^1 = B\big(0;\frac{n}{2}+\frac{n j}{ 2\kappa}\big), \quad B_j^2 = B\big(0;\frac{n}{2}+\frac{n(2j + 1) }{ 4\kappa}\big). \]
We also introduce the sub-annulus 
\[A_j \supset A_j' =Ann\left(\frac{n}{2}+\frac{n(2j + 1) }{ 4\kappa}, \frac{n}{2} + \frac{n(4j + 3)}{8 \kappa} \right) = B\left(0; \frac{n}{2} + \frac{nj}{2\kappa} + \frac{3}{8\kappa}\right) \setminus B_j^2\ .  \]
In words, $B_j^1$ is the inner box of $A_j$, $B_j^2$ the box which extends halfway across $A_j$, and $A_j'$ is an annulus which begins halfway across $A_j$ and ends three quarters of the way across $A_j$.
See Figure \ref{fig:VolumeUpperBd1} for an illustration.

\begin{figure}
    \centering
    \includegraphics[width=10cm]{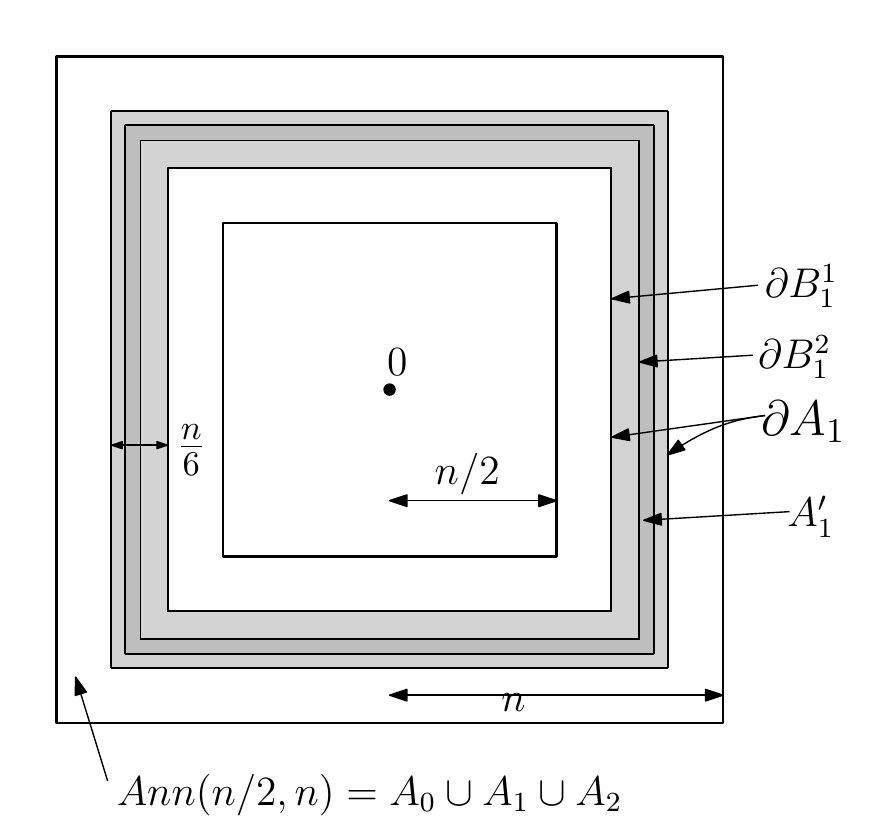}
    \caption{Here $\kappa=3$ and $Ann(n/2,n)$ is divided into 3 annulus $A_0, A_1, A_2$.}
    \label{fig:VolumeUpperBd1}
\end{figure}

We note that $\kappa < C n / (\log n)^{\alpha / 3}$ for some $C = C(\alpha)$. The fact that $\alpha / 3 > d/2$ will be used in the proof of Lemma \ref{lem:Regann}, essentially to ensure that the annuli above are thick enough to recover some independence between the portions of the cluster $\fC_{B(n)}(0)$ in different $A_j$'s. We will need $n$ to be larger than some dimension-dependent constant, guaranteeing in particular $n \geq 64 \kappa$. The smaller values of $n$ are covered by adjusting constants.

The main components of the proof involve showing that, on the event $\{0 \lra \partial B(n)\}$, the vertex set $\fC_{B(n)}(0) \cap A_j$ typically  contains order $(n/\kappa)^4$ vertices, and that $\fC_{B(n)}(0) \cap A_j$ and $\fC_{B(n)}(0) \cap A_k$ have ``enough independence'' for $j \neq k$. This allows us to argue that $|\fC_{B(n)}(0)|$ conditionally stochastically dominates  $c (n/\kappa)^4$ times a sum of independent Bernoulli random variables, so is very likely to be of size at least order $\kappa\times (n/\kappa)^4 \approx \lambda n^4$. 
We note that of course this strategy will only work if our estimates are uniform in $n$ large and in $\lambda > (\log n)^{\alpha} / n^3$, which they will be. \emph{Henceforth, ``uniform in $n$ and $\lambda$ [or $\kappa$]'' means uniform over $n$ larger than some $C = C(d)$ and $\lambda > (\log n)^{\alpha} / n^3$.}

\subsection{New cluster notation}\label{sec:newclust}
For each $j = 0, \ldots, \kappa-1$, our construction will involve exploring $\fC(0) \cap A_j$ in stages. To avoid unmanageably long expressions, we will condense our usual notation for open clusters here; the notation introduced in this section will be in force until 
the end of Section \ref{sec:mainvol}. 
 Because we generally work with a fixed value of $j$, the $j$-dependence is often suppressed in our notation.

We will often write $\fC(x;G)$ instead of $\fC_G(x)$; this improves readability when $G$ is represented by a complicated expression. The symbol $\cC$ will always stand for a vertex subset of $B_j^1$ such that $\prob(\fC(0;B_j^1) = \cC) > 0$. We define the event 
\[\clan := \{\fC(0; B^1_j) = \cC\}\ .\]
When conditioning on $\clan$, we recall that edges within $B_j^1$ on the boundary of $\cC$ are conditionally closed, 
but edges connecting $\mathcal{C}$ to $\Zd \setminus B^1_j$ remain i.i.d.~Bernoulli$(p_c)$ random variables.
On the event $\clan$, we write, for each $x \notin \cC$, the shorthand 
\[\fC^*(x) := \{y \in B_j^2:\, y \sa{B_{j}^2 \setminus \cC} x \} = \{y \in B_j^2: x \in \fC(y; B_j^2 \setminus \cC)\} \ ;\]
in other words, $\fC^*(x)$ is the union of $\fC(x; B_j^2 \setminus \cC)$ with those vertices of $\cC$ which have an open connection to $x$ in $B_j^2$ which touches $\cC$ only at its initial point.

For each $y \in \partial B_j^2$, we fix a neighbor $y' \notin B_j^2$. We write $\fC^{**}(y') := \fC(y'; B_{j+1}^1 \setminus [\cC \cup \fC^*(y)] )$. See Figure \ref{fig:VolumeUpperBd2} for an illustration.
\begin{figure}
    \centering
    \includegraphics[width=10cm, height=8cm]{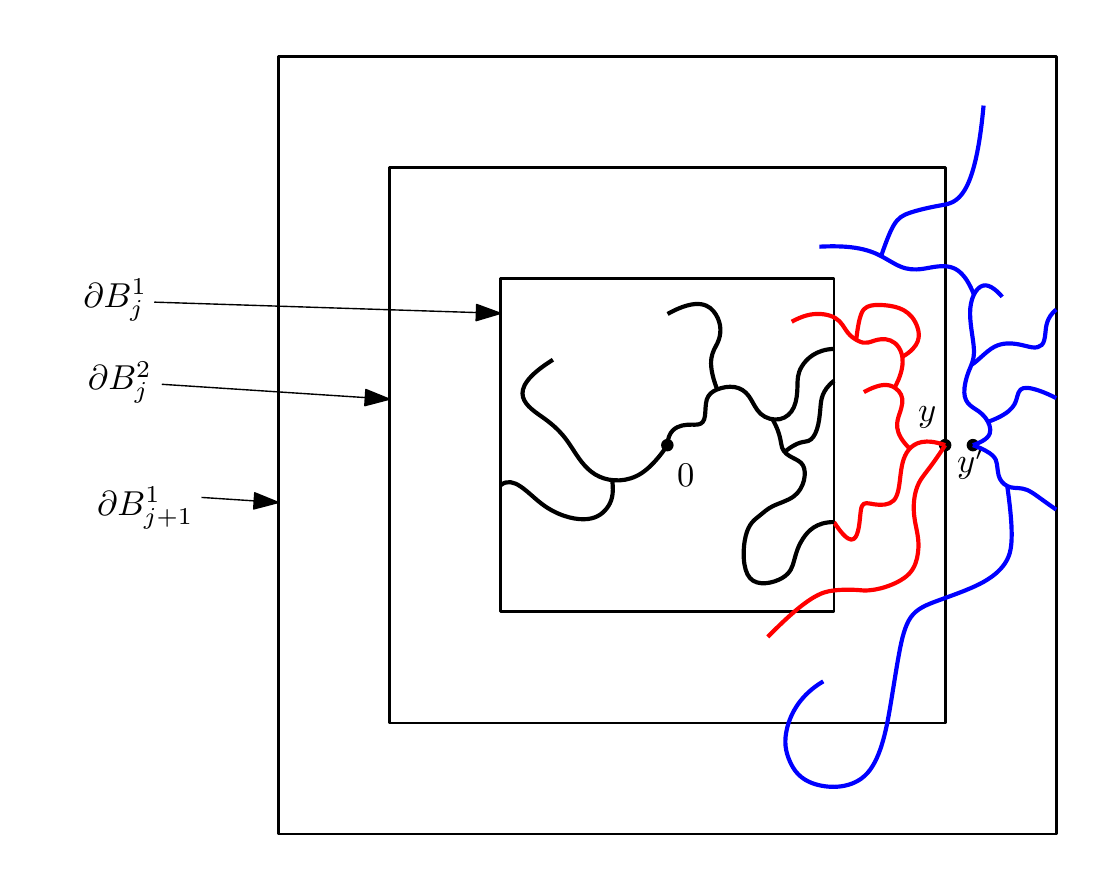}
    \caption{The black, red, and blue lines represent the clusters $\mathcal C,\fC^*(y)$ and $\fC^{**}(y')$ respectively.}
    \label{fig:VolumeUpperBd2}
\end{figure}
The set of vertices of $\partial B_j^2$ through which connections from $\cC$ can proceed will be denoted 
\[\Xi_j^* := \{y \in \partial B_j^2: \, \fC^*(y) \cap \cC \neq \varnothing\}\ , \quad \text{with } X_j^* = |\Xi_j^*|\ . \]

As we mentioned above, much of our proof will revolve around showing $\fC(0;A_j)$ is large conditional on the value of $\fC(0;B_j^1)$. Thus, until Section \ref{sec: nospanning}, we work conditional on $\clan$ for some $\cC$ as above, then derive results which are uniform in $\cC$ which satisfy a further condition. Indeed, by \eqref{eq:onearmprob} and Lemma \ref{lem: bdy-extension}, we can choose a $c_0$ uniform in $n$, $\kappa$, and $j$ such that
\begin{equation}
\label{eq:cCprop0}
\prob(0 \lra \partial B_{j+1}^1 \mid \clan \cap \{X_j^* \leq 2 c_0 (n/\kappa)^2 \}) \leq 1/4\ .
\end{equation}
We will restrict our attention to $\cC$ satisfying the condition
\begin{equation}
\label{eq:cCprop}
\text{for uniform $c_0 > 0$ as in \eqref{eq:cCprop0}},  \quad\E[X_j^* \mid \clan] \geq c_0 (n/\kappa)^2\ .
\end{equation}
As we will argue in Section \ref{sec:mainvol}, when $\cC$ does not satisfy \eqref{eq:cCprop}, the event $\clan$ is not too likely conditional on $\{0 \lra \partial B(n)\}.$

\subsubsection{Regularity}
As usual, we will need some version of cluster regularity to ensure that open connections from $\cC$ can be extended. We would like not to impose very stringent conditions on $\cC$, so that we recover some amount of independence between the portions of the cluster in distinct annuli. This makes the situation somewhat delicate: the open cluster of $\cC$ in $B_j^2$ need not be regular if $\cC$ is not. For instance, if $\cC = B_j^1$, then $\cC$ is typically connected to order $|A_j| (n/\kappa)^{-2}$ vertices of $A_j$, making $\fC(0) \cap A_j$ much larger than four-dimensional. We introduced the sets $\fC^*(y)$ above to mitigate this problem: the $\fC^*(y)$'s will typically be regular, and that will suffice for our purposes.

In all that follows, $\cC$ is an arbitrary set such that $\prob(\clan) > 0$ and such that \eqref{eq:cCprop} holds.
\begin{defin}\label{reg-def-2}
	Suppose $x \in \partial B_{j}^2$. We write 
	\[\cT_s^*(x; \delta) := \{|\fC(x;B_{j+1}^1 \setminus \cC) \cap B(x;s)| < s^{5-\delta} \}\ . \]
	We note that the cluster considered here is the union of $\fC^*(x) \setminus \cC$ with the $\fC^{**}(x')$ clusters attached to it.
	
	Given $\delta > 0$, we say that $x$ is $s$-*-bad if 
	\[\prob(\cT_s^*(x;\delta) \mid \fC(x; B_{j}^2 \setminus \cC)) \leq 1 - \exp(- s^{1/3})\ . \]
		We say that $x$ is $K$-*-irregular if there is an  $s$ with  $K \leq s$ such that $x$ is $2^s$-*-bad.
	\end{defin}
	We will fix the value of $\delta$ in Lemma \ref{lem:easyreg} below, depending only on the dimension $d$ and the value of $\alpha > 3d/2$.  Since we will not alter $\delta$ thereafter, we will generally suppress it in our notation and write $\cT_s^*(x) = \cT_s^*(x;\delta)$. 
    We note that the event $\cT_s^*(x)$ is independent of $\clan$, since we need not examine edges of $\cC$ to determine $\fC(x; B_j^2 \setminus \cC)$ or $\fC(x; B_j^1\setminus \cC)$. In other words,
    \[ \text{for each $\mathcal{D}$, we have } \prob(\cT_s^*(x;\delta) \mid \fC(x; B_{j}^2 \setminus \cC) = \cD) = \prob(\cT_s^*(x;\delta) \mid\clan,\, \{\fC(x; B_{j}^2 \setminus \cC\}) = \cD)\ .\]

Recalling the random set $\Xi_j^*$ and its cardinality $X_j^*$, we write $\Xi_j^{*K}$ for the set of $x \in \Xi_j^*$ which are $K$-*-regular, and let $X_j^{*K} = |\Xi_j^{*K}|$.
The main statement on regularity we need is as follows:
\begin{lem}\label{lem:Regann}
	Let  $\alpha > d/2$ as in the statement of Theorem \ref{thm:volann} be fixed but arbitrary. There exists $K_0 < \infty$ such that, for each $K > K_0$, there exist $c, C = c(K), C(K) > 0$ such that the following holds. Uniformly in $n$ and  $\kappa$ satisfying $\kappa \leq \min\{n/16, \,n / (\log n)^{\alpha}\}$, in $j$, in  $y \in \partial B_{j}^2$ satisfying $\prob(y \in \Xi_j^* \mid \clan) \geq n^{-d}$, and in $\cC$ satisfying  a) $\prob(\clan) > 0$,  b) $\cC\cap \partial B_j^1\neq \emptyset,$ and c) the condition \eqref{eq:cCprop}, we have
	\[\prob\left(y \in \Xi_j^{*K} \mid \clan\right) \geq \frac{1}{2} \,\prob\left(y \in \Xi_j^* \mid \clan\right)\ .  \]
	\end{lem}
\begin{proof}
    The proof is similar to that of Lemma \ref{lem:easyreg}, with some modifications due to the differing geometry and conditioning. We will refer to elements of the earlier proof, avoiding repetition of essentially identical steps.
    
    Let us consider an annulus of the form $Ann(k, C_0 k^{d/2})$ for a large constant $C_0$. Taking a union bound and using  \eqref{eq:onearmprob}, the probability of there being an open crossing of this annulus (that is, an open path connecting $\partial B(k)$ to $\partial B(C_0 k^{d/2})$ is at most
		\[(k+1)^d \pi(k) \leq C C_0^{-2} \leq 1/2 \]
		for $C_0$ chosen large depending only on the lattice. We henceforth take this value of $C_0$ fixed.
		
		We first prove the lemma in the case that $C_0 s^{d/2} \leq n/8\kappa$. This setting is easier to handle because we will  need to examine the cluster of $y$ only within $B(y;C_0 s^{d/2}) \subseteq A_j$ to give a good upper bound on the size of $\fC(y; B_{j+1}^1 \setminus \cC) \cap B(y;s)$. Letting $\delta < 1$ be arbitrary for now, we define the event
		\[ A_s := \{\text{for each $w \in B(y; C_0s^{d/2})$, we have $|\fC_{B(y; C_0 s^{d/2})}(w) \cap B(y;s)| \leq s^{9/2 - \delta/2}$} \}\ ,\]
		We also let
		\[ A'_s := \{\text{there are no more than $s^{1/2 - \delta/2}$ disjoint connections from $B(y;s)$ to $\partial B(y;C_0 s^{d/2})$} \}.\]
	We bound $\prob(A_s)$ using the cluster tail bound of Lemma \ref{lem:aiznew}, and we bound $\prob(A_s')$ using the choice of $C_0$ and the BK inequality \eqref{eqn: BK-reimer}.
	
	We conclude
	\begin{equation}
		\begin{split}
		\label{eq:7preedmund}
		\prob(A_s) &\geq 1 - \exp(-c s^{1/2 - \delta/2});\\
		\prob(A_s') &\geq 1 - (1/2)^{s^{1/2 - \delta/2}} = 1 - \exp(-c s^{1/2 - \delta/2})\ .
		\end{split}
		\end{equation}
	    Similarly to the discussion after \eqref{eq:preedmund}, if there are at most $\ell$ disjoint crossings of $B(y;C_0 s^{d/2}) \setminus B(y;s)$, then 
		\[ \fC(y) \cap B(y;s)\subset \cup_{\cC} [\cC \cap B(y;s)],\] where the union is over at most $ \ell+1$ clusters $\cC$ of $B(y;C_0 s^{d/2})$.
		
		 In particular,   
		 \begin{equation}
		 \label{eq:07edmundfitzgerald}
		\text{on the event $A_s \cap A_s'$,} \quad|\fC(y) \cap B(y;s)| \leq s^{5-\delta}.
		 \end{equation}
		We will show
		\begin{equation}
		\label{eq:7edmundfitzgerald}
		 \prob(A_s \cap A_s' \mid \clan, y \in \Xi_j^*) \geq 1 - \exp(-c s^{1/2-\delta/2})\ .
		 \end{equation} We do this by conditioning on  $\fC(0; B_{j+1}^1 \setminus B(y;C_0 s^{d/2}))$, noting that $A_s$ and $A_s'$ are independent of the status of edges outside $B(y;C_0 s^d)$. We write
		 \begin{align}
		 \prob( \clan, y \in \Xi_j^* \setminus &[A_s \cap A_s']) \leq \sum_{\widehat \cC} \prob(\fC(0; B_{j+1}^1 \setminus B(y;C_0 s^{d/2})) = \widehat \cC) [1 - \prob(A_s \cap A_s')]  \nonumber\\
		 &\leq \exp(-c s^{1/2-\delta/2}) \sum_{\widehat \cC} \prob(\fC(0; B_{j+1}^1 \setminus B(y;C_0 s^{d/2})) = \widehat \cC)\ , \label{eqn: 7C-sum}
		 \end{align}
		 where the sum is over $\widehat \cC$ compatible with the event $\clan \cap \{y \in \Xi_j^*\}$ and we have used \eqref{eq:7preedmund}. Here the ``compatibility'' means exactly that $\clan$ occurs and that $\fC(0; B_j^2)$ contains a neighbor of $B(y; C_0 s^{d/2})$ when  $\fC(0; B_{j+1}^1 \setminus B(y;C_0 s^{d/2})) = \widehat \cC$ (we note that both of these conditions are determined by the value of $\fC(0; B_{j+1}^1 \setminus B(y;C_0 s^{d/2}))$).

		 To show \eqref{eq:7edmundfitzgerald}, we need to compare the sum on the right to $\prob(\clan, y \in \Xi_j^*)$. 
		 This is done by arguments similar to those at \eqref{eq:tinyglue}, here using the fact that $s$ is small enough to ensure $B(y; C_0 s^{d/2}) \cap \cC = \varnothing$. Independence and Lemma \ref{lem:knapriori} imply
		 \begin{align*}
		 \prob\left(\fC(0; B_{j+1}^1 \setminus B(y;C_0 s^{d/2})) = \widehat \cC, \, \clan, y \in \Xi_j^*  \right)\\
		 \geq c \exp(- C \log^2 s) \prob\left(\fC(0; B_{j+1}^1 \setminus B(y;C_0 s^{d/2})) = \widehat \cC \right)\ .
		 \end{align*}
		 Inserting this bound into \eqref{eqn: 7C-sum} and performing the sum over $\widehat \cC$ gives
		 \begin{align*} \prob( \clan, y \in \Xi_j^*, |\fC(y; B_{j+1}^1 \setminus \cC)| > s^{5-\delta}) &\leq C \exp(C \log^2 s) \exp(-c s^{1/2 - \delta/2})\prob(\clan, y \in \Xi_j^*)\\ &\leq C \exp(-c s^{1/2 - \delta/2})\prob(\clan, y \in \Xi_j^*)\ . \end{align*}

		 The above was all derived under the assumption that $C_0 s^{d/2} \leq n / 8\kappa$. We next handle the case that $C_0 s^{d/2} > n/8\kappa$. In this case, we use the fact that $\prob(y \in \Xi_j^* \mid \clan) \geq n^{-d}$ to upper bound
			\begin{align*} 
			&\prob(\{y \in \Xi_j^*\} \setminus \cT_s^*(y) \mid \clan) \\
			\leq~&\frac{\prob\left(\left[ \cT_s^*(y)\right]^c \mid \clan\right) \prob(y \in \Xi_j^* \mid \clan)}{\prob(y \in \Xi_j^* \mid \clan)}\\
			\leq~&C n^d \prob(y \in \Xi_j^* \mid \clan) \prob(|\fC(y;B_{j+1}^1 \setminus \cC) \cap B(y;s)| > s^{5-\delta} \mid \clan)  \\
			\leq~&C n^d \prob(y \in \Xi_j^* \mid \clan) \prob(\left|\fC(y) \cap B(y;s) \right| > s^{5-\delta})\\
			\leq~&C n^d \exp(-c s^{1-\delta}) \prob(y \in \Xi_j^* \mid \clan)\ .
	\end{align*}
	Since $s \geq c (n/\kappa)^{2/d} \geq (\log n)^{1+c}$ by our choice of $\alpha$, for each $\delta > 0$ sufficiently small, the above is at most 
	\[C \exp(-c s^{1-\delta})\ . \]
		 
		 Combining the two cases, \eqref{eq:7edmundfitzgerald} follows for all $s$ as in the statement of the lemma.
		 It remains to argue for the conclusion of the lemma given \eqref{eq:7edmundfitzgerald}. We write
		 \begin{equation}
		 \label{eq:regmark}
		 \begin{split}
		     \prob(\mathcal{T}_s^*(y), \, y \in \Xi_j^*,\, \clan) &= \sum_{\widetilde \cC} \prob(\mathcal{T}_s^*(y),\, \fC(y; B_j^2 \setminus \cC) = \widetilde \cC,\,\clan)\\
		     &\geq (1- e^{-c s^{1/2 - \delta/2}}) \prob(y \in \Xi_j^*,\,\clan)\ , 
		     \end{split}
		 \end{equation}
		 where the sum is over cluster realizations $\widetilde \cC$ such that $\{y \in \Xi_j^*\}$ occurs. The inequality appearing in \eqref{eq:regmark} follows from \eqref{eq:07edmundfitzgerald} and \eqref{eq:7edmundfitzgerald}.
		 
		  We break the sum over $\widetilde \cC$ into two terms depending on whether 
		  $y$ is $s$-*-bad or not on the event $\{\fC(y; B_{j}^2 \setminus \cC) = \widetilde \cC\}$. Performing the sum and applying Definition \ref{reg-def-2}, we can upper bound the sum appearing in \eqref{eq:regmark} by
		  \begin{equation*}
		      \begin{split}
		      (1-e^{-s^{1/3}})\prob\left(\{\prob(\cT_s^*(y) \mid \fC(x; B_{j}^2 \setminus \cC)) \leq 1 - \exp(- s^{1/3})\},y \in \Xi_j^*,\clan \right)\\
		      + \prob\left(\{\prob(\cT_s^*(y) \mid \fC(x; B_{j}^2 \setminus \cC)) > 1 - \exp(- s^{1/3})\},\, y \in \Xi_j^*, \clan \right),
		      \end{split}
		  \end{equation*}
		  so we obtain
		  \begin{equation}
		      \label{eq:regmark2}
		      \begin{split}
		       \prob(\mathcal{T}_s^*(y), \, y \in \Xi_j^*,\, \clan) \leq &\prob(y \in \Xi_j^*, \clan)\\ - e^{-s^{1/3}}&\prob\left(\{\prob(\cT_s^*(y) \mid \fC(x; B_{j}^2 \setminus \cC)) \leq 1 - \exp(- s^{1/3})\},\, y \in \Xi_j^*,\clan \right)\ .
		       \end{split}
		  \end{equation}

		 Comparing \eqref{eq:regmark2} with the lower bound of \eqref{eq:regmark}, we see that there is an $s_0 = s_0(d, a)$ such that, for all $s > s_0$,
		 \begin{equation}
		     \label{eq:regmark3}
		     \begin{split}
		     \prob\left(\{\prob(\cT_s^*(y) \mid \fC(x; B_{j}^2 \setminus \cC)) \leq 1 - \exp(- s^{1/3})\}\mid \{y \in \Xi_j^*\},\,\clan \right) \leq\exp(-s^{1/2 - \delta})\ .
		     \end{split}
		 \end{equation}
		 We sum over $s \geq K$ to obtain the bound
		 \[ \prob(y \notin \Xi_j^{*K} \mid \{y \in \Xi_j^*\}, \clan) \leq C \exp(-c K^{1/3})\ .\]
		 Choosing $K_0$ large enough that the right-hand side of the last display is smaller than $1/2$ when $K > K_0$ and multiplying both sides of that display by $\prob(y \in \Xi_j^{*} \mid \clan)$ completes the proof.
		
\end{proof}

\subsection{$\fC(0; B_{j+1}) \cap A_j$ is large with positive probability\label{sec:annlarge}}
We use Lemma \ref{lem:Regann} to argue that $\fC(0; B_{j+1}) \cap A_j$ is frequently large on the event $\clan$. Formally, we prove the following intermediate lemma, which furthermore decouples $\fC(0;B_j) \cap A_{j}$ from $\fC(0; B_i), \, i < j$:
\begin{lem}
    \label{lem:byann}
    There exists $c_v > 0$ such that the following holds uniformly in $n$, in $j$, and $\kappa$.
    For each $\cC$ satisfying \eqref{eq:cCprop}, we have
    \[\prob(|\fC(0; B_{j+1}^1) \cap A_j| > c_v (n/\kappa)^4 \mid \clan) \geq c_v\ . \]
\end{lem}
The proof of Lemma \ref{lem:byann} is based on the second moment method. In this section, we define and prove facts about events $\mathcal{A}(y,z)$ on which the second moment argument will be based. In Section \ref{sec:ann1st}, we prove the necessary first moment bounds; in Section \ref{sec:ann2nd} we prove the second moment bound and complete the proof of the lemma.

Recall that for each $y \in \partial B_j^2$, we have chosen  a deterministic neighbor $y' \in B_{j+1}^{1} \setminus B_j^2$. For each such edge $\{y, y'\}$, and for each $z \in  A_j'$, we define
\begin{equation}
    \begin{split}
&\mathcal{A}(y, z, y', \cC)\\
=&~\mathcal{A}(y, z) \\
= &~\clan \cap \{y \in \Xi_j^*\} \cap \left\{\begin{array}{c}
\{y, y'\} \text{ is open  and pivotal for $y \sa{B_{j+1}^{1} \setminus \cC} z$},\\ \text{ and $\fC^{**}(y')$ contains no vertices adjacent to $B_j^1$} \end{array} \right\}. \label{eq:calA}
\end{split}
\end{equation}
We usually omit $\cC$ from the notation because, as we have noted, all our bounds will be uniform in $\cC$.

We will wish to argue that $\fC(0; B_{j+1}^1) \cap A_j$ is at least the number of pairs $(y,z)$ for which $\mathcal{A}(y,z)$ occurs. For this, we will use the following proposition:
\begin{prop}\label{prop:forcalA}
	Suppose that $y_1\neq y_2 \in \partial B_j^2$. Then for each $z \in A_j'$, we have $\mathcal{A}(y,z) \subseteq \{z \sa{B^1_{j+1}} 0\}$. Moreover, for each pair $z_1, z_2 \in A_j'$,
	\begin{equation}
	    \label{eq:calAconseq}
	\mathcal{A}(y_1, z_1) \cap \mathcal{A}(y_2, z_2) \subseteq \{\fC^{**}(y_1') \cap [\fC^*(y_2) \cup \fC^{**}(y_2')] = \varnothing\}\ ,
	\end{equation}
	and so (taking $z = z_1 = z_2$) we have $\mathcal{A}(y_1, z) \cap \mathcal{A}(y_2,z) = \varnothing$.  
	\end{prop}
\begin{proof}
	We first prove the containment $\mathcal{A}(y,z) \subseteq \{z \sa{B^2_{j+1}}0\}$, which is relatively easy. On $\mathcal{A}(y,z)$, there is an open connection from $y$ to $\cC$ by assumption, and ( by the definition of $\clan$) thus $\fC(0; B_j^2) \ni y$. Then by the openness of $\{y, y'\}$, we have $y' \in \fC(0;B_{j+1})$; finally, this openness and the pivotality of this edge ensure $y' \sa{B_{j+1}} z$, completing this part of the proof.
	
	We will argue by contradiction for \eqref{eq:calAconseq}: we assume that $\omega \in \mathcal{A}(y_1, z_1) \cap \mathcal{A}(y_2, z_2) \cap \{\fC^{**}(y_1') \cap [\fC^*(y_2) \cup \fC^{**}(y_2')] \neq \varnothing\}$ and then show $\omega$ has contradictory properties. We further decompose this event and break the proof into two cases.
	\paragraph{Case 1: $\omega \in \{\fC^*(y_1) = \fC^*(y_2)\}.$}
 We assume first that $\omega$ has the additional property that, in $\omega$, the clusters $\fC^*(y_1)$ and $\fC^*(y_2)$ are identical. In this case, by definition we have that $\fC^{**}(y_1') \cap \fC^{*}(y_1) = \varnothing,$ and therefore $\fC^{**}(y_1') \cap \fC^*(y_2) = \varnothing.$
 To show $\fC^{**}(y_1') \cap \fC^{**}(y_2') = \varnothing$, we suppose that $\fC^{**}(y_1') \cap \fC^{**}(y_2') \neq \varnothing$, which implies (again using $\fC^*(y_1) = \fC^*(y_2)$) that  $\fC^{**}(y_1') = \fC^{**}(y_2')$. Let $\gamma$ be the concatenation of a) an open path in $\fC^{**}(y_2')$ from $y_2'$ to $z_1$, b) the edge $\{y_2, y_2'\}$, and c) an open path in $\fC^*(y_2)$ from $y_2$ to $y_1$. By construction, the path $\gamma$ avoids $\{y_1, y_1'\}$.
But since $\omega \in \mathcal{A}_1(y_1, z_1)$,  the pivotal edge $\{y_1, y_1'\}$ must be in $\gamma$, a contradiction.
	
	\paragraph{Case 2: $\omega \in \{\fC^*(y_1) \neq \fC^*(y_2)\}.$} We suppose instead that $\fC^*(y_1)$ and $\fC^*(y_2)$ are distinct (and hence $\fC^*(y_1) \cap \fC^*(y_2)$ may contain only vertices of $\cC$) in outcome $\omega.$ 
	We first show that $\fC^{**}(y_1') \cap \fC^*(y_2) = \varnothing$ by assuming these clusters instead had nonempty intersection and deriving a contradiction. Under this assumption, let $\gamma$ be a path in $\fC^{**}(y_1')$ from $y_1'$ to a vertex $\tilde w \in \fC^*(y_2).$
	
	We produce an open path by appending the segment of $\gamma$ from $y_1'$ to $\tilde w$ to a path lying entirely in $\fC^*(y_2) \cap A_j$ from $\tilde w$ to a vertex adjacent to $\cC$. This is a path in $B_{j+1}$ from $y_1'$ to a vertex adjacent to $B_j^1$. It avoids $\fC^*(y_1)$ because $\gamma$ avoids $\fC^*(y_1)$ and because $\fC^*(y_1) \cap \fC^*(y_2) \cap A_j = \varnothing$. In particular, this path guarantees that $\fC^{**}(y_1')$ contains a vertex adjacent to $B_j^1$, a contradiction. This shows $\fC^{**}(y_1') \cap \fC^*(y_2) = \varnothing$ (and similarly $\fC^{**}(y_2') \cap \fC^*(y_1) = \varnothing$).
	
	We again show $\fC^{**}(y_1') \cap \fC^{**}(y_2') = \varnothing$ by assuming the contrary and deriving a contradiction. Under our assumption, we choose a vertex $w \in \fC^{**}(y_1') \cap \fC^{**}(y_2')$ and let $\gamma_i$ be a path in $\fC^{**}(y_i')$ from $y_i'$ to $w$ (for $i = 1, 2$). Appending $\gamma_1$ to $\gamma_2$, we produce an open path which (by the previous paragraph) lies outside $\fC^*(y_1) \cup \fC^*(y_2)$ and connects $y_1'$ to $y_2'$. Adjoining to this the open edge $\{y_2', y_2\}$ and a path in $\fC^*(y_2)$ from $y_2$ to a neighbor of $\cC$, we see again that $\fC^{**}(y_1')$ contains a vertex adjacent to $B_j^1$, a contradiction.
	
	\paragraph{Proof of final claim.} Finally, to show $\mathcal{A}(y_1, z) \cap \mathcal{A}(y_2, z) = \varnothing,$ we note that on $\mathcal{A}(y_i, z)$, we have $z \in \fC^{**}(y_i)$, then we apply \eqref{eq:calAconseq}.
	\end{proof}
	
	As we have discussed, we wish to lower bound the size of $\fC(0; B_{j+1}) \cap A_j$ on $\clan$. In fact, it helps (see \eqref{eq:Znsecond} below) to consider a portion of this cluster whose connections in $A_j$ ``do not wander too far'', and which have a pivotal edge touching $\partial B_j^2$ for their connection to $\cC$:
\begin{equation}\label{eq:Zndef}
Z_j := \{(y, z): y\in \partial B_j^2, \, z \in A_j' \cap B(y; n/16 \kappa), \, \text{and } \mathcal{A}(y,z) \text{ occurs} \}\ . 
\end{equation}
Proposition \ref{prop:forcalA} immediately implies the following corollary.
\begin{cor}\label{cor:Zloop}
	\(\text{On $\clan$}, \quad \left| \fC(0;B_{j+1}) \cap A_j \right| \geq | Z_j |\ . \)
\end{cor}
We will use Corollary \ref{cor:Zloop} to show Theorem \ref{thm:volann}. As already discussed, in the next two sections we use the second moment method to show that $|Z_j|$ is often of order $(n/\kappa)^4$ conditional on $\clan$. Using Corollary \ref{cor:Zloop}, we see that $\fC(0; B_{j+1}^1) \cap A_j$ has uniformly positive probability to be of order $(n/\kappa)^4$. In Section \ref{sec:ann2nd}, we use this fact to show that in fact with high probability $\fC(0; B_{j+1}^1) \cap A_j$ is of order $(n/\kappa)^4$ simultaneously for at least $c \kappa$ values of $j$ and complete the proof of Theorem \ref{thm:volann}.

\subsection{Bounding the first moment of $|Z_j|$\label{sec:ann1st}}
We now have the following result allowing us to extend connections from $\mathcal{C}$ to points $z$ in the annulus $A_j^2$, which we will subsequently use to lower bound the first moment of $|Z_j|$. The $K_1$ appearing here depends only on the lattice $\Zd$ under consideration and the value of $\alpha$ as in Theorem \ref{thm:volann}.
\begin{lem}\label{lem:pivsubst}
	There is a $K_1 > K_0$ such that the following holds. For each $K > K_1$, there exists a $c > 0$ such that, uniformly in $n$ and $\kappa$ satisfying the additional assumption $n/\kappa \geq 32K$, for all $j$, all $\cC \subseteq B_j^1$ such that \eqref{eq:cCprop} holds, all $y \in \partial B_j^2$, and all $M$ satisfying $2K \leq M \leq n/16\kappa$,
	\begin{equation}
	\label{eq:pivsubst}
	  \sum_{z \in B(y;M) \cap A_j' } \prob(\{y,y'\}\text{ open, pivotal for }  y \sa{B_{j+1}^1 \setminus \cC} z \mid \clan, \, y \in \Xi_j^{*K}) \geq c M^2\ .  
	 \end{equation}
	\end{lem}
\begin{proof}
    The proof uses a variant of the Kozma-Nachmias cluster extension method \cite[Theorem 2]{KN09}, using the notion of regularity we have introduced for this particular case, which poses somewhat different issues than the extension arguments of Proposition \ref{prop:yprop} above. We provide the details for the reader's convenience.
    
    We define the events
    \begin{align*}
        \mathcal{E}_1(y)&=\{\clan, y\in \Xi_j^{*K}\},\\
        \mathcal{E}_2(y,y^*,z)&=\{y^* \sa{B^1_{j+1}\setminus [\mathcal{C} \cup \fC^*(y)] } z\},\\
        \mathcal{E}_3(y,y^*)&=\{ \fC(y; B_{j+1}^1 \setminus \cC))\cap \fC(y^*;B^1_{j+1} \setminus \cC)=\varnothing\}.
    \end{align*}
   Defining
    \[\Delta(y)=B(y;K)\setminus (B^2_j+B(0,K/2)),\]
    we show that there is a $c > 0$ such that, for each $K$ larger than some constant $K_3 > K_0$ (depending only on the lattice), given  values of other parameters as in the statement of the lemma, there is a $y^* \in \Delta(y)$ with
    \begin{equation}
        \label{eq:ystar}
    \sum_{z \in B(y;M) \cap A_j'}\mathbb{P}(\mathcal{E}_1(y)\cap \mathcal{E}_2(y,y^*,z)\cap \mathcal{E}_3(y,y^*))\ge cM^2\mathbb{P}(\mathcal{E}_1(y)).
    \end{equation}
    
    We first show the existence of a $K_2 > K_0$ and a constant $c$ uniform in $K > K_2$ as well as in $n$, $\kappa$, $\cC$,$j$, and $y$ as in the statement of the lemma, and in \emph{all}
    $y^*\in \Delta(y)$ such that
    \begin{equation}
        \label{eq:allystar}
   \sum_{z \in B(y;M)\cap A_j'} \mathbb{P}(\mathcal{E}_1(y)\cap \mathcal{E}_2(y, y^*, z))\ge cM^2\mathbb{P}(\mathcal{E}_1(y)).
    \end{equation}
    Summing over $\mathcal{D}$ consistent with the event $\{\fC^*(y)=\mathcal{D}, y\in \Xi^{*K}_j\}$, we have
    \[\mathbb{P}(y^*\sa{B^1_{j+1}\setminus[\mathcal{C} \cup \fC^*(y)]} z, \clan, y\in \Xi^{*K}_j)=\sum_{\mathcal{D}}\mathbb{P}(y^*\sa{B^1_{j+1}\setminus [\mathcal{C} \cup \fC^*(y)]} z \mid \clan, \fC^*(y)=\mathcal{D})\mathbb{P}(\clan, \fC^*(y)=\mathcal{D}).\]
    For the conditional probability, we have the lower bound 
    \begin{align*}
        &\mathbb{P}(y^* \sa{B^1_{j+1}\setminus [\mathcal{C} \cup \fC^*(y)]} z\mid \clan, \fC^*(y)=\mathcal{D})\\
        \geq~& \mathbb{P}(y^* \sa{A_j\setminus \mathcal{D}} z)\\
        \geq~& \mathbb{P}(y^* \sa{A_j} z)-\sum_{\zeta \in \mathcal{D}}\mathbb{P}(\zeta \lra y^* \circ \zeta \lra z)\\
        \ge~& \mathbb{P}(y^* \sa{A_j} z)-C\sum_{\zeta \in \mathcal{D}}\mathbb{P}(\zeta \lra y^*)\|\zeta - z\|^{2-d}.  
    \end{align*}
    We have used the BK inequality and \eqref{eq:onearmprob} in the last step.
    Summing over $z$ using \eqref{eq:boxtwopt} , we obtain the lower bound
    \begin{equation}
        \label{eq:toAj}
  cM^2-CM^2\sum_{\zeta\in\mathcal{D}}\|\zeta- y^*\|^{2-d}.
    \end{equation}
    We note that if $\zeta\in B^2_j$, we have $\|\zeta-y^*\|\ge K/2$. So the sum appearing in the second term is bounded by
    \begin{align}
        &C\sum_{k\ge \log_2 (K/2)}|\mathcal{D}\cap B(y^*,2^k)|2^{(2-d)k} \nonumber \\
        \le&C\sum_{k\ge \log_2 (K/2)}|\mathcal{D} \cap B(y,2^{k+1})|2^{(2-d)k}. \label{eqn: tk2}
    \end{align}
For $\mathcal{C}$, $\mathcal{D}$ consistent with $\{y\in \Xi^{*K}_j\}$, we have
\[|\mathcal{D}\cap B(y,2^{k+1})|\le C2^{(5-\delta)k}\ .\]
Applying this estimate in \eqref{eqn: tk2}, we obtain 
\begin{align*}
\sum_{\zeta\in \mathcal{D}}\|\zeta-y^*\|^{2-d}&\le \sum_{k\ge \log (K/2)} 2^{(7 - d - \delta)k}\\
&\le CK^{7-d-\delta}.
\end{align*}
Since $d > 6$, we can make the second term of \eqref{eq:toAj} negligible for each $K$ larger than some uniform $K_2$. We obtain \eqref{eq:allystar}.

Next, we show the existence of a $K_1 > K_2$ and a $c > 0$ uniform in $n$, $\kappa$, $m$, $\cC$, $K > K_1$, and $y$ with 
\begin{equation}\label{eqn: green}
\begin{split}
\frac{1}{|\Delta(y)|}\sum_{y^*\in \Delta(y)} \sum_{z \in B(y;M) \cap A_j^2}&\mathbb{P}(\mathcal{E}_1(y)\cap \mathcal{E}_2(y,y^*,z)\setminus \mathcal{E}_3(y,y^*))\\
&\le CM^2  K^{7-d-\delta}\mathbb{P}(\mathcal{E}_1(y)).
\end{split}
\end{equation}
Choosing the value of $y^*$ which minimizes the inner sum of \eqref{eqn: green} and combining it with \eqref{eq:allystar} clearly implies \eqref{eq:ystar}.

The event on the left-hand side of \eqref{eqn: green} implies the existence of a vertex $\zeta \in B_{j+1}^1 \setminus \cC$ such that
\[\{\mathcal{E}_1(y), y \sa{B^1_{j+1}\setminus \mathcal{C}} \zeta\}\circ \{\zeta\lra y^*\}\circ \{\zeta\lra z\}.\]
Using the BK inequality, we have the upper bound:
\begin{align}
&\frac{1}{|\Delta(y)|}\sum_{y^*\in \Delta(y)} \sum_{z \in B(y;M) \cap A_j'} \sum_{\zeta}\mathbb{P}(\mathcal{E}_1(y), \fC^*(y)\sa{B^1_{j+1}\setminus \mathcal{C}} \zeta)\mathbb{P}(y^*\lra \zeta)\mathbb{P}(\zeta\lra z)\nonumber\\
\leq &~\frac{C M^2}{|\Delta(y)|}\sum_{y^*\in \Delta(y)}\sum_{\zeta}\mathbb{P}(\mathcal{E}_1(y),\{\fC^*(y)\sa{B^1_{j+1}\setminus \mathcal{C}} \zeta\})\,\| \zeta-y^*\|^{2-d}. \label{eq:splitzeta}
\end{align}

We break up the sum according to the distance $\|\zeta-y^*\|$ 
 and the value $\mathcal{D}$ of $\fC^*(y)$ (consistent with the event $\cE_1(y)$).  Thus \eqref{eq:splitzeta} is bounded by
\begin{align}
&\frac{C M^2}{|\Delta(y)|}\sum_{y^*\in \Delta(y)}\sum_{k>k_0} \sum_{\mathcal{D}}\sum_{\zeta\in Ann(y^*; 2^{k-1}, 2^k) }\mathbb{P}[\{\zeta\sa{B^1_{j+1}\setminus \mathcal{C}} \mathcal{D}\},\clan, \fC^*(y)=\mathcal{D}]\ \| \zeta-y^*\|^{2-d} \ .
\label{eqn: buzzard0}
\end{align}

We split the sum according to whether $k> k_0$ or $k \leq k_0$, where $k_0=\log_2(K/2)$. We first bound the $k > k_0$ terms; the inner sums over $k$, $\cD$, and $\zeta$ of \eqref{eqn: buzzard0} are bounded by \begin{align}
\le &  C \sum_{k>k_0} \sum_{\mathcal{D}}\mathbb{E}[|\mathcal{B}_k(y^*)|\mid \clan, \fC^*(y)=\mathcal{D}]\mathbb{P}(\clan, \fC^*(y)=\mathcal{D}) 2^{(2-d)k}. \label{eqn: buzzard}
\end{align}
Here we have introduced, for $w$ an arbitrary vertex, the notation
\[\mathcal{B}_k(w)=\{\fC(y; B_{j+1}^1 \setminus \cC)\cap B(w; 2^k)\}.\]

We estimate the conditional expectation
\[\mathbb{E}[|\mathcal{B}_k(y^*)|\mid \clan, \fC^*(y)=\mathcal{D}]\]
uniformly in $y^*$ using the inclusion
\[\mathcal{B}_k(y^*)\subset \mathcal{B}_{k+1}(y),\]
which is implied by $y^*\in \Delta(y)$.
 If $y\in \Xi^{*K}_j$, the definition of $K^*$-regularity implies 
\[\mathbb{E}[|\mathcal{B}_{k+1}(y)| \mathbf{1}_{\neg\mathcal{T}^*_{k+1}(y)}\mid \clan, \fC^*(y)=\mathcal{D}]\le 2^{(k+2)d}e^{-2^{k/3}},\]
and
\[\mathbb{E}[|\mathcal{B}_{k+1}(y)| \mathbf{1}_{\mathcal{T}^*_{k+1}(y)}\mid \clan, \fC^*(y)=\mathcal{D}]\le 2^{(5-\delta)k+5}.\]
Thus, we find
\begin{equation} \label{eqn: willard}
    \mathbb{E}[|\mathcal{B}_k(y^*)|\mid \clan, \fC^*(y)=\mathcal{D}]\le C2^{(5-\delta)}, \quad k> k_0.
\end{equation}
Applying this bound, we see that \eqref{eqn: buzzard} is at most
\begin{equation}
    \label{eqn: buzzard2}
    C \sum_{k > k_0} 2^{(7 - d - \delta)k} \leq C K_0^{7-d-\delta}\ .
\end{equation}

We now turn to the $k \leq k_0$ terms of \eqref{eqn: buzzard0}, for which it is useful to first perform the $y^*$ sum. Indeed, we have uniformly in $\zeta$ and $y$
\[\sum_{y^* \in \Delta(y)} \|\zeta - y^*\|^{2-d} \leq C K^2\ . \]
Applying this last display, we see the $k \leq k_0$ terms of \eqref{eqn: buzzard0} are bounded above by
\begin{align*}
     C M^2 K^{2-d} \sum_{\mathcal{D}}\mathbb{E}[|\mathcal{B}_{K+2}(y)| \mid \clan, \fC^*(y)=\mathcal{D}]\ \prob(\clan, \fC^*(y)=\mathcal{D})
     \leq C M^2 K^{7-d-\delta} \ ,
\end{align*}
where we have bounded the expectation as in the estimates producing \eqref{eqn: willard}.
Pulling the last display together with \eqref{eqn: buzzard2}, we have shown \eqref{eqn: green}.
Finally, combining \eqref{eqn: green} with \eqref{eq:allystar} and assuming $K$ is large, we see that \eqref{eq:ystar} holds.

To obtain \eqref{eq:pivsubst} from \eqref{eq:ystar}, we use an edge modification argument inside  a box of diameter order $K$, again similar to the one appearing in the proof of Lemma \ref{lem:toinduct} or \cite[Lemma 5.1]{KN09}. The edge modification shows
\begin{align*}
&\prob(\{y,y'\}\text{ open, pivotal for } \fC^*(y) \sa{B_{j+1}^1 \setminus \cC} z \mid \clan, \, y \in \Xi_j^{*K})\\
\ge &~c(K)\mathbb{P}( \mathcal{E}_2(y,y^*,z)\cap \mathcal{E}_3(y,y^*) \mid \cE_1(y)),
\end{align*}
and the proof of the lemma follows using \eqref{eq:ystar}.
\end{proof}	
	
Our next goal is to slightly adapt the content of Lemma \ref{lem:pivsubst} to instead involve the events $\mathcal{A}(y,z)$, which can be used in the application of of Corollary \ref{cor:Zloop}:
\begin{lem}
	\label{lem:pivsubst2}
	 For each $K > K_1$ (the constant from Lemma \ref{lem:pivsubst}), the following holds. There exists a $c > 0$ such that, for all $n,$ all $\kappa$, for all $j$, and for all $\cC$ such that \eqref{eq:cCprop} holds
	\begin{equation}
	\label{eq:pivsubst2}
	\E[|Z_j| \mid \clan] = \sum_{\substack{y \in \partial B_j^2\\z \in A_j' \cap B(y; n/16 \kappa)}} \prob(\mathcal{A}(y,z)\mid \clan) \geq c (n/\kappa)^2 \E[X_j^{*K} \mid \clan]\ . 
	\end{equation}
	\end{lem}
	\begin{proof}
	We express the left-hand side of \eqref{eq:pivsubst2} in the form
	\begin{equation}
	\label{eqn: clementine}
	\sum_{y,z}\prob(\mathcal{A}(y,z)\mid \clan) = \sum_{y,z}\prob(\mathcal{A}(y,z)\mid \clan, y \in \Xi_j^{*K})\, \prob(y \in \Xi_j^{*K}\mid \clan)\ .
	\end{equation}
	 We will lower bound the conditional probability of $\mathcal{A}(y,z)$ on the right-hand side using  Lemma \ref{lem:pivsubst} --- the missing ingredient is to show that the connection from $y$ to $z$ in the event from \eqref{eq:pivsubst} does not make a connection from $y'$ to neighbors of $B_j^1$ too likely. 
	To do this, we must restrict the sum over $z$ somewhat --- it will be easier to rule out such loops back into $B_j^1$ for $z$ comparatively near to $y$. Let us introduce a parameter $0 < a < 1/16$, to be chosen small but fixed relative to $n$, $\lambda$, $j$, $y$, and $\cC$. Indeed, the value of $a$ will be chosen based on the constant appearing in \eqref{eq:pivsubst} and the constants in the one-arm probability bound \eqref{eq:onearmprob}. On $\clan$, we define the random set
	\begin{equation}
	\label{eq:smallcloop}
	Y(a,y):= \left\{z \in B(y;an/\kappa) \cap A'_j:\, \{y,y'\}\text{ open, pivotal for } y \sa{B_{j+1}^1 \setminus \cC} z \right\}\ .
	\end{equation}
	Applying \eqref{eq:pivsubst} with $an$ playing the role of $M$, we find a $c = c(K) > 0$ such that,
	\begin{equation}\label{eq:EYlb} \text{for each $n,\cC,y,a,j, \kappa$ as in \eqref{eq:pivsubst}}, \quad \E[|Y(a,y)| \mid \clan, y \in \Xi_j^{*K}] \geq c a^2n^2 \ . \end{equation}
	
     The event $\clan \cap\{y \in \Xi_j^{*K}  \} \cap \{z \in Y(a,y)\}\setminus \mathcal{A}(y,z)$ implies that one of the following two events occurs:
	\begin{itemize}
		\item
		$L_1 := \bigcup_{\zeta \in \partial B(y; n/8\kappa)} \{\zeta \sa{B(y;n/8 \kappa) \setminus \fC^*(y)} y'\} \circ \{\zeta \sa{\Zd \setminus [\cC \cup \fC^*(y)]} z \};  $
		\item
		$L_2:=\bigcup_{\zeta \in  B(y; n/8 \kappa)} \{\zeta \sa{B(y;n/8 \kappa)\setminus \fC^*(y)} y'\} \circ \{\zeta \sa{B(y;n/8 \kappa)\setminus \fC^*(y)} z \} \circ \{\zeta \sa{\Zd \setminus \fC^*(y)} \partial B(y;3n/16 \kappa) \}.  $
		\end{itemize}
		That is, either $y'$ is connected to $z$ (off $\fC^*(y)$) by a path exiting the box $B(y; n/8\kappa)$, or $y'$ and $z$ are connected within this box and are connected to the boundary of a slightly larger box by a further open path.
	In particular, for each $y, z$:
	\begin{equation}
	\label{eq:AtoA}
	\begin{split}
	& \prob(\mathcal{A}(y,z) \mid \clan), y \in \Xi_j^{*K})\\
	\geq & ~\prob(z \in Y(a,y) \mid \clan, y \in \Xi_j^{*K} ) - \prob(L_1 \cup L_2 \mid \clan, y \in \Xi_j^{*K})\ .
	\end{split}
	\end{equation}
	
    We can decompose the event $\clan \cap \{y \in \Xi_j^{*K}\}$ into a union of events of the form $\clan \cap \{ \fC^*(y) = \cD\}$; to upper-bound the probability of $L_1$, we thus provide an upper bound on $\prob(L_1 \mid \clan, \fC^*(y) = \cD)$  uniform in realizations $\cD$ of $\fC^*(y)$ such that $y \in \Xi_j^{*K}$. Using the half-space two-point function bound \eqref{eqn: hs-tp}, we find
	\begin{align*}
	\prob(L_1 \mid \clan, \fC^*(y) = \cD) \leq C |\partial B(y; n/8\kappa)| (n/\kappa)^{1-d} (n/\kappa)^{2-d} \leq C (n/\kappa)^{2-d}\ ,
	\end{align*}
	where the constant $C$ is uniform in the same parameters as \eqref{eq:EYlb}.
	Similarly, we bound the probability of $L_2$ using the two-point function and the value of the (full-space) one-arm exponent \eqref{eq:onearmprob}:
	\begin{align*}
	\prob(L_2 \mid \clan, \fC^*(y) = \cD) \leq C (n/\kappa)^{-2} \sum_{\zeta \in B(y;n/8\kappa)} \|\zeta - y\|^{2-d} \|\zeta - z\|^{2-d} = C (n/\kappa)^{2-d}\ .
	\end{align*}

	Applying the last two displays in \eqref{eq:AtoA} and using \eqref{eq:EYlb}, we see
	\[\sum_{ z \in A_j' \cap B(y;an/ \kappa)} \prob(\mathcal{A}(y,z) \mid \clan, y \in \Xi_j^{*K}) \geq c a^2 (n/\kappa)^2 - C a^d(n/\kappa)^2.\]
	
	Choosing $a$ small relative to the uniform constants in the last display (but fixed relative to all other parameters) and summing over $y\in \partial B_j^2$ in \eqref{eqn: clementine}, the right-hand side is at least $c (n/\kappa)^2 \E[X_{n}^{*K} \mid \clan]$ uniform in $K$ large but fixed relative to $n$, in $n$, and in $\cC$. This completes the proof.
	\end{proof}
	
	\begin{cor}\label{cor:Zmo}
 There exists a $c > 0$ uniform in the same parameters as Lemma \ref{lem:pivsubst} such that
	\begin{equation*}\E[\,|Z_n|\, \mid \clan] \geq c (n/\kappa)^4\ . \end{equation*}
	\end{cor}
\begin{proof}
    By Lemma \ref{lem:pivsubst2}, it suffices to show
    \begin{equation}
        \label{eq:toregX}
        \E[X_j^{*K} \mid \clan] \geq c \E[X_j^* \mid \clan] \geq c(n/\kappa)^2
    \end{equation}
    holds uniformly in the same parameters as Lemma \ref{lem:pivsubst}.  The second inequality follows from \eqref{eq:cCprop}; it remains to show the first.
    
    We write
    \begin{align*}
        \E[X_j^{*K}\mid \clan] &= \sum_{\substack{y \in \partial B_j^2\\ \prob(y \in \Xi_j^* \mid \clan) \leq n^{-d}}} \prob(y \in \Xi_j^{*K} \mid \clan) + \sum_{\substack{y \in \partial B_j^2\\ \prob(y \in \Xi_j^* \mid \clan) > n^{-d}}} \prob(y \in \Xi_j^{*K} \mid \clan)\\
        &\geq \frac{1}{2}\sum_{\substack{y \in \partial B_j^2\\ \prob(y \in \Xi_j^* \mid \clan) > n^{-d}}} \prob(y \in \Xi_j^{*} \mid \clan)\\
        &\geq \frac{1}{2}\sum_{\substack{y \in \partial B_j^2}} \prob(y \in \Xi_j^{*} \mid \clan) -  \frac{C_1}{n}  = \frac{1}{2} \E[X_j^* \mid \clan] - \frac{C_1}{n} \ ,
    \end{align*}
    where in the second line we have used Lemma \ref{lem:Regann}. The corollary follows by applying \eqref{eq:cCprop}.
	\end{proof}
	
	\subsection{Bounding the second moment of $|Z_j|$}\label{sec:ann2nd}
We produce an upper bound on the second moment of $|Z_j|$ complementing that of Corollary \ref{cor:Zmo}: 
\begin{prop}\label{prop:Z2ndmom}
There is a constant $C$ such that the following holds uniformly in $n$, in $j$, and in $\cC$ satisfying \eqref{eq:cCprop}:
\[\E[\,|Z_j|^2 \mid \clan] \leq C \E[\,|Z_j|\, \mid \clan]^2\ . \]
	\end{prop}
\begin{proof}
	
	We write
	\begin{align}
	 \E[ |Z_j|^2 \mid \clan] = &\sum_{y_1, y_2 \in \partial B_j^1}\Large[ \prob(y_1, y_2 \in \Xi_j \mid \clan)\nonumber\\
	 &\sum_{\substack{z_1 \in A_j' \cap B(y_1; n/16 \kappa)\\z_2 \in A_j' \cap B(y_2; n/16 \kappa)}} \prob(\mathcal{A}(y_1, z_1) \cap \mathcal{A}(y_2, z_2) \mid \clan \cap \{y_1, y_2 \in \Xi_j\})\large]\ .\label{eq:todiag}
	 	\end{align}
	 	We condition the inner sum further on the value of $\fC^{*}(y_1)$ and $\fC^{*}(y_2)$; an upper bound for the inner sum will follow once we bound
	 	\begin{equation}
	 	    \label{eq:2ndfurther}
	 	    \prob(\mathcal{A}(y_1, z_1) \cap \mathcal{A}(y_2, z_2) \mid \clan \cap \{\fC^*(y_1) = \cD_1, \, \fC^{*}(y_2) = \cD_2\})
	 	    \end{equation}
	 	    uniformly in realizations $\cD_1$ and $\cD_2$ such that $y_1, y_2 \in \Xi_j$ when $\fC^*(y_1) = \cD_i,\, i = 1, 2.$
	The bounds on the inner sum appearing in \eqref{eq:todiag} are similar but slightly different depending on whether $y_1 = y_2$ or $y_1 \neq y_2$. 
	
	In the case $y_1 \neq y_2$, we apply Proposition \ref{prop:forcalA} to bound the conditional probability in \eqref{eq:2ndfurther} by
	\begin{equation}
	    \label{eq:yoffdiag}
	\prob\left(\{y_1' \sa{B_{j+1}^1 \setminus (\cC \cup \cD_1)} z_1\} \circ \{y_2' \sa{B_{j+1}^1 \setminus (\cC \cup \cD_2)} z_2\} \right) \leq C \|y_1' - z_1\|^{2-d} \|y_2' - z_2\|^{2-d}\ .\end{equation}
    In case $y_1 = y_2$, we can instead upper bound the probability in \eqref{eq:2ndfurther} by
    	\begin{equation}
	    \label{eq:diag}
	    \begin{split}
	\prob\left(z_1, z_2 \in \fC(y_1; B_{j+1}^1 \setminus (\cC \cup \cD_1) \right) &\leq \prob\left(z_1, z_2 \in \fC(y_1) \right)\\
	&\leq \sum_{w \in \Zd} \prob\left(\left\{y_1 \lra w \right\} \circ \left\{z_1 \lra w \right\} \circ \left\{z_2 \lra w \right\}\right)\ .
	\end{split}
	\end{equation}
	
	Applying the upper bounds of \eqref{eq:yoffdiag} and \eqref{eq:diag} to \eqref{eq:2ndfurther}, we sum over $z_1, z_2$ in \eqref{eq:todiag} and then perform the outer sum over $y_1, y_2$. We arrive at the upper bound
	\begin{equation}
	\label{eq:Znsecond}
	\begin{split}
	\E[\,|Z_j|^2 \mid \clan] &\leq C (n/\kappa)^4 \E[X_j \mid \clan]^2\\
	&\quad + C (n/\kappa)^6 \E[X_j\, \mid \clan]\\
	&\leq \E[|Z_j|\, \mid \clan]^2\ .
	\end{split}
	\end{equation}
	Here the constant $C$ is uniform in $n$ and $\cC$ satisfying \eqref{eq:cCprop}; the final inequality of \eqref{eq:Znsecond} is furnished by \eqref{eq:pivsubst2} and \eqref{eq:toregX}. 
	\end{proof}

	\begin{proof}[Proof of Lemma \ref{lem:byann}]
	We use Proposition \ref{prop:Z2ndmom} in the Paley-Zygmund inequality. This yields $\prob(|Z_n| \geq (1/2)\E[|Z_n| \mid \clan] \mid \clan) \geq c$ for a uniform $c$, and then the uniform lower bound on $\E[Z_n \mid A(\cC)]$ from  Corollary \ref{cor:Zmo}  translates this into the statement of the lemma.
		\end{proof}
		
	We have now accomplished the goal of showing that $\fC_{B(n)}(0) \cap A_j$ is large, which we began working towards in Section \ref{sec:annlarge}. In the next section, we extend this result to many annuli at once and complete the proof of Theorem \ref{thm:volann}.

\subsection{The main argument} \label{sec:mainvol}
The main goal of the section is to complete the proof of Theorem \ref{thm:volann}, with Lemma \ref{lem:byann} as a main input. 
\begin{proof}[Proof of the upper bound from Theorem \ref{thm:volann}]
We recall the constant $c_0$ from \eqref{eq:cCprop0} and 
 the constant $c_v$ appearing in Lemma \ref{lem:byann}. For each $1 \leq j \leq \kappa$, we define the events
\begin{align*}
    R_j&=\{\fC(0; B_{j+1}^1) \cap A_j \ge c_v n^4/\kappa^4\},\\
    S_j&=\{X_j^* \geq c_0 n^2/\kappa^2\}.
\end{align*}
We will prove estimates on the probabilities of these events which are uniform in $n$ and $\kappa$ and which will suffice to establish the theorem.

Indeed, for each $\varphi > 0$, we have
\begin{align}
\prob(|\fC(0)| < \varphi c_v \lambda n^4 \mid 0 \lra \partial B(n)) &\leq
\prob(|\fC(0)| < \varphi c_v \kappa (n/\kappa)^4 \mid 0 \lra \partial B(n))\nonumber\\
&\leq \prob\left(\left| \left\{1 \leq j \leq\kappa: \, R_j \text{ occurs} \right\} \right| \leq \varphi \kappa \mid 0 \lra \partial B(n)\right)\ .\label{eq:Rjtoclust}
\end{align}
We will show 
\begin{equation}
\label{eq:Rjtoshow}
\text{there exist $c, \varphi > 0$ uniform in $n, \kappa$ such that }\eqref{eq:Rjtoclust} \leq c^{-1}(1-c)^\kappa\ ;
\end{equation}
because $\kappa = \lceil \lambda^{-1/3}\rceil$, the right side of \eqref{eq:Rjtoshow} is of the same form as the probability considered in Theorem \ref{thm:volann}. Thus, the theorem will be proved once \eqref{eq:Rjtoshow} has been established.

We define, for each $0 \leq j \leq \kappa-1,$
\[\mathfrak{Z}_j= \mathbf{1}_{\{0 \lra \partial B(n/2)\}} \prod_{k=1}^{j} \mathbf{1}_{\{0 \lra \partial B_{k+1}^1\}} (1+\mathbf{1}_{R^c_k}).\]
We first show an upper bound for the expectation of $\mathfrak{Z}_j$, depending on $\varphi$ and $j$ but not on $n$ or $\kappa$. To do this, we use successive conditioning.

Since $R_j$ is in the sigma-algebra generated by $\fC_{B_{j+1}^1}(0),$ we can write $\mathfrak{Z}_j = \mathfrak{Z}_j(\fC_{B_{j+1}^1}(0))$. To shorten notation, we define $\clan$ as in Section \ref{sec:newclust}, but with $j = \kappa-1$:
\[\clan = \{\fC(0; B_{\kappa-1}^1) = \cC\}\ . \]

Then, by conditioning, we see
\begin{align}
    \E[\mathfrak{Z}_{\kappa-1}] &= \sum_{\cC} \prob(\clan) \, \E[\mathfrak{Z}_{\kappa-1}\mid \clan]\nonumber\\
    &= \sum_{\cC} \prob(\clan) \, \mathfrak{Z}_{\kappa-2}(\cC)\, \E[(1 + \mathbf{1}_{R_{\kappa-1}^c}) \mathbf{1}_{\{0 \lra \partial B_{\kappa}^1\}} \mid \clan ].\label{eq:peelZ}
\end{align}
We estimate the conditional expectation in \eqref{eq:peelZ} differently depending on whether $\cC$ satisfies \eqref{eq:cCprop} or not. If $\E[X_{\kappa-1}^* \mid \clan] \geq c_0 (n/\kappa)^2,$ then invoking Lemma \ref{lem:byann}, we see
\begin{align}
   \E[(1 + \mathbf{1}_{R_{\kappa-1}^c}) \mathbf{1}_{\{0 \lra \partial B_{\kappa}^1\}} \mid \clan ] \leq 1 +  \prob(R_{\kappa-1}^c \mid \clan) \leq 2 - c\ ,\label{eq:Rk1}
\end{align}
where the constant $c > 0$ is uniform in $n$, $\kappa$.

On the other hand, if $\cC$ does not satisfy \eqref{eq:cCprop} --- that is, if 
\begin{equation}
    \E[X_\kappa^* \mid \clan] < c_0 (n/\kappa)^2\label{eqn: not36}
\end{equation}
--- then
\begin{align}
\E[(1 + \mathbf{1}_{R_{\kappa-1}^c}) \mathbf{1}_{\{0 \lra \partial B_{\kappa}^1\}} \mid \clan ] &\leq 2 \prob(0 \lra \partial B_{\kappa}^1 \mid \clan)\nonumber\\
&\leq 2 \prob(X_\kappa^* \geq 2 c_0 (n/\kappa)^2 \mid \clan)\\
&\ + 2 \prob(0 \lra \partial B_{\kappa}^1 \mid \clan \cap \{X_{j}^* \leq 2 c_0 (n/\kappa)^2\})\nonumber\\
&\leq 2(1/2 + 1/4) = 3/2 .\label{eq:Rk2}
\end{align}
Here the term $1/2$ comes from \eqref{eqn: not36} and Markov's inequality, and the term $1/4$ comes from \eqref{eq:cCprop0}.
Pulling together \eqref{eq:Rk1} and \eqref{eq:Rk2} and then performing the sum over $\cC$ in \eqref{eq:peelZ}, we see that there exists a $c > 0$ uniform in $n$ and $\kappa$ such that
\begin{equation}
    \label{eq:peelZ2}
    \E[\mathfrak{Z}_{\kappa-1}] \leq (2-c) \E[\mathfrak{Z}_{\kappa-2}]\ .
\end{equation}

We now apply the same argument on the expectation on the right-hand side of \eqref{eq:peelZ2} to show $ \E[\mathfrak{Z}_{\kappa-2}] \leq (2-c) \E[\mathfrak{Z}_{\kappa-3}]$. The constant $c$ here is the same as in \eqref{eq:peelZ2} because that constant $c$ originated in \eqref{eq:cCprop}, \eqref{eq:cCprop0}, and Lemma \ref{lem:byann} (and these gave bounds which were uniform in the choice of annulus $A_j$). Inducting and then at last taking the expectation over the $\mathbf{1}_{\{0 \lra \partial B(n/2)\}}$ in the definition of $\mathfrak{Z}_{\kappa-1}$, we find
\begin{equation}
    \label{eq:peelZ3}
   \text{there is an $\varphi > 0$ such that, uniformly in $n$, $\kappa$,}\quad \E[\mathfrak{Z}_{\kappa-1}] \leq \prob(0 \lra \partial B(n/2)) (2-2\varphi)^\kappa\ ,
\end{equation}
where we have renamed the constant to connect to the statements of \eqref{eq:Rjtoclust} and \eqref{eq:Rjtoshow}.

Indeed, choosing $\varphi$ as in \eqref{eq:peelZ3}, if $R_j^c$ occurs for more than $(1-\varphi) \kappa$  values of $j$, then we have $\mathfrak{Z}_{\kappa-1} \geq 2^{\kappa(1-\varphi)}.$ In particular, to show \eqref{eq:Rjtoshow}, we can write
\begin{align*}
    &\prob(|\{0 \leq j \leq \kappa-1: R_j^c \text{ occurs}\}| > (1-\varphi )\kappa, \, 0 \lra \partial B(n))\\
    \leq~ & 2^{-\kappa(1-\varphi) } \E[\mathfrak{Z}_{\kappa-1}]\\
    \text{(by \eqref{eq:peelZ3})} \qquad \leq~& 2^{-\kappa(1-\varphi)} 2^{\kappa} 2^{\kappa\log_2(1-\varphi)} \prob(0 \lra \partial B(n/2))\\
    \leq~ & 2^{-c \kappa} \prob(0 \lra \partial B(n/2))\ ,
\end{align*}
where as usual $c$ is uniform in $n$ and $\kappa$. Dividing the last display by $\prob(0 \lra \partial B(n))$ and using \eqref{eq:onearmprob} yields \eqref{eq:Rjtoshow}. As we noted just below \eqref{eq:Rjtoshow}, this completes the proof of Theorem~\ref{thm:volann}.
\end{proof}

\section{The number of spanning clusters} \label{sec: nospanning}
We denote by $\mathscr{S}_n$ the set of spanning clusters of $B(n)$:
\[\mathscr{S}_n:= \{\fC(x), x \in B(n): \, \exists y_1, y_2 \in \fC(x) \, \text{ such that } y_1(1) = -n, y_2(1) = n \}\ .\]
This quantity was analyzed in \cite{A97}, where it was shown that
\[\mathbb{P}(|\mathscr{S}_n|\ge o(1)n^{d-6})\rightarrow 1,\]
along with a matching upper bound \emph{provided} only clusters of size $\approx n^4$ are counted. Using Theorem \ref{thm:volann}, we remove the latter condition.

\begin{thm}
There is a $C > 0$ such that $\E[|\mathscr{S}_n|] \leq C n^{d-6}$. In particular, the sequence of random variables $(n^{6-d}|\mathscr{S}_n|)_{n=1}^\infty$ is tight.
\end{thm}
\begin{proof}
    We decompose based on the cardinality of spanning clusters; we then use Theorem \ref{thm:volann} to control the contribution of abnormally sparse spanning clusters. We define
    \[\mathscr{S}_{n,0}:= \{ \cC \in \mathscr{S}_n:\, |\cC| \geq n^4\} \cup \{ \cC \in \mathscr{S}_n:\, |\cC| \leq n^2\} \]
    and, for $1 \leq k \leq 2\log_2 n $, we set
    \[\mathscr{S}_{n,k}:= \{ \cC \in \mathscr{S}_n:\, 2^{-k} \leq |\cC|/n^4 < 2^{-k+1}\}\ . \]
    We then have $\E[|\mathscr{S}_n|] \leq \sum_{k=0}^{\lceil 2 \log_2 n \rceil} \E[|\mathscr{S}_{n,k}|]$, and it suffices to bound each term on the right-hand side of this inequality.
    
    For $k = 0$, we write (using Theorem \ref{thm:volann})
    \begin{align*}
        \E[|\mathscr{S}_{n,0}| ] &\leq \frac{1}{n^4} \sum_{x \in B(n)} \prob(x \lra \partial B(x;n),\,|\fC(x)| \geq n^4 ) + \sum_{x \in B(n)} \prob(x \lra \partial B(x;n), |\fC(x)| \leq n^2)\\ 
        &\leq \frac{1}{n^4} \sum_{x \in B(n)} \pi(n) + C n^d \pi(n) \exp(-c n^{2/3})  \leq C n^{d-6}\ .
    \end{align*}
    For $k \geq 1$, we bound similarly
    \begin{align*}
        \E[|\mathscr{S}_{n,k}| ] &\leq \frac{2^k}{n^4} \sum_{x \in B(n)} \prob(\fC(x) \in \mathscr{S}_{n,k})\\ 
        &\leq \frac{2^k}{n^4} \sum_{x \in B(n)} \pi(n) \prob(|\fC(x)| < 2^{-k+1} n^4 \mid x \lra B(x;n)) \leq C n^{d-6} 2^k \exp(-c 2^{k/3})\ ,
    \end{align*}
    where in the last inequality we again used Theorem \ref{thm:volann}. Summing these estimates over $k$ completes the proof.
\end{proof}

\section*{Acknowledgements}
The authors thank Akira Sakai for helpful discussions about the problem addressed in Theorem \ref{thm:scalingub}.
\bigskip
\addtocontents{toc}{\protect\setcounter{tocdepth}{1}}
\bibliographystyle{amsplain}
\bibliography{PercolationBib}

\providecommand{\bysame}{\leavevmode\hbox to3em{\hrulefill}\thinspace}
\providecommand{\MR}{\relax\ifhmode\unskip\space\fi MR }
\providecommand{\MRhref}[2]{%
  \href{http://www.ams.org/mathscinet-getitem?mr=#1}{#2}
}
\providecommand{\href}[2]{#2}
\begin{thebibliography}{10}

\bibitem{A97}
Michael Aizenman, \emph{On the number of incipient spanning clusters}, Nuclear
  Physics B \textbf{485} (1997), no.~3, 551--582.

\bibitem{AB87}
Michael Aizenman and David~J. Barsky, \emph{Sharpness of the phase transition
  in percolation models}, Communications in Mathematical Physics \textbf{108}
  (1987), no.~3, 489--526.

\bibitem{AB99}
Michael Aizenman and Almut Burchard, \emph{H{\"o}lder regularity and dimension
  bounds for random curves}, Duke Mathematical Journal \textbf{99} (1999),
  no.~3, 419--453.

\bibitem{AN84}
Michael Aizenman and Charles~M. Newman, \emph{Tree graph inequalities and
  critical behavior in percolation models}, Journal of Statistical Physics
  \textbf{36} (1984), no.~1--2, 107--143.

\bibitem{BA91}
David~J. Barsky and Michael Aizenman, \emph{Percolation critical exponents
  under the triangle condition}, The Annals of Probability (1991), 1520--1536.

\bibitem{BCF19}
G\'erard Ben~Arous, Manuel Cabezas, and Alexander Fribergh, \emph{Scaling limit
  for the ant in high-dimensional labyrinths}, Communications on Pure and
  Applied Mathematics \textbf{72} (2019), no.~4, 669--763.

\bibitem{BCKS99}
Christian Borgs, Jennifer~T. Chayes, Harry Kesten, and Joel Spencer,
  \emph{Uniform boundedness of critical crossing probabilities implies
  hyperscaling}, Random Structures \& Algorithms \textbf{15} (1999), no.~3-4,
  368--413.

\bibitem{CH20}
Shirshendu Chatterjee and Jack Hanson, \emph{Restricted percolation critical
  exponents in high dimensions}, Communications on Pure and Applied Mathematics
  \textbf{73} (2020), no.~11, 2370--2429.

\bibitem{DHS21}
Michael Damron, Jack Hanson, and Philippe Sosoe, \emph{Strict inequality for
  the chemical distance exponent in two-dimensional critical percolation},
  Communications on Pure and Applied Mathematics \textbf{74} (2021), no.~4,
  679--743.

\bibitem{duminil2016new}
Hugo Duminil-Copin and Vincent Tassion, \emph{A new proof of the sharpness of
  the phase transition for {B}ernoulli percolation and the {I}sing model},
  Communications in Mathematical Physics \textbf{343} (2016), no.~2, 725--745.

\bibitem{FH17}
Robert Fitzner and Remco van~der Hofstad, \emph{Generalized approach to the
  non-backtracking lace expansion}, Probability Theory and Related Fields
  \textbf{169} (2017), no.~3-4, 1041--1119.

\bibitem{garban2013pivotal}
Christophe Garban, G{\'a}bor Pete, and Oded Schramm, \emph{Pivotal, cluster,
  and interface measures for critical planar percolation}, Journal of the
  American Mathematical Society \textbf{26} (2013), no.~4, 939--1024.

\bibitem{G99}
Geoffrey Grimmett, \emph{Percolation}, Springer, 1999.

\bibitem{H57}
J.M. Hammersley, \emph{Percolation processes: Lower bounds for the critical
  probability}, Annals of Mathematical Statistics \textbf{28} (1957), no.~3,
  790--795.

\bibitem{H90}
Takashi Hara, \emph{Mean-field critical behaviour for correlation length for
  percolation in high dimensions}, Probability Theory and Related Fields
  \textbf{86} (1990), no.~3, 337--385.

\bibitem{H08}
\bysame, \emph{Decay of correlations in nearest-neighbor self-avoiding walk,
  percolation, lattice trees and animals}, The Annals of Probability
  \textbf{36} (2008), no.~2, 530--593.

\bibitem{HS90}
Takashi Hara and Gordon Slade, \emph{Mean-field critical behaviour for
  percolation in high dimensions}, Communications in Mathematical Physics
  \textbf{128} (1990), no.~2, 333--391.

\bibitem{HHS03}
Takashi Hara, Remco van~der Hofstad, and Gordon Slade, \emph{Critical two-point
  functions and the lace expansion for spread-out high-dimensional percolation
  and related models}, The Annals of Probability \textbf{31} (2003), no.~1,
  349--408.

\bibitem{HH07}
Markus Heydenreich and Remco van~der Hofstad, \emph{Random graph asymptotics on
  high-dimensional tori}, Communications in Mathematical Physics \textbf{270}
  (2007), 335--358.

\bibitem{HH11}
\bysame, \emph{Random graph asymptotics on high-dimensional tori ii: volume,
  diameter and mixing time}, Probability Theory and Related Fields \textbf{149}
  (2011), 397--415.

\bibitem{HH17}
\bysame, \emph{Progress in high-dimensional percolation and random graphs},
  Springer, 2017.

\bibitem{HVH14a}
Markus Heydenreich, Remco van~der Hofstad, and Tim Hulshof,
  \emph{High-dimensional incipient infinite clusters revisited}, Journal of
  Statistical Physics \textbf{155} (2014), no.~5, 966--1025.

\bibitem{HMS}
Tom Hutchcroft, Emmanuel Michta, and Gordon Slade, \emph{High-dimensional
  near-critical percolation and the torus plateau}, arXiv preprint
  arXiv:2107.12971 (2021).

\bibitem{K87}
Harry Kesten, \emph{A scaling relation at criticality for 2d-percolation},
  Percolation theory and ergodic theory of infinite particle systems, Springer,
  1987, pp.~203--212.

\bibitem{K87a}
\bysame, \emph{Scaling relations for 2d-percolation}, Communications in
  Mathematical Physics \textbf{109} (1987), no.~1, 109--156.

\bibitem{KZ87}
Harry Kesten and Yu~Zhang, \emph{Strict inequalities for some critical
  exponents in two-dimensional percolation}, Journal of Statistical Physics
  \textbf{46} (1987), no.~5-6, 1031--1055.

\bibitem{KZ93}
\bysame, \emph{The tortuosity of occupied crossings of a box in critical
  percolation}, Journal of Statistical Physics \textbf{70} (1993), no.~3-4,
  599--611.

\bibitem{K14}
Demeter Kiss, \emph{Large deviation bounds for the volume of the largest
  cluster in 2d critical percolation}, Electronic Communications in Probability
  \textbf{19} (2014), 1--11.

\bibitem{KN09}
Gady Kozma and Asaf Nachmias, \emph{The {A}lexander-{O}rbach conjecture holds
  in high dimensions}, Inventiones Mathematicae \textbf{178} (2009), no.~3,
  635.

\bibitem{KN11}
\bysame, \emph{Arm exponents in high dimensional percolation}, Journal of the
  American Mathematical Society \textbf{24} (2011), no.~2, 375--409.

\bibitem{LSW02}
Gregory Lawler, Oded Schramm, Wendelin Werner, et~al., \emph{One-arm exponent
  for critical 2d percolation}, Electronic Journal of Probability \textbf{7}
  (2002), 1--13.

\bibitem{LSW01}
Gregory~F Lawler, Oded Schramm, and Wendelin Werner, \emph{Values of {B}rownian
  intersection exponents, {I}: Half-plane exponents}, Acta Mathematica
  \textbf{187} (2001), no.~2, 237--273.

\bibitem{LSW01a}
\bysame, \emph{Values of {B}rownian intersection exponents, {II}: Plane
  exponents}, Acta Mathematica \textbf{187} (2001), no.~2, 275--308.

\bibitem{LP17}
Russell Lyons and Yuval Peres, \emph{Probability on trees and networks},
  vol.~42, Cambridge University Press, 2017.

\bibitem{morrow2005sizes}
Gregory~J. Morrow and Yu~Zhang, \emph{The sizes of the pioneering, lowest
  crossing and pivotal sites in critical percolation on the triangular
  lattice}, The Annals of Applied Probability \textbf{15} (2005), no.~3,
  1832--1886.

\bibitem{SR21}
Lily Reeves and Philippe Sosoe, \emph{An estimate for the radial chemical
  distance in $2d$ critical percolation clusters}, arXiv:2001.07872 (2020),
  1--27.

\bibitem{S04}
Akira Sakai, \emph{Mean-field behavior for the survival probability and the
  percolation point-to-surface connectivity}, Journal of Statistical Physics
  \textbf{117} (2004), no.~1-2, 111--130.

\bibitem{Sc00}
Oded Schramm, \emph{Scaling limits of loop-erased random walks and uniform
  spanning trees}, Israel Journal of Mathematics \textbf{118} (2000), 221--288.

\bibitem{schramm2011conformally}
\bysame, \emph{Conformally invariant scaling limits: an overview and a
  collection of problems}, Selected Works of Oded Schramm (2011), 1161--1191.

\bibitem{S01}
S.~Smirnov, \emph{Critical percolation in the plane: Conformal invariance,
  {C}ardy's formula, scaling limits}, Comptes Rendus de l'Acad{\'e}mie des
  Sciences-Series I-Mathematics \textbf{333} (2001), no.~3, 239--244.

\bibitem{B15}
Wouter~Cames van Batenburg, \emph{The dimension of the incipient infinite
  cluster}, Electronic Communications in Probability \textbf{20} (2015), 1--10.

\bibitem{vandenbergconijn}
Jacob van~den Berg and Rene Conijn, \emph{On the size of the largest cluster in
  2d critical percolation}, Electronic Communications in Probability
  \textbf{17} (2012), (none).

\bibitem{HJ04}
Remco van~der Hofstad and Antal~A. J\'arai, \emph{The incipient infinite
  cluster for high-dimensional unoriented percolation}, Journal of statistical
  physics \textbf{114} (2004), no.~3--4, 625--663.

\bibitem{HS14}
Remco van~der Hofstad and Art{\"e}m Sapozhnikov, \emph{Cycle structure of
  percolation on high-dimensional tori}, Annales de l'Institut Henri
  Poincar{\'e}, Probabilit{\'e}s et Statistiques, vol.~50, Institut Henri
  Poincar{\'e}, 2014, pp.~999--1027.

\bibitem{W09}
Wendelin Werner, \emph{Lectures on two-dimensional critical percolation},
  IAS-Park City Mathematical Sciences \textbf{16} (2009), no.~Statistical
  Mechanics, 297--360.

\end{thebibliography}

\end{document}